\documentclass[10pt,draft]{article}
\usepackage{graphicx}
\usepackage{amssymb}
\usepackage{color,graphicx}
\usepackage{amsmath,amsthm,amscd}
\usepackage[latin1]{inputenc}
 \numberwithin{equation}{section}
\newtheorem{theorem}{Theorem}

\newtheorem{proposition}[theorem]{Proposition}
\newtheorem{lemma}[theorem]{Lemma}

\newtheorem{corollary}[theorem]{Corollary}

\newtheorem{definition}[theorem]{Definition}
\theoremstyle{definition}
\newtheorem{example}[theorem]{Example}

\newtheorem{remark}[theorem]{Remark}

\newcommand{\cG}{\mbox{${\cal G}$}}

\newcommand{\cO}{\mbox{${\cal O}$}}

\newcommand{\cU}{\mbox{${\cal U}$}}

\newcommand{\cW}{\mbox{${\cal W}$}}

\title{\textbf{Projective modules and Gröbner
bases for skew PBW extensions}}
\author{Oswaldo Lezama \& Claudia Gallego
\\
Seminario de Álgebra Constructiva - $\text{SAC}^2$\\
Departamento de Matemáticas\\
Universidad Nacional de Colombia, Bogotá, COLOMBIA\\
}
\date{}
\begin{document}
\maketitle
\begin{abstract}\noindent
Many rings and algebras arising in quantum mechanics, algebraic analysis, and non-commutative
algebraic geometry can be interpreted as skew $PBW$ (Poincaré-Birkhoff-Witt) extensions. In the
present paper we study two aspects of these non-commutative rings: its finitely generated
projective modules from a matrix-constructive approach, and the construction of the Gröbner theory
for its left ideals and modules. These two topics could be interesting in future eventual
applications of skew $PBW$ extensions in functional linear systems and in non-commutative algebraic
geometry.

\bigskip

\noindent \textit{Key words and phrases.} Skew $PBW$ extensions, noncommutative Gröbner bases,
projective modules, matrix-constructive methods, Buchberger's algorithm, stably free modules,
Hermite rings, stable rank.

\bigskip

\bigskip

\noindent 2010 \textit{Mathematics Subject Classification.} Primary: 16Z05. Secondary: 16D40,
15A21.
\end{abstract}

\newpage

\section{Introduction}

Many rings and algebras arising in quantum mechanics, algebraic analysis, and non-commutative
algebraic geometry can be interpreted as skew $PBW$ (Poincaré-Birkhoff-Witt) extensions. Indeed,
Weyl algebras, enveloping algebras of finite-dimensional Lie algebras (and its quantization), well
known classes of Ore algebras (for example, the algebra of shift operators and the algebra for
multidimensional discrete linear systems), Artamonov quantum polynomials, diffusion algebras, Manin
algebra of quantum matrices, Witten's deformation of $\mathcal{U}(\mathfrak{sl}(2,K)$, among many
others, are examples of skew $PBW$ extensions. This type of non-commutative rings were defined
firstly in \cite{Gallego2} and represent a generalization of $PBW$ extensions introduced by Bell
and Goodearl in \cite{Bell}. Some other authors have classified quantum algebras and other
non-commutative rings of polynomial type by similar notions: Levandovskyy in \cite{Levandovskyy}
defined the $G$-algebras, Bueso, Gómez-Torrecillas and Verschoren in \cite{Gomez-Torrecillas2}
introduced the $PBW$ rings, Panov in \cite{Panov} defined the so called $Q$-solvable algebras. In
all of cases they assume either that the ring of coefficients is a field or the variables commute
with the coefficients. As we will see below, for the skew $PBW$ extensions the ring of coefficients
is arbitrary and the variables non necessarily commute. Ring and module theoretical properties of
skew $PBW$ extensions have been studied in some recent papers (\cite{lezamareyes1},
\cite{lezamaore}, \cite{Reyes5}), in the present paper we are interested in two aspects of these
non-commutative rings: the study of finitely generated projective modules from a
matrix-constructive approach, and the construction of the Gröbner theory for left ideals and
modules. These two topics could be interesting in future eventual applications of skew $PBW$
extensions in functional linear systems (as it has been done for Ore algebras in \cite{Quadrat1},
\cite{Chyzak1}, \cite{Chyzak}, \cite{Chyzak3}, \cite{Chyzak2}, \cite{Quadrat2}, \cite{Quadrat7},
\cite{Fabianska}, \cite{Pommaret1}, \cite{Pommaret2}, \cite{Pommaret3}, \cite{Pommaret4},
\cite{Pommaret}, \cite{Quadrat}, \cite{Quadrat6}, \cite{Zerz} and \cite{Zhang}), and in
non-commutative algebraic geometry (see in \cite{Rogalski} Section 1.4 about non-commutative
Gröbner bases for some quantum algebras).

\section{Skew $PBW$ extensions}

In this first section we recall the definition of skew $PWB$ extensions, some their elementary
properties and we present some examples of this class of non-commutative rings of polynomial type
(see \cite{Gallego2} and \cite{lezamareyes1}).

\subsection{Definitions and elementary examples}\label{1.1}
We will see next that the skew $PBW$ extensions are a generalization of PBW extensions defined by
Bell and Goodearl in 1988 in \cite{Bell}.
\begin{definition}\label{gpbwextension}
Let $R$ and $A$ be rings, we say that $A$ is a skew $PBW$ extension of $R$ {\rm(}also called
$\sigma-PBW$ extension{\rm)}, if the following conditions hold:
\begin{enumerate}
\item[\rm (i)]$R\subseteq A$.
\item[\rm (ii)]There exist finite elements $x_1,\dots ,x_n\in A$ such $A$ is a left $R$-free module with basis
\begin{center}
$Mon(A):=Mon\{x_1,\dots,x_n\}=\{x^{\alpha}=x_1^{\alpha_1}\cdots
x_n^{\alpha_n}|\alpha=(\alpha_1,\dots ,\alpha_n)\in \mathbb{N}^n\}$.
\end{center}
\item[\rm (iii)]For every $1\leq i\leq n$ and $r\in R-\{0\}$ there exists $c_{i,r}\in R-\{0\}$ such that
\begin{equation}\label{sigmadefinicion1}
x_ir-c_{i,r}x_i\in R.
\end{equation}
\item[\rm (iv)]For every $1\leq i,j\leq n$ there exists $c_{i,j}\in R-\{0\}$ such that
\begin{equation}\label{sigmadefinicion2}
x_jx_i-c_{i,j}x_ix_j\in R+Rx_1+\cdots +Rx_n.
\end{equation}
Under these conditions we will write $A=\sigma(R)\langle x_1,\dots ,x_n\rangle$.
\end{enumerate}
\end{definition}
\begin{remark}\label{notesondefsigampbw}
(i) Since that $Mon(A)$ is a $R$-basis for $A$, the elements $c_{i,r}$ and $c_{i,j}$ in the above
definition are unique.

(ii) If $r=0$, then $c_{i,0}=0$: in fact, $0=x_i0=c_{i,0}x_i+s_i$, with $s_i\in R$, but since
$Mon(A)$ is a $R$-basis, then $c_{i,0}=0=s_i$.

(iii) In (iv), $c_{i,i}=1$: in fact, $x_i^2-c_{i,i}x_i^2=s_0+s_1x_1+\cdots+s_nx_n$, with $s_i\in
R$, hence $1-c_{i,i}=0=s_i$.

(iv) Let $i<j$, by (\ref{sigmadefinicion2}) there exist $c_{j,i},c_{i,j}\in R$ such that
$x_ix_j-c_{j,i}x_jx_i\in R+Rx_1+\cdots +Rx_n$ and $x_jx_i-c_{i,j}x_ix_j\in R+Rx_1+\cdots +Rx_n$,
but since $Mon(A)$ is a $R$-basis then $1=c_{j,i}c_{i,j}$, i.e., for every $1\leq i<j\leq n$,
$c_{i,j}$ has a left inverse and $c_{j,i}$ has a right inverse.

(v) Each element $f\in A-\{0\}$ has a unique representation in the form $f=c_1X_1+\cdots+c_tX_t$,
with $c_i\in R-\{0\}$ and $X_i\in Mon(A)$, $1\leq i\leq t$.
\end{remark}
The following proposition justifies the notation that we have introduced for the skew $PBW$
extensions.
\begin{proposition}\label{sigmadefinition}
Let $A$ be a skew $PBW$ extension of $R$. Then, for every $1\leq i\leq n$, there exist an injective
ring endomorphism $\sigma_i:R\rightarrow R$ and a $\sigma_i$-derivation $\delta_i:R\rightarrow R$
such that
\begin{center}
$x_ir=\sigma_i(r)x_i+\delta_i(r)$,
\end{center}
for each $r\in R$.
\end{proposition}
\begin{proof} See \cite{Gallego3}
\end{proof}
A particular case of skew $PBW$ extension is when all derivations $\delta_i$ are zero. Another
interesting case is when all $\sigma_i$ are bijective and the constants $c_{ij}$ are invertible. We
have the following definition.
\begin{definition}\label{sigmapbwderivationtype}
Let $A$ be a skew $PBW$ extension.
\begin{enumerate}
\item[\rm (a)]
$A$ is quasi-commutative if the conditions {\rm(}iii{\rm)} and {\rm(}iv{\rm)} in Definition
\ref{gpbwextension} are replaced by
\begin{enumerate}
\item[\rm ($iii'$)]For every $1\leq i\leq n$ and $r\in R-\{0\}$ there exists $c_{i,r}\in R-\{0\}$ such that
\begin{equation}
x_ir=c_{i,r}x_i.
\end{equation}
\item[\rm ($iv'$)]For every $1\leq i,j\leq n$ there exists $c_{i,j}\in R-\{0\}$ such that
\begin{equation}
x_jx_i=c_{i,j}x_ix_j.
\end{equation}
\end{enumerate}
\item[\rm (b)]$A$ is bijective if $\sigma_i$ is bijective for
every $1\leq i\leq n$ and $c_{i,j}$ is invertible for any $1\leq i<j\leq n$.
\end{enumerate}
\end{definition}
Some elementary but interesting examples of skew $PBW$ extensions are the following.
\begin{example}\label{gpbwexample}
(i) Any $PBW$ extension is a bijective skew $PBW$ extension since in this case $\sigma_{i}=i_{R}$
for each $1\leq i\leq n$ and $c_{i,j}=1$ for every $1\leq i,j\leq n$ (see \cite{Bell}).

(ii) Any \textit{skew polynomial ring $R[x;\sigma ,\delta]$ of injective type}, i.e., with $\sigma$
injective, is a skew $PBW$ extension; in this case we have $R[x;\sigma ,\delta]=\sigma(R)\langle
x\rangle$. If additionally $\delta=0$, then $R[x;\sigma]$ is quasi-commutative.

(iii) Let $R[x_1;\sigma_1 ,\delta_1]\cdots [x_n;\sigma_n ,\delta_n]$ be an \textit{iterated skew
polynomial ring of injective type}, i.e., if the following conditions hold:
\begin{center}
For $1\leq i\leq n$, $\sigma_i$ is injective

For every $r\in R$ and $1\leq i\leq n$, $\sigma_i(r),\delta_i(r)\in R$

For $i<j$, $\sigma_j(x_i)=cx_i+d$, with $c,d\in R$ and $c$ has a left inverse.

For $i<j$, $\delta_j(x_i)\in R+Rx_1+\cdots +Rx_i$.
\end{center}
Then, $R[x_1;\sigma_1 ,\delta_1]\cdots [x_n;\sigma_n ,\delta_n]$ is a skew $PBW$ extension. Under
these conditions we have
\begin{center}
$R[x_1;\sigma_1 ,\delta_1]\cdots [x_n;\sigma_n ,\delta_n]=\sigma(R)\langle x_1,\dots,x_n\rangle$.
\end{center}
In particular, any \textit{Ore extension $R[x_1;\sigma_1 ,\delta_1]\cdots [x_n;\sigma_n ,\delta_n]$
of injective type}, i.e., for $1\leq i\leq n$, $\sigma_i$ is injective, is a skew $PBW$ extension.
In fact, in Ore extensions for every $r\in R$ and $1\leq i\leq n$, $\sigma_i(r),\delta_i(r)\in R$,
and for $i<j$, $\sigma_j(x_i)=x_i$ and $\delta_j(x_i)=0$. An important subclass of Ore extension of
injective type are the \textit{Ore algebras of injective type}, i.e., when $R=K[t_1,\dots,t_m]$,
$m\geq 0$. Thus, we have
\begin{center}
$K[t_1,\dots,t_m][x_1;\sigma_1 ,\delta_1]\cdots [x_n;\sigma_n
,\delta_n]=\sigma(K[t_1,\dots,t_m])\langle x_1,\dots,x_n\rangle$.
\end{center}
Some concrete examples of Ore algebras of injective type are the following.

\textit{The algebra of shift operators}: let $K$ be a field and $h\in K$, then the algebra of shift
operators is defined by $S_h:=K[t][x_h;\sigma_h,\delta_h]$, where $\sigma_h(p(t)):=p(t-h)$, and
$\delta_h:=0$ (observe that $S_h$ can be considered also as a skew polynomial ring of injective
type). Thus, $S_h$ is a quasi-commutative bijective skew $PBW$ extension.

The \textit{mixed algebra $D_h$}: let again $K$ be a field and $h\in K$, then the mixed algebra
$D_h$ is defined by $D_h:=K[t][x;i_{K[t]},\frac{d}{dt}][x_h;\sigma_h,\delta_h]$, where
$\sigma_h(x):=x$. Then, $D_h$ is a quasi-commutative bijective skew $PBW$ extension.

The \textit{algebra for multidimensional discrete linear systems} is defined by
$D:=K[t_1,\dots,t_n][x_1;\sigma_1,0]\cdots[x_n;\sigma_n,0]$, where $K$ is a field and
\begin{center}
$\sigma_i(p(t_1,\dots,t_n)):=p(t_1,\dots,t_{i-1},t_{i}+1,t_{i+1},\dots,t_n), \ \sigma_i(x_i)=x_i$,
$1\leq i\leq n$.
\end{center}
Thus, $D$ is a quasi-commutative bijective skew $PBW$ extension. Observe that all of these examples
are not $PBW$ extensions.

(iv) \textit{Additive analogue of the Weyl algebra}: let $K$ be a field, the $K$-algebra
$A_n(q_1,\dots,q_n)$ is generated by $x_1,\dots,x_n,y_1,\dots,y_n$ and subject to the relations:
\begin{center}
$x_jx_i = x_ix_j, y_jy_i = y_iy_j, \ 1 \leq i,j \leq n$,

$y_ix_j=x_jy_i, \ i\neq j$,

$y_ix_i = q_ix_iy_i + 1, \ 1\leq i\leq n$,
\end{center}
where $q_i\in K-\{0\}$. We observe that $A_n(q_1, \dots, q_n)$ is isomorphic to the iterated skew
polynomial ring $K[x_1, \dots, x_n][y_1; \sigma_1, \delta_1]\cdots [y_n; \sigma_n, \delta_n]$ over
the commutative polynomial ring $K[x_1, \dots, x_n]$:
\begin{center}
$\sigma_j(y_i):=y_i,\delta_j(y_i):=0,  \ 1\leq i < j\leq n$,

$\sigma_i(x_j):=x_j,\delta_i(x_j):=0, \ i\neq j$,

$\sigma_i(x_i):=q_ix_i,\delta_i(x_i):=1, \ 1\leq i\leq n$.

\end{center}
Thus, $A_n(q_1, \dots, q_n)$ satisfies the conditions of (iii) and is bijective; we have
\begin{center}
$A_n(q_1, \dots, q_n)=\sigma(K[x_1,\dots,x_n])\langle y_1,\dots,y_n\rangle$.
\end{center}

(v) \textit{Multiplicative analogue of the Weyl algebra}: let $K$ be a field, the $K$-algebra
$\mathcal{O}_n(\lambda_{ji})$ is generated by $x_1,\dots,x_n$ and subject to the relations:
\begin{center}
$x_jx_i =\lambda_{ji}x_ix_j ,\ 1\leq i<j\leq n$,
\end{center}
where $\lambda_{ji}\in K-\{0\}$. We note that $\mathcal{O}_n(\lambda_{ji})$ is isomorphic to the
iterated skew polynomial ring $K[x_1][x_2;\sigma_2]\cdots [x_n;\sigma_n]$
\begin{center}
$\sigma_j(x_i):=\lambda_{ji}x_i, \ 1\leq i<j\leq n$.
\end{center}
Thus, $\mathcal{O}_n(\lambda_{ji})$ satisfies the conditions of (iii), and hence
$\mathcal{O}_n(\lambda_{ji})$ is an iterated skew polynomial ring of injective type but it is not
Ore. Thus,
\begin{center}
$\mathcal{O}_n(\lambda_{ji})=\sigma(K[x_1])\langle x_2,\dots,x_n\rangle$.
\end{center}
Moreover, note that $\mathcal{O}_n(\lambda_{ji})$ is quasi-commutative and bijective.

(vi) \textit{$q$-Heisenberg algebra}: let $K$ be a field, the $K$-algebra $H_n(q)$ is generated by
$x_1,\dots,x_n,y_1,\dots,y_n,z_1,\dots,z_n$ and subject to the relations:
\begin{center}
$x_jx_i = x_ix_j, z_jz_i=z_iz_j, y_jy_i = y_iy_j, \ 1 \leq i,j \leq n$,

$z_jy_i=y_iz_j,z_jx_i=x_iz_j,y_jx_i=x_iy_j, \ i\neq j$,

$z_iy_i = qy_iz_i,z_ix_i=q^{-1}x_iz_i+y_i,y_ix_i = qx_iy_i, \ 1\leq i\leq n$,
\end{center}
with $q\in K-\{0\}$. Note that $H_n(q)$ is isomorphic to the iterated skew polynomial ring $K[x_1,
\dots, x_n][y_1; \sigma_1]\cdots [y_n; \sigma_n][z_1; \theta_1, \delta_1]\cdots [z_n; \theta_n,
\delta_n]$ on the commutative polynomial ring $K[x_1, \dots, x_n]$:
\begin{center}
$\theta_j(z_i):=z_i, \ \delta_j(z_i):=0, \sigma_j(y_i):=y_i, \ 1\leq i< j\leq n$,

$\theta_j(y_i):=y_i, \ \delta_j(y_i):=0,\theta_j(x_i):=x_i, \ \delta_j(x_i):=0,\sigma_j(x_i):=x_i,
\ i\neq j$,

$\theta_i(y_i):=qy_i, \ \delta_i(y_i):=0,\theta_i(x_i):=q^{-1}x_i, \
\delta_i(x_i):=y_i,\sigma_i(x_i):=qx_i, \ 1\leq i\leq n$,
\end{center}
Since $\delta_i(x_i)=y_i\notin K[x_1,\dots,x_n]$, then $H_n(q)$ is not a  skew $PBW$ extension of
$K[x_1,\dots,x_n]$, however, with respect to $K$, $H_n(q)$ satisfies the conditions of (iii), and
hence, $H_n(q)$ is a bijective skew $PBW$ extension of $K$:
\begin{center}
$H_n(q)=\sigma(K)\langle x_1,\dots,x_n;y_1,\dots,y_n;z_1,\dots,z_n\rangle$.
\end{center}
\end{example}

\begin{remark}We want to remark that the skew $PBW$ extensions are not a subclass of the collection of
iterated skew polynomial rings, take for example $\mathcal{U}(\mathcal{G})$ or the diffusion
algebra (see \cite{lezamareyes1} and Section \ref{1.3} below). On other hand, the skew polynomial
rings are not included in the class of skew $PBW$ extensions, take $R[x;\sigma,\delta]$, with
$\sigma$ not injective.
\end{remark}

\subsection{Basic properties}

Next we present some basic important properties of skew $PBW$ extensions. We start with some
notation that we will use frequently in the rest of this work.

\begin{definition}\label{1.1.6}
Let $A$ be a skew $PBW$ extension of $R$ with endomorphisms $\sigma_i$, $1\leq i\leq n$, as in
Proposition \ref{sigmadefinition}.
\begin{enumerate}
\item[\rm (i)]For $\alpha=(\alpha_1,\dots,\alpha_n)\in \mathbb{N}^n$,
$\sigma^{\alpha}:=\sigma_1^{\alpha_1}\cdots \sigma_n^{\alpha_n}$,
$|\alpha|:=\alpha_1+\cdots+\alpha_n$. If $\beta=(\beta_1,\dots,\beta_n)\in \mathbb{N}^n$, then
$\alpha+\beta:=(\alpha_1+\beta_1,\dots,\alpha_n+\beta_n)$.
\item[\rm (ii)]For $X=x^{\alpha}\in Mon(A)$,
$\exp(X):=\alpha$ and $\deg(X):=|\alpha|$.
\item[\rm (iii)]Let $0\neq f\in A$, $t(f)$ is the finite
set of terms that conform $f$, i.e., if $f=c_1X_1+\cdots +c_tX_t$, with $X_i\in Mon(A)$ and $c_i\in
R-\{0\}$, then $t(f):=\{c_1X_1,\dots,c_tX_t\}$.
\item[\rm (iv)]Let $f$ be as in {\rm(iii)}, then $\deg(f):=\max\{\deg(X_i)\}_{i=1}^t.$
\end{enumerate}
\end{definition}

The skew $PBW$ extensions can be characterized in a similar way as was done in
\cite{Gomez-Torrecillas} for $PBW$ rings.
\begin{theorem}\label{coefficientes}
Let $A$ be a left polynomial ring over $R$ w.r.t. $\{x_1,\dots,x_n\}$, i.e. the conditions (i) and
(ii) in Definition \ref{gpbwextension} are satisfied. $A$ is a skew $PBW$ extension of $R$ if and
only if the following conditions hold:
\begin{enumerate}
\item[\rm (a)]For every $x^{\alpha}\in Mon(A)$ and every $0\neq
r\in R$ there exist unique elements $r_{\alpha}:=\sigma^{\alpha}(r)\in R-\{0\}$ and $p_{\alpha
,r}\in A$ such that
\begin{equation}\label{611}
x^{\alpha}r=r_{\alpha}x^{\alpha}+p_{\alpha , r},
\end{equation}
where $p_{\alpha ,r}=0$ or $\deg(p_{\alpha ,r})<|\alpha|$ if $p_{\alpha , r}\neq 0$. Moreover, if
$r$ is left invertible, then $r_\alpha$ is left invertible.

\item[\rm (b)]For every $x^{\alpha},x^{\beta}\in Mon(A)$ there
exist unique elements $c_{\alpha,\beta}\in R$ and $p_{\alpha,\beta}\in A$ such that
\begin{equation}\label{612}
x^{\alpha}x^{\beta}=c_{\alpha,\beta}x^{\alpha+\beta}+p_{\alpha,\beta},
\end{equation}
where $c_{\alpha,\beta}$ is left invertible, $p_{\alpha,\beta}=0$ or
$\deg(p_{\alpha,\beta})<|\alpha+\beta|$ if $p_{\alpha,\beta}\neq 0$.
\end{enumerate}
\end{theorem}
\begin{proof} See \cite{Gallego3}
\end{proof}
\begin{remark}\label{identities}
(i) A left inverse of $c_{\alpha,\beta}$ will be denoted by $c_{\alpha,\beta}'$. We observe that if
$\alpha=0$ or $\beta=0$, then $c_{\alpha,\beta}=1$ and hence $c_{\alpha,\beta}'=1$.

(ii) Let $\theta,\gamma,\beta\in \mathbb{N}^n$ and $c\in R$, then we have the following identities:
\begin{center}
$\sigma^\theta(c_{\gamma,\beta})c_{\theta,\gamma+\beta}=c_{\theta,\gamma}c_{\theta+\gamma,\beta}$,

$\sigma^\theta(\sigma^\gamma (c))c_{\theta,\gamma}=c_{\theta,\gamma}\sigma^{\theta+\gamma}(c)$.
\end{center}
In fact, since $x^{\theta}(x^{\gamma}x^{\beta})=(x^{\theta}x^{\gamma})x^{\beta}$, then
\begin{center}
$x^{\theta}(c_{\gamma,\beta}x^{\gamma+\beta}+p_{\gamma,\beta})=(c_{\theta,\gamma}x^{\theta+\gamma}+p_{\theta,\gamma})x^{\beta}$,

$\sigma^{\theta}(c_{\gamma,\beta})c_{\theta,\gamma+\beta}x^{\theta+\gamma+\beta}+p=c_{\theta,\gamma}c_{\theta+\gamma,\beta}x^{\theta+\gamma+\beta}+q$,
\end{center}
with $p=0$ or $\deg(p)<|\theta+\gamma+\beta|$, and, $q=0$ or $\deg(q)<|\theta+\gamma+\beta|$. From
this we get the first identity. For the second, $x^{\theta}(x^{\gamma}c)=(x^{\theta}x^{\gamma})c$,
and hence
\begin{center}
$x^{\theta}(\sigma^{\gamma}(c)x^{\gamma}+p_{\gamma,c})=(c_{\theta,\gamma}x^{\theta+\gamma}+p_{\theta,\gamma})c$,

$\sigma^{\theta}(\sigma^{\gamma}(c))c_{\theta,\gamma}x^{\theta+\gamma}+p=c_{\theta,\gamma}\sigma^{\theta+\gamma}(c)x^{\theta+\gamma}+q$,
\end{center}
with $p=0$ or $\deg(p)<|\theta+\gamma|$, and, $q=0$ or $\deg(q)<|\theta+\gamma|$. This proves the
second idenity.

(iii) We observe if $A$ is quasi-commutative, then from the proof of Theorem \ref{coefficientes}
(see \cite{Gallego3}) we conclude that $p_{\alpha,r}=0$ and $p_{\alpha,\beta}=0$ for every $0\neq
r\in R$ and every $\alpha,\beta \in \mathbb{N}^n$. On the other hand, note that the evaluation
function at $0$, i.e., $A\to R$, $f\in A\mapsto f(0)\in R$, is a ring surjective homomorphism with
kernel $\langle x_1,\dots,x_n \rangle$ the two-sided ideal generated by $x_1,\dots,x_n$. Thus,
$A/\langle x_1,\dots,x_n \rangle\cong R$.

(iv) If $A$ is bijective, then $c_{\alpha,\beta}$ is invertible for any $\alpha,\beta\in
\mathbb{N}^n$.

(v) In $Mon(A)$ we define
\begin{center}
$x^{\alpha}\succeq x^{\beta}\Longleftrightarrow
\begin{cases}
x^{\alpha}=x^{\beta}\\
\text{or} & \\
x^{\alpha}\neq x^{\beta}\, \text{but} \, |\alpha|> |\beta| & \\
\text{or} & \\
x^{\alpha}\neq x^{\beta},|\alpha|=|\beta|\, \text{but $\exists$ $i$ with} &
\alpha_1=\beta_1,\dots,\alpha_{i-1}=\beta_{i-1},\alpha_i>\beta_i.
\end{cases}$
\end{center}
It is clear that this is a total order on $Mon(A)$ called \textit{deglex} order. If
$x^{\alpha}\succeq x^{\beta}$ but $x^{\alpha}\neq x^{\beta}$, we write $x^{\alpha}\succ x^{\beta}$.
Each element $f\in A-\{0\}$ can be represented in a unique way as $f=c_1x^{\alpha_1}+\cdots
+c_tx^{\alpha_t}$, with $c_i\in R-\{0\}$, $1\leq i\leq t$, and $x^{\alpha_1}\succ \cdots \succ
x^{\alpha_t}$. We say that $x^{\alpha_1}$ is the \textit{leader monomial} of $f$ and we write
$lm(f):=x^{\alpha_1}$ ; $c_1$ is the \textit{leader coefficient} of $f$, $lc(f):=c_1$, and
$c_1x^{\alpha_1}$ is the \textit{leader term} of $f$ denoted by $lt(f):=c_1x^{\alpha_1}$. If $f=0$,
we define $lm(0):=0,lc(0):=0,lt(0):=0$, and we set $X\succ 0$ for any $X\in Mon(A)$ (see also
Section \ref{sec3}). We observe that
\begin{center}
$x^{\alpha}\succ x^{\beta}\Rightarrow lm(x^{\gamma}x^{\alpha}x^{\lambda})\succ
lm(x^{\gamma}x^{\beta}x^{\lambda})$, for every $x^{\gamma},x^{\lambda}\in Mon(A)$.
\end{center}
\end{remark}
Natural and useful results that we will use later are the following properties.
\begin{proposition}\label{1.1.10a}
Let $A$ be a bijective skew $PBW$ extension of a ring $R$. Then,
\begin{enumerate}
\item[\rm (i)] $A$ is a right $R$-free module with basis $Mon(A)$.
\item[\rm(ii)] If $R$ is a domain, the $A$ is a domain.
\end{enumerate}
\end{proposition}
\begin{proof} See \cite{lezamareyes1}
\end{proof}

\begin{proposition}\label{associatedpbw} Let $A$ be a skew $PBW$ extension of $R$. Then, there
exists a quasi-commutative skew $PBW$ extension $A^{\sigma}$ of $R$ in \;$n$ variables
$z_1,\dots,z_n$ defined by
\begin{center}
$z_ir=c_{i,r}z_i$, $z_jz_i=c_{i,j}z_iz_j$, $1\leq i,j\leq n$,
\end{center}
where $c_{i,r},c_{i,j}$ are the same constants that define $A$. If $A$ is bijective then
$A^{\sigma}$ is also bijective.
\end{proposition}
\begin{proof}
See \cite{lezamareyes1}.
\end{proof}

\begin{theorem}\label{1.3.2}
Let $A$ be an arbitrary skew $PBW$ extension of the ring $R$. Then, $A$ is a filtered ring with
filtration given by
\begin{equation}\label{eq1.3.1a}
F_{m}:=
\begin{cases}
R,  &  \text{if $m=0$},\\
\{f\in A|\,\deg(f)\leq m\},  &  \text{if $m\geq 1$}
\end{cases}
\end{equation}
and the corresponding graded ring $Gr(A)$ is a quasi-commutative skew $PBW$ extension of $R$.
Moreover, if $A$ is bijective, then $Gr(A)$ is a quasi-commutative bijective skew $PBW$ extension
of $R$.
\end{theorem}
\begin{proof}
See \cite{lezamareyes1}.
\end{proof}

The next theorem characterizes the quasi-commutative skew $PBW$ extensions.

\begin{theorem}\label{1.3.3}
Let $A$ be a quasi-commutative skew $PBW$ extension of a ring $R$. Then,
\begin{enumerate}
\item[\rm (i)]$A$ is isomorphic to an iterated skew polynomial ring of
endomorphism type.
\item[\rm (ii)]If $A$ is bijective, then each
endomorphism is bijective.
\end{enumerate}
\end{theorem}
\begin{proof}
See \cite{lezamareyes1}.
\end{proof}
\begin{theorem}[Hilbert Basis Theorem]\label{1.3.4}
Let $A$ be a bijective skew $PBW$ extension of $R$. If $R$ is a left Noetherian ring then $A$ is
also a left Noetherian ring.
\end{theorem}
\begin{proof}
We repeat the proof given in \cite{lezamareyes1}. According to Theorem \ref{1.3.2}, $Gr(A)$ is a
quasi-commutative skew $PBW$ extension, and by the hypothesis, $Gr(A)$ is also bijective. By
Theorem \ref{1.3.3}, $Gr(A)$ is isomorphic to an iterated skew polynomial ring
$R[z_1;\theta_1]\cdots [z_{n};\theta_n]$ such that each $\theta_i$ is bijective, $1\leq i\leq n$.
This implies that $Gr(A)$ is a left Noetherian ring, and hence, $A$ is left Noetherian (see
\cite{McConnell}, Theorem 1.6.9).
\end{proof}

Many other properties of skew $PBW$ extensions have been studied recently, for example, the Ore's
theorem and Goldie's theorem were proved in \cite{lezamaore}, prime ideals were investigated in
\cite{Acosta2}, the groups $K_i$, $i\geq 0$, of algebraic $K$-theory were computed in
\cite{lezamareyes1}, etc. We want to conclude this section with two results that estimate the
global and Krull dimension of bijective skew $PBW$ extensions. We denote by ${\rm lgld}(S)$ the
left global dimension of the ring $S$ and by ${\rm lKdim}(S)$ its left Krull dimension (see
\cite{Rotman2} and \cite{McConnell}).

\begin{theorem}\label{3.1.1a}
Let $A=\sigma(R)\langle x_1,\dots,x_n\rangle$ be a bijective skew $PBW$ extension of a ring $R$.
Then,
\begin{center}
${\rm lgld}(R)\leq {\rm lgld}(A)\leq{\rm lgld}(R)+n$, if ${\rm lgld}(R)<\infty$.
\end{center}
If $A$ is quasi-commutative, then
\begin{center}
 ${\rm lgld}(A)={\rm lgld}(R)+n$.
\end{center}
In particular, if $R$ is semisimple, then ${\rm lgld}(A)=n$.
\end{theorem}
\begin{proof}
See \cite{lezamareyes1}.
\end{proof}

\begin{theorem}\label{3.3.4a}
Let $A$ be a bijective skew $PBW$ extension of a left Noetherian ring $R$. Then,
\begin{center}
${\rm lKdim}(R)\leq {\rm lKdim}(A)\leq {\rm lKdim}(R)+n$.
\end{center}
If $A$ is quasi-commutative, then
\begin{center}
${\rm lKdim}(A)={\rm lKdim}(R)+n$.
\end{center}
In particular, if $R=K$ is a field, then ${\rm lKdim}(A)=n$.
\end{theorem}
\begin{proof}
See \cite{lezamareyes1}.
\end{proof}

\begin{remark}
The last three theorems are valid for the right side.
\end{remark}

\subsection{More examples}\label{1.3}

\noindent Many other important and interesting  examples of bijective skew $PBW$ extensions were
presented and discussed in \cite{lezamareyes1} and \cite{Reyes2}. In this section we recall other
key examples, some of them will be used later to illustrate the algorithms that will be presented
later in this paper.

\begin{example}\label{1.1.18}
According to \cite{diffusionauthor}, a \textit{diffusion algebra} $\mathcal{D}$ over a field $K$ is
generated by $\{D_i,x_i\mid 1\le i\le n\}$ over $K$ with relations
\begin{equation*}
x_ix_j=x_jx_i, \ \ x_iD_j=D_jx_i, \ \ 1\leq i,j\leq n.
\end{equation*}
\begin{equation*}\label{diffusionrela}
c_{ij}D_iD_j-c_{ji}D_jD_i=x_jD_i-x_iD_j,\ \ i<j, c_{ij},c_{ji}\in K^{*}.
\end{equation*}
Thus, $\mathcal{D}\cong \sigma(\boldsymbol{K[x_1,\dots,x_n]})\langle D_1,\dotsc,D_n\rangle$ is a
bijective non quasi-commutative skew $PBW$ extension of $K[x_1,\dots,x_n]$. Observe that
$\mathcal{D}$ is not a $PBW$ extension neither an iterated skew polynomial ring of bijective type
(see Example \ref{gpbwexample}).
\end{example}

\begin{example}
Viktor Levandovskyy has defined in \cite{Levandovskyy} the $G$-algebras and he has constructed the
theory of Gröbner bases for them (see Section \ref{chapter5} of the present overview for the
Gröbner theory of bijective skew $PBW$ extensions). Let $K$ be a field, a $K$-algebra $A$ is called
a \textit{$G$-algebra} if $K \subset Z(A)$ (center of $A$) and $A$ is generated by a finite set
$\{x_1, \ldots, x_n\}$ of elements that satisfy the following conditions: (a) the collection of
standard monomials of $A$ is a $K$-basis of $A$. (b) $x_jx_i= c_{ij}x_ix_j+d_{ij}$, for $1 \leq i <
j \leq n$, with $c_{ij} \in K-\{0\}$ and $d_{ij}\in A$. (c) There exists a total order $<_A$ on
${\rm Mon}(A)$ such that for $i<j$, $lm(d_{ij}) <_A x_ix_j$. According to this definition,
$G$-algebras appear like more general than skew $PBW$ extensions since $d_{ij}$ is not necessarily
linear; however, in $G$-algebras the coefficients of polynomials are in a field and they commute
with the variables $x_1,\dots,x_n$. Note that the class of $G$-algebras does not include the class
of skew $PBW$ extensions over fields. For example, consider the $K$-algebra $\mathcal{A}$ generated
by $x,y,z$ subject to the relations
\begin{equation*}
yx-q_2xy=x,\ \ \ \ \ \ zx-q_1xz=z,\ \ \ \ \ \ \ \ zy=yz, \ \ \ \ q_1,q_2\in K.
\end{equation*}
Thus, $\mathcal{A}$ is not a $G$-algebra in the sense of \cite{Levandovskyy}. Note that if
$q_1,q_2\neq 0$, then $\mathcal{A}\cong \sigma(K)\langle x,y,z\rangle$ is a bijective non
quasi-commutative skew $PBW$ extension of $K$.
\end{example}
\begin{example}
\textit{Witten's deformation of} $\mathcal{U}(\mathfrak{sl}(2,K)$. E. Witten introduced and studied
a 7-parameter deformation of the universal enveloping algebra \linebreak
$\mathcal{U}(\mathfrak{sl}(2,K))$ over the field $K$, depending on a 7-tuple of parameters
\linebreak $\underline{\xi}=(\xi_1,\dotsc,\xi_7)$ of $K$ and subject to relations
\begin{equation*}
    xz-\xi_1zx=\xi_2x,\ \ \ zy-\xi_3yz=\xi_4y,\ \ \ \ yx-\xi_5xy=\xi_6z^2+\xi_7z.
\end{equation*}
    The resulting algebra is denoted by $W(\underline{\xi})$ and it is assumed that $\xi_1\xi_3\xi_5\neq 0$ (see \cite{Levandovskyy}).
    Note that if $\xi_2\xi_4\xi_6\neq 0$, then $W(\underline{\xi})\cong \sigma(\boldsymbol{\sigma(K[x])\langle z\rangle})\langle y\rangle$ is a bijective
    non quasi-commutative skew $PBW$ extension of
    $\sigma(K[x])\langle z\rangle$, and in turn, $\sigma(K[x])\langle z\rangle$ is a bijective non quasi-commutative skew $PBW$ extension of
    $K[x]$. In \cite{Levandovskyy} is proved that the only way that $W(\underline{\xi})$ is a
    $G$-algebra is when $\xi_1=\xi_3$ and $\xi_2=\xi_4$. Thus, in general,
    $W(\underline{\xi})$ is a skew $PBW$ extension but is not a $G$-algebra.
\end{example}

\begin{example}\label{1.1.21}
In \cite{Gomez-Torrecillas} (see also \cite{Gomez-Torrecillas2}) Bueso, Gómez-Torrecillas and
Verschoren defined a type of rings and algebras called \textit{left $PBW$ rings}. Many of rings and
algebras considered in \cite{lezamareyes1} (see also \cite{Reyes2}) can be interpreted also as left
$PBW$ rings. Next we present an example of skew $PBW$ extension that is not a left $PBW$ ring: let
$K$ be a field; for any $0\neq q\in K$, let $\mathcal{R}$ be an algebra generated by the variables
$a,b,c,d$ subject to the relations
\begin{align*}
ba&=qab,\ \ \ \ db=qbd,\ \ \ \ ca=qac,\ \ \ \ dc=qcd\\
bc&=\mu cb,\ \ \ \ ad-da=(q^{-1}-q)bc.
\end{align*}
for some $\mu\in K$. Then $\mathcal{R}$ is not a left $PBW$ ring unless $\mu=1$ (see
\cite{Gomez-Torrecillas2}). Thus, for $\mu\neq 1$, $\mathcal{R}\cong \sigma(K[b])\langle
a,c,d\rangle$ is a bijective non quasi-commutative skew $PBW$ extension of $K[b]$ that is not a
left $PBW$ ring.
\end{example}


\section{Finitely generated projective modules}

\noindent One of the main purposes of the present work is to study finitely generated projective
modules over skew $PBW$ extensions. Recall that if $S$ is a ring and $P$ is a module over $S$, $P$
is said to be \textit{projective} is there exists a $S$-module $P'$ and a free $S$-module $F$ such
that $P\oplus P'\cong F$; in particular, $P$ is a \textit{finitely generated projective module} if
there exists $r\geq 0$ such that $P\oplus P'\cong S^r$. Note that any free module is projective
(the null module $0=S^0$ is free by definition). Given a ring $S$, one of classical questions in
homological algebra is to determine if any finitely generated projective $S$-module is free. It is
well known that this is the case when $S$ is a principal ideal domain, or when $S$ is local (see a
matrix constructive proof of this fact below, Proposition \ref{5.1.1}), or when
$S=R[x_1,\dots,x_n]$, with $R$ a principal ideal domain (Quillen-Suslin Theorem, see \cite{Lam}).
For skew $PBW$ extensions, in general, the answer to this question is negative, the next trivial
example shows this (\cite{Lam}): if $K$ is a division ring, then $S:=K[x,y]$ has a module $P$ such
that $P\oplus S\cong S^2$, but $P$ is not free. Thus, instead of this problem, we can ask if for
skew $PBW$ extensions is true the \textit{Serre's theorem}, i.e., if any finitely generated
projective module $P$ is \textit{stably free}, i.e., there exist $r,s\geq 0$ such that $P\oplus
S^s\cong S^r$ (see Definition \ref{6.2.7}). We will say that a ring $S$ is $PSF$ is any finitely
generated projective $S$-module is stably free (Definition \ref{PFrings}).

\subsection{Serre's theorem}

\noindent Next we will prove the Serre's theorem for bijective skew $PBW$ extensions (see also
\cite{lezamareyes1}). Some preliminaries are needed.

\begin{proposition}\label{3.2.1b}
Let $S$ be a filtered ring. If $Gr(S)$ is left regular, then $S$ is left regular.
\end{proposition}
\begin{proof}
See \cite{McConnell}, Proposition 7.7.4.
\end{proof}
\begin{proposition}\label{3.2.3}
If $R$ is a left regular and left Noetherian ring and $\sigma$ is an automorphism, then
$R[x;\sigma]$ is left regular.
\end{proposition}
\begin{proof}
See \cite{McConnell}, Theorem 7.7.5.
\end{proof}

\begin{proposition}\label{3.5}
If $B$ is a filtered ring with filtration $\{B_p\}_{p\geq 0}$ such that $Gr(B)$ is left Noetherian,
left regular, and flat as right $B_0$-module, then $B$ is $PSF$ when $B_0$ is $PSF$.
\end{proposition}
\begin{proof}
See \cite{McConnell}, Theorem 12.3.2.
\end{proof}

\begin{theorem}\label{3.2.1a}
Let $A$ be a bijective skew $PBW$ extension of a ring $R$. If $R$ is a left regular and left
Noetherian ring, then $A$ is left regular.
\end{theorem}
\begin{proof}
Theorems \ref{1.3.2} and \ref{1.3.3} say that $Gr(A)$ is isomorphic to a iterated skew polynomial
ring of automorphism type with coefficients in $R$, then the result follows from Propositions
\ref{3.2.3} and \ref{3.2.1b}.
\end{proof}

\begin{theorem}[Serre's theorem]\label{3.6}
Let $A=\sigma(R)\langle x_1,\dots,x_n\rangle$ be a bijective skew $PBW$ extension of a ring $R$
such that $R$ is left Noetherian, left regular and $PSF$. Then $A$ is $PSF$.
\end{theorem}
\begin{proof}
By Theorem \ref{1.3.2}, $A$ is filtered, $A_0=R$, and $Gr(A)$ is a quasi-commutative bijective skew
$PBW$ extension of $R$; Theorem \ref{1.3.4} says that $Gr(A)$ is left Noetherian, and Theorem
\ref{3.2.1a} implies that $Gr(A)$ is left regular. Moreover, $Gr(A)$ is flat as right $R$-module
(see Proposition \ref{1.1.10a}), then assuming that $R$ is $PSF$ we get from Proposition \ref{3.5}
that $A$ is $PSF$.
\end{proof}

From Serre's theorem we conclude that the study of finitely generated projective modules over
bijective skew $PBW$ extensions is reduced to the investigation of stably free modules (of course
under certain conditions on the ring $R$ of coefficients). In a more general framework, and as
preparatory material for posterior studies in next sections, we are interested in studying when
stably free modules over enough arbitrary non-commutative rings are free. A well known result in
this direction is the Stafford's Theorem that we will show later. Many characterizations of stably
free modules will be presented also. There are different techniques to investigate stably free
modules, one of the purposes of the present work is to combine homological and matrix constructive
methods.

\subsection{$\mathcal{RC}$ and $\mathcal{IBN}$ rings}

In this section we recall some notations and well known elementary properties of linear algebra for
left modules over non-commutative rings. All rings are non-commutative and modules will be
considered on the left; $S$ will represent an arbitrary non-commutative ring; $S^r$ is the left
$S$-module of columns of size $r\times 1$; if $S^s\xrightarrow{f}S^r$ is an $S$-homomorphism then
there is a matrix associated to $f$ in the canonical bases of $S^r$ and $S^s$, denoted $F:=m(f)$,
and disposed by columns, i.e., $F\in M_{r\times s}(S)$. In fact, if $f$ is given by
\begin{align*}
S^s& \xrightarrow{f} S^r\ , \ \textbf{\emph{e}}_j \mapsto \textbf{\emph{f}}_j
\end{align*}
where $\{\textbf{\emph{e}}_1,\dots,\textbf{\emph{e}}_s\}$ is the canonical basis of $S^s$, $f$ can
be represented by a matrix, i.e., if $\textbf{\emph{f}}_j:=\begin{bmatrix}f_{1j} & \dots &
f_{rj}\end{bmatrix}^T$, then the matrix of $f$ in the canonical bases of $S^s$ and $S^r$ is
\begin{center}
$F:=\begin{bmatrix}\textbf{\emph{f}}_1 & \cdots & \textbf{\emph{f}}_s\end{bmatrix}=
\begin{bmatrix}
f_{11} & \cdots & f_{1s}\\
\vdots & & \vdots\\
f_{r1} & \cdots & f_{rs}
\end{bmatrix}\in M_{r\times s}(S)$.
\end{center}
Note that $Im(f)$ is the column module of $F$, i.e., the left $S$-module generated by the columns
of $F$, denoted by $\langle F \rangle$:
\begin{center}
$Im(f)=\langle f(\textbf{\emph{e}}_1),\dots,f(\textbf{\emph{e}}_s)\rangle=\langle
\textbf{\emph{f}}_1,\dots ,\textbf{\emph{f}}_s \rangle =\langle F \rangle$.
\end{center}
Moreover, observe that if $\textbf{\emph{a}}:=(a_1,\dots,a_s)^T\in S^s$, then
\begin{equation}\label{equ3.1.1}
f(\textbf{\emph{a}})=(\textbf{\emph{a}}^TF^T)^T.
\end{equation}
In fact,
\begin{align*}
f(\textbf{\emph{a}})&=a_1f(\textbf{\emph{e}}_1)+\cdots+a_sf(\textbf{\emph{e}}_s)=a_1\textbf{\emph{f}}_1+\cdots+a_s\textbf{\emph{f}}_s\\
&=a_1\begin{bmatrix}f_{11}\\ \vdots \\
f_{r1}\end{bmatrix}+\cdots+a_s\begin{bmatrix}f_{1s}\\ \vdots \\
f_{rs}\end{bmatrix}\\
&=\begin{bmatrix}a_1f_{11}+\cdots +a_sf_{1s}\\
\vdots \\
a_1f_{r1}+\cdots +a_sf_{rs}
\end{bmatrix}\\
& =(\begin{bmatrix}a_1 & \cdots & a_s\end{bmatrix}
\begin{bmatrix}f_{11} & \cdots & f_{r1}\\
\vdots & & \vdots \\f_{1s} & \cdots & f_{rs}\end{bmatrix})^T\\
&=(\textbf{\emph{a}}^TF^T)^T.
\end{align*}
Observe that function $m:Hom_{S}(S^s,S^r)\rightarrow M_{r\times s}(S)$ is bijective; moreover, if
$S^r\xrightarrow{g}S^p$ is a homomorphism, then the matrix of $gf$ in the canonical bases is
$m(gf)=(F^TG^T)^T$. Thus, $f:S^r\rightarrow S^r$ is an isomorphism if and only if $F^T\in GL_r(S)$.
Finally, let $C\in M_r(S)$; the columns of $C$ conform a basis of $S^r$ if and only if $C^T\in
GL_r(S)$.

We recall also that
\begin{center}
$Syz(\{\textbf{\emph{f}}_1,\dots ,\textbf{\emph{f}}_s\}):=\{\textbf{\emph{a}}:=(a_1,\dots,a_s)^T\in
S^s|a_1\textbf{\emph{f}}_1+\cdots+a_s\textbf{\emph{f}}_s=\textbf{0}\}$.
\end{center}
Note that
\begin{equation}
Syz(\{\textbf{\emph{f}}_1,\dots ,\textbf{\emph{f}}_s\})=\ker(f),
\end{equation}
but $Syz(\{\textbf{\emph{f}}_1,\dots ,\textbf{\emph{f}}_s\})\neq \ker(F)$ since we have
\begin{equation}\label{313}
\textbf{\emph{a}}\in Syz(\{\textbf{\emph{f}}_1,\dots ,\textbf{\emph{f}}_s\})\Leftrightarrow
\textbf{\emph{a}}^TF^T=\textbf{0}.
\end{equation}
A matrix characterization of finitely generated (f.g.) projective modules can be formulated in the
following way.
\begin{proposition}\label{5.1.1}
Let $S$ be an arbitrary ring and $M$ a $S$-module. Then, $M$ is a f.g. projective $S$-module if and
only if there exists a square matrix $F$ over $S$ such that $F^T$ is idempotent and $M=\langle
F\rangle$.
\end{proposition}
\begin{proof}
$\Rightarrow)$: If $M=0$, then $F=0$; let $M\neq 0$, there exists $s\geq 1$ and a $M'$ such that
$S^s=M\oplus M'$; let $f:S^s\to S^s$ be the projection on $M$ and $F$ the matrix of $f$ in the
canonical basis of $S^s$. Then, $f^2=f$ and $(F^TF^T)^T=F$, so $F^TF^T=F^T$; note that
$M=Im(f)=\langle F\rangle$.

$\Leftarrow)$: Let $f:S^s\to S^s$ be the homomorphism defined by $F$ (see (\ref{equ3.1.1})); from
$F^TF^T=F^T$ we get that $f^2=f$, moreover, since $M=\langle F\rangle$, then $Im(f)=M$ and hence
$M$ is direct summand of $S^s$, i.e., $M$ is f.g. projective (observe that the complement $M'$ of
$M$ is $\ker(f)$ and $f$ is the projection on $M$).
\end{proof}
\begin{remark}\label{5.1.2}
(i) When $S$ is commutative, or when we consider right modules instead of left modules,
(\ref{equ3.1.1}) says that $f(\textbf{\emph{a}})=F\textbf{\emph{a}}$. Moreover, in such cases
$Syz(\{\textbf{\emph{f}}_1,\dots ,\textbf{\emph{f}}_s\})=\ker(F)$ and the matrix of a compose
homomorphism $gf$ is given by $m(gf)=m(g)m(f)$. Note that $f:S^r\rightarrow S^r$ is an isomorphism
if and only if $F\in GL_r(S)$; moreover, $C\in GL_r(S)$ if and only if its columns conform a basis
of $S^r$. In addition, Proposition \ref{5.1.1} says that $M$ is a f.g. projective $S$-module if and
only if there exists a square matrix $F$ over $S$ such that $F$ is idempotent and $M=\langle
F\rangle$.

(ii) When the matrices of homomorphisms of left modules are disposed by rows instead of by columns,
i.e., if $S^{1\times s}$ is the left free module of rows vectors of length $s$ and the matrix of
the homomorphism $S^{1\times s} \xrightarrow{f} S^{1\times r}$ is defined by
\begin{center}
$F'=\begin{bmatrix}
f'_{11} & \cdots & f'_{1r}\\
\vdots & & \vdots\\
f'_{s1} & \cdots & f'_{sr}
\end{bmatrix}:=
\begin{bmatrix}
f_{11} & \cdots & f_{r1}\\
\vdots & & \vdots\\
f_{1s} & \cdots & f_{rs}
\end{bmatrix}\in M_{s\times r}(S)$,
\end{center}
then
\begin{equation}\label{euq3.2.1}
f(a_1,\dots,a_s)=(a_1,\dots,a_s)F',
\end{equation}
i.e., $f(\textbf{\emph{a}}^T)=\textbf{\emph{a}}^TF^T$. Thus, the values given by (\ref{euq3.2.1})
and (\ref{equ3.1.1}) agree since $F'=F^T$. Moreover, the composed homomorphism $gf$ means that $g$
acts first and then acts $f$, and hence, the matrix of $gf$ is given by $m(gf)=m(g)m(f)$. Note that
$f:S^{1\times r}\rightarrow S^{1\times r}$ is an isomorphism if and only if $m(f)\in GL_r(S)$;
moreover, $C\in GL_r(S)$ if and only if its rows conform a basis of $S^{1\times r}$. This left-row
notation is used in \cite{Cohn1}. Observe that with this notation, the proof of Proposition
\ref{5.1.1} says that $M$ is a f.g. projective $S$-module if and only if there exists a square
matrix $F$ over $S$ such that $F$ is idempotent and $M=\langle F\rangle$, but in this case $\langle
F\rangle$ represents the module generated by the rows of $F$. Note that Proposition \ref{5.1.1}
could has been formulated this way: In fact, the set of idempotents matrices of $M_s(S)$ coincides
with the set $\{F^T|F\in M_s(S), F^T \ \text{idempotent}\}$.
\end{remark}
\begin{definition}[\cite{Lam}]\label{8.1.2}
Let $S$ be a ring.
\begin{enumerate}
\item[{\rm(i)}]$S$ satisfies the rank condition
{\rm(}$\mathcal{RC}${\rm)} if for any integers $r,s\geq 1$, given an epimorphism
$S^r\xrightarrow{f} S^s$, then $r\geq s$.
\item[{\rm(ii)}]$S$ is
an $\mathcal{IBN}$ ring {\rm(}Invariant Basis Number{\rm)} if for any integers $r,s\geq 1$,
$S^r\cong S^s$ if and only if $r=s$.
\end{enumerate}
\begin{proposition}
Let $S$ be a ring.
\begin{enumerate}
\item[\rm (i)]$S$ is $\mathcal{RC}$ if and only if given any matrix $F\in M_{s\times
r}(S)$ the following condition holds:
\begin{center}
if $F$ has a right inverse then $r\geq s$.
\end{center}
\item[\rm (ii)]$S$ is $\mathcal{RC}$ if and only if given any matrix $F\in M_{s\times
r}(S)$ the following condition holds:
\begin{center}
if $F$ has a left inverse then $s\geq r$.
\end{center}
\end{enumerate}
\begin{proof}
See \cite{Gallego3}.
\end{proof}
\end{proposition}
\end{definition}
\begin{proposition}
$\mathcal{RC}\Rightarrow \mathcal{IBN}$.
\end{proposition}
\begin{proof}
Let $S^r\xrightarrow{f}S^s$ be an isomorphism, then $f$ is an epimorphism, and hence $r\geq s$;
considering $f^{-1}$ we get that $s\geq r$.
\end{proof}
\begin{example}\label{5.1.5}
Most of rings are $\mathcal{RC}$, and hence, $\mathcal{IBN}$.

(i) Any field $K$ is $\mathcal{RC}$: let $K^r\xrightarrow{f}K^s$ be an epimorphism, then
$\dim(K^r)=r=\dim(\ker(f))+s$, so $r\geq s$.

(ii) Let $S$ and $T$ be rings and let $S\xrightarrow{f} T$ be a ring homomorphism, if $T$ is a
$\mathcal{RC}$ ring then $S$ is also a $\mathcal{RC}$ ring. In fact, $T$ is a right $S$-module,
$t\cdot s:=tf(s)$; suppose that $S^r\xrightarrow{f}S^s$ is an epimorphism, then $T\otimes_S
S^r\xrightarrow{i_T\otimes f}T\otimes_S S^s$ is also an epimorphism of left $T$-modules, i.e., we
have an epimorphism $T^r\rightarrow T^s$, so $r\geq s$ (a similar result and proof is valid for the
$\mathcal{IBN}$ property).

(iii) We can apply the property proved in (ii) in many situations. For example, any commutative
ring $S$ is $\mathcal{RC}$: let $J$ be a maximal ideal of $S$, then the canonical homomorphism
$S\rightarrow S/J$ shows that $S$ is $\mathcal{RC}$ since $S/J$ is a field.

(iv) Any ring $S$ with finite uniform dimension (Goldie dimension, see \cite{McConnell} and
\cite{Goodearl}) is $\mathcal{RC}$: in fact, suppose that $S^r\xrightarrow{f}S^s$ is an
epimorphism, then $S^r\cong S^s\oplus M$ and hence $r\, {\rm udim}(S)=s\, {\rm udim}(S)+{\rm
udim}(M)$, so $r\geq s$.

(v) Since any left Noetherian ring $S$ has finite uniform dimension, then $S$ is $\mathcal{RC}$. In
particular, any left Artinian ring is $\mathcal{RC}$.
\end{example}
Since the objects studied in the present monograph are the skew $PBW$ extensions, it is natural to
investigate the $\mathcal{IBN}$ and $\mathcal{RC}$ properties for these rings.
\begin{proposition}\label{7.3.1}
Let $B$ be a filtered ring. If $Gr(B)$ is $\mathcal{RC}$ $($$\mathcal{IBN}$$)$, then $B$ is
$\mathcal{RC}$ $($$\mathcal{IBN}$$)$.
\end{proposition}
\begin{proof}
Let $\{B_p\}_{p\geq 0}$ be the filtration of $B$ and $f:B^r\to B^s$ an epimorphism. For $M:=B^r$ we
consider the standard positive filtration given by
\begin{center}
$F_0(M):=B_0\cdot e_1+\cdots+ B_0\cdot e_r$, $F_p(M):=B_pF_0(M)$, $p\geq 1$,
\end{center}
where $\{e_i\}_{i=1}^r$ is the canonical basis of $B^r$. Let $e_i':=f(e_i)$, then $B^s$ is
generated by $\{e_i'\}_{i=1}^r$ and $N:=B^s$ has an standard positive filtration given by
\begin{center}
$F_0(N):=B_0\cdot e_1'+\cdots+ B_0\cdot e_r'$, $F_p(N):=B_pF_0(N)$, $p\geq 1$.
\end{center}
Note that $f$ is filtered and strict: In fact, $f(F_p(M))=B_pf(F_0(M))=B_p(B_0\cdot f(e_1)+\cdots
+B_0\cdot f(e_r))=B_p(B_0\cdot e_1'+\cdots+ B_0\cdot e_r')=B_pF_0(N)=F_p(N)$. This implies that
$Gr(M)\xrightarrow{Gr(f)}Gr(N)$ is surjective. If we prove that $Gr(M)$ and $Gr(N)$ are free over
$Gr(B)$ with bases of $r$ and $s$ elements, respectively, then from the hypothesis we conclude that
$r\geq s$ and hence $B$ is $\mathcal{RC}$.

Since every $e_i\in F_0(M)$ and $F_p(M)=\sum_{i=1}^r \oplus B_p\cdot e_i$, $M$ is filtered-free
with filtered-basis $\{e_i\}_{i=1}^r$, so $Gr(M)$ is graded-free with graded-basis
$\{\overline{e_i}\}_{i=1}^r$, $\overline{e_i}:=e_i+F_{-1}(M)=e_i$ (recall that by definition of
positive filtration, $F_{-1}(M):=0$). For $Gr(N)$ note that $N$ is also filtered-free with respect
the filtration $\{F_p(N)\}_{p\geq 0}$ given above: Indeed, we will show next that the canonical
basis $\{f_j\}_{j=1}^s$ of $N$ is a filtered basis. If $f_j=x_{j1}\cdot e_1'+\cdots+x_{jr}\cdot
e_r'$, with $x_{ji}\in B_{p_{ij}}$, let $p:=\max\{p_{ij}\}$, $1\leq i\leq r$, $1\leq j\leq s$, then
$f_j\in F_p(N)$, moreover, for every $q$, $B_{q-p}\cdot f_1\oplus \cdots \oplus B_{q-p}\cdot
f_s\subseteq B_{q-p}F_p(N)\subseteq F_q(N)$ (recall that for $k<0$, $B_k=0$); in turn, let $x\in
F_q(N)$, then $x=b_1\cdot f_1+\cdots+b_s\cdot f_s$ and in $Gr(N)$ we have $\overline{x}\in
Gr(N)_q$, $\overline{x}=\overline{b_1}\cdot \overline{f_1}+\cdots+\overline{b_s}\cdot
\overline{f_s}$, if $b_j\in B_{u_j}$, let $u:=\max\{u_{j}\}$, so $\overline{b_j}\cdot
\overline{f_j}\in Gr(N)_{u+p}$, so $q=u+p$, i.e., $u=q-p$ and hence $x\in B_{q-p}\cdot f_1\oplus
\cdots \oplus B_{q-p}\cdot f_s$, Thus, we have proved that $B_{q-p}\cdot f_1\oplus \cdots \oplus
B_{q-p}\cdot f_s=F_q(N)$, for every $q$, and consequently, $\{f_j\}_{j=1}^s$ is a filtered basis of
$N$. From this we conclude that $Gr(N)$ is graded-free with graded-basis
$\{\overline{f_j}\}_{j=1}^s$, $\overline{f_j}:=f_j+F_{p-1}(N)$.

We can repeat the previuos proof for the $\mathcal{IBN}$ property but assuming that $f$ is an
isomorphism.
\end{proof}

\begin{corollary}
Let $A$ be a skew $PBW$ extension of a ring $R$. Then, $A$ is $\mathcal{RC}$ $(\mathcal{IBN})$ if
and only if $R$ is $\mathcal{RC}$ $(\mathcal{IBN})$.
\end{corollary}
\begin{proof}
We consider only the proof for $\mathcal{RC}$, the case $\mathcal{IBN}$ is completely analogous.

$\Rightarrow)$: Since $R\hookrightarrow A$, Example \ref{5.1.5} shows that if $A$ is
$\mathcal{RC}$, then $R$ is $\mathcal{RC}$.

$\Leftarrow)$: We consider first the skew polynomial ring $R[x;\sigma]$ of endomorphism type, then
$R[x;\sigma]\to R$ given by $p(x)\to p(0)$ is a ring homomorphism, so $R[x;\sigma]$ is
$\mathcal{RC}$ since $R$ is $\mathcal{RC}$. By Theorem \ref{1.3.3}, $Gr(A)$ is isomorphic to an
iterated skew polynomial ring $R[z_1;\theta_1]\cdots [z_{n};\theta_n]$, so $Gr(A)$ is
$\mathcal{RC}$. Only rest to apply Proposition \ref{7.3.1}.
\end{proof}

\begin{remark}\label{5.1.6}
(i) The condition $\mathcal{IBN}$ for rings is independent of the side we are considering the
modules. In fact, if we define left $\mathcal{IBN}$ rings and right $\mathcal{IBN}$ rings,
depending on left or right free $S$-modules, then $S$ is left $\mathcal{IBN}$ if and only if $S$ is
right $\mathcal{IBN}$ (see \cite{Lezama6}). The same is true for the $\mathcal{RC}$ property.

(ii) \textbf{From now on we will assume that all rings considered in the present work are
$\mathcal{RC}$}.
\end{remark}

\subsection{Characterizations of stably free modules}

\begin{definition}\label{6.2.7}
Let $M$ be a $S$-module and $t\geq 0$ an integer. $M$ is stably free of rank $t\geq 0$ if there
exist an integer $s\geq 0$ such that $S^{s+t}\cong S^s\oplus M$.
\end{definition}
The rank of $M$ is denoted by $\rm{rank}(M)$. Note that any stably free module $M$ is finitely
generated and projective. Moreover, as we will show in the next proposition, $\rm{rank}(M)$ is well
defined, i.e., $\rm{rank}(M)$ is unique for $M$.
\begin{proposition}\label{stablyfree1}
Let $t,t',s,s'\geq 0$ integers such that $S^{s+t}\cong S^s\oplus M$ and $S^{s'+t'}\cong
S^{s'}\oplus M$. Then, $t'=t$.
\end{proposition}
\begin{proof}
We have $S^{s'}\oplus S^{s+t}\cong S^{s'}\oplus S^s\oplus M$ and $S^{s}\oplus S^{s'+t'}\cong
S^{s}\oplus S^{s'}\oplus M$, then since $S$ is an $\mathcal{IBN}$ ring,  $s'+s+t=s+s'+t'$, and
hence $t'=t$.
\end{proof}
\begin{corollary}\label{527}
$M$ is stably free of rank $t\geq 0$ if and only if there exist integers $r,s\geq 0$ such that
$S^r\cong S^s\oplus M$, with $r\geq s$ and $t=r-s$.
\end{corollary}
\begin{proof}
If $M$ is stably free of rank $t$, then $S^{s+t}\cong S^s\oplus M$ for some integers $s,t\geq 0$,
then taking $r:=s+t$ we get the result. Conversely, if there exist integers $r,s\geq 0$ such that
$S^r\cong S^s\oplus M$, with $r\geq s$, then $S^{s+r-s}\cong S^s\oplus M$, i.e., $M$ is stably free
of rank $r-s$.
\end{proof}
\begin{proposition}\label{528}
Let $M$ be an $S$-module and let $r,s\geq 0$ integers such that $S^r\cong S^s\oplus M$. Then $r\geq
s$.
\end{proposition}
\begin{proof}
The canonical projection $S^r\rightarrow S^s$ is an epimorphism, but since we are assuming that $S$
is $\mathcal{RC}$, then $r\geq s$.
\end{proof}
\begin{corollary}\label{5.2.9}
$M$ is stably free if and only if there exist integers $r,s\geq 0$ such that $S^r\cong S^s\oplus
M$.
\end{corollary}
\begin{proof}
This is a direct consequence of Corollary \ref{527} and Proposition \ref{528}.
\end{proof}
Next we will present many characterizations of stably free modules over non-commutative rings
\begin{theorem}\label{stablefreeconditions}
Let $M$ be an $S$-module. Then, the following conditions are equivalent
\begin{enumerate}
\item[\rm (i)]$M$ is stably free.
\item[\rm (ii)]$M$ is projective and has a finite free resolution:
\begin{center}
$0 \rightarrow S^{t_k}\xrightarrow{f_k}S^{t_{k-1}}\xrightarrow{f_{k-1}}\cdots \xrightarrow{f_2}
S^{t_1}\xrightarrow{f_1} S^{t_0}\xrightarrow{f_0} M\rightarrow 0$.
\end{center}
In this case
\begin{equation}\label{eq5.2.2}
{\rm rank}(M)=\sum_{i=0}^k (-1)^it_i.
\end{equation}
\item[\rm (iii)]$M$ is isomorphic to the kernel of an
epimorphism of free modules: $M\cong \ker(\pi)$, $\pi:S^r\rightarrow S^s$.
\item[\rm (iv)]$M$ is projective and has a finite presentation $S^s\xrightarrow{f_1} S^r\xrightarrow{f_0} M\rightarrow
0$, where $\ker(f_0)$ is stably free.
\item[\rm (v)]$M$ has a finite presentation $S^s\xrightarrow{f_1} S^r\xrightarrow{f_0} M\rightarrow
0$, where $f_1$ has a left inverse.
\end{enumerate}
\end{theorem}
\begin{proof}
See \cite{Lang}, Chapter 21, \cite{Sturmfels}, and \cite{McConnell}, Chapter 11).
\end{proof}
\begin{definition}
A finite presentation
\begin{equation}\label{equ6.2.14}
S^s\xrightarrow{f_1} S^r\xrightarrow{f_0} M\rightarrow 0
\end{equation}
of a $S$-module $M$ is minimal if $f_1$ has a left inverse.
\end{definition}
\begin{corollary}
Let $M$ be an $S$-module. Then, $M$ is stably free if and only if $M$ has a minimal presentation.
\end{corollary}
\begin{proof}
Theorem \ref{stablefreeconditions}, (i)$\Leftrightarrow$(v).
\end{proof}
Unimodular matrices are closely related with the stably free modules.
\begin{definition}\label{unimodular}
Let $F$ be a matrix over $S$ of size $r\times s$. Then
\begin{enumerate}
\item[\rm (i)]Let $r\geq s$. $F$ is unimodular if and only if $F$ has a left inverse.
\item[\rm (ii)]Let $s\geq r$. $F$ is unimodular if and only if $F$ has a right inverse.
\end{enumerate}
The set of unimodular column matrices of size $r\times 1$ is denoted by $Um_{c}(r,S)$. $Um_r(s,S)$
is the set of unimodular row matrices of size $1\times s$.
\end{definition}
\begin{remark}\label{sizeinverses}
Note that a column matrix is unimodular if and only if the left ideal generated by its entries
coincides with $S$, and a row matrix is unimodular if and only if the right ideal generated by its
entries is $S$.
\end{remark}

We can add some others characterizations of stably free modules (compare with \cite{Quadrat}, Lemma
16).
\begin{corollary}\label{stablyfreecharacterization1}
Let $M$ be an $S$-module. Then the following conditions are equivalent:
\begin{enumerate}
\item[\rm (i)]$M$ is stably free.
\item[\rm (ii)]$M$ is projective and has a finite system of generators $\textbf{f}_1, \dots ,\textbf{f}_r$
such that $Syz\{\textbf{f}_1, \dots ,\textbf{f}_r\}$ is the module generated by the columns of a
matrix $F_1$ of size $r\times s$ such that $F_1^{T}$ has a right inverse.
\item[\rm (iii)]$M$ is projective and has a finite system of generators $\textbf{f}_1, \dots ,\textbf{f}_r$
such that $Syz\{\textbf{f}_1, \dots ,\textbf{f}_r\}$ is the module generated by the columns of a
matrix $F_1$ of size $r\times s$ such that $F_1^{T}$ is unimodular.
\end{enumerate}
\end{corollary}
\begin{proof}
See \cite{Gallego3}.
\end{proof}
Another interesting result about stably free modules over arbitrary $\mathcal{RC}$ rings  is
presented next. For this, we recall that if $M$ is a finitely presented left $S$-module with
presentation $S^s\xrightarrow{f_1} S^r\xrightarrow{f_0} M\rightarrow 0$ and $F_1$ is the matrix of
$f_1$ in the canonical bases, then the right $S$-module $M^T$ defined by $M^T:=S^s/Im(f_1^{T})$,
where $f_1^T:S^r\rightarrow S^s$ is the homomorphism of right free $S$-modules induced by the
matrix $F_1^T$, is called the \textit{transposed module} of $M$. Thus, $M^T$ is given by the
presentation $S^r\xrightarrow{f_1^T} S^s\rightarrow M^T\rightarrow 0$.
\begin{theorem}\label{3.2.19a}
Let $M$ be an $S$-module with exact sequence $0\rightarrow
S^s\xrightarrow{f_1}S^r\xrightarrow{f_0}M\rightarrow 0$. Then, $M^T\cong Ext_S^{1}(M,S)$ and the
following conditions are equivalent:
\begin{enumerate}
\item[\rm (i)]$M$ is stably free.
\item[\rm (ii)]$M$ is projective.
\item[\rm (iii)]$M^T=0$.
\item[\rm (iv)]$F_1^T$ has a right inverse.
\item[\rm (v)]$f_1$ has a left inverse.
\end{enumerate}
\end{theorem}
\begin{proof}
See \cite{Chyzak3}.
\end{proof}

\subsection{Stafford's theorem: a constructive proof}

A well known result due Stafford says that any left ideal of the Weyl algebras $D:=A_n(K)$ or
$B_n(K)$, with $\text{\rm char}(K)=0$, is generated by two elements, (see \cite{Stafford2} and
\cite{Quadrat}). From the Stafford's Theorem follows that any stably free left module $M$ over $D$
with $\text{\rm rank}(M)\geq 2$ is free. In \cite{Quadrat} is presented a constructive proof of
this result that we want to study for arbitrary $\mathcal{RC}$ rings. Actually, we will consider
the generalization given in \cite{Quadrat} staying that any stably free left $S$-module $M$ with
$\text{\rm rank}(M)\geq \text{\rm sr}(S)$ is free, where $\text{\rm sr}(S)$ denotes the stable rank
of the ring $S$. Our proof have been adapted from \cite{Quadrat}, however we do not need the
involution of ring $S$ used in \cite{Quadrat} because of our left notation for modules and column
representation for homomorphism. This could justify our special left-column notation. In order to
apply the results to bijective skew $PBW$ extensions we will estimate the stable rank of such
extensions.

\begin{definition}
Let $S$ be a ring and $\textbf{v}:=\begin{bmatrix}v_1 & \dots & v_r\end{bmatrix}^T\in Um_c(r,S)$ an
unimodular column vector. $\textbf{v}$ is called stable $($reducible$)$ if there exists
$a_1,\dots,a_{r-1}\in S$ such that $\textbf{v}':=\begin{bmatrix}v_1+a_1v_r & \dots &
v_{r-1}+a_{r-1}v_r\end{bmatrix}^T$ is unimodular. It says that the left stable rank of $S$ is
$d\geq 1$, denoted ${\rm sr}(S)=d$, if $d$ is the least positive integer such that every unimodular
column vector of length $d+1$ is stable. It says that ${\rm sr}(S)=\infty$ if for every $d\geq 1$
there exits a non stable unimodular column vector of length $d+1$.
\end{definition}
\begin{remark}\label{7.3.2}
In a similar way is defined the right stable rank of $S$, however, both ranks coincide; we list
next some well known properties of the stable rank (see \cite{Artamonov2}, \cite{Bass2},
\cite{Chen}, \cite{McConnell}, \cite{Quadrat}, \cite{Stafford2}, \cite{Stafford3}),
\cite{Vaserstein}, or also \cite{Garcia}).
\end{remark}

\begin{proposition}\label{6.2.29}
Let $S$ be a ring and $\textbf{v}:=\begin{bmatrix}v_1 & \dots & v_r\end{bmatrix}^T$ an unimodular
stable column vector over $S$, then there exists $U\in E_r(S)$ such that
$U\textbf{v}=\textbf{e}_1$.
\end{proposition}
\begin{proof}
See \cite{Quadrat}
\end{proof}
Next we present a lemma that checks when a stably free module is free.
\begin{lemma}\label{matricialhermite2}
Let $S$ be a ring and $M$ a stably free $S$-module given by a minimal presentation
$S^s\xrightarrow{f_1}S^r\xrightarrow{f_0}M\rightarrow 0$. Let $g_1:S^r\to S^s$ such that
$g_1f_1=i_{S^s}$. Then the following conditions are equivalent:
\begin{enumerate}
\item[\rm (i)]$M$ is free of dimension $r-s$.
\item[\rm (ii)]There exists a matrix $U\in GL_r(S)$ such that
$UG_1^T=\begin{bmatrix}I_s\\
0\end{bmatrix}$, where $G_1$ is the matrix of $g_1$ in the canonical bases. In such case, the last
$r-s$ columns of $U^T$ conform a basis for $M$. Moreover, the first $s$ columns of $U^T$ conform
the matrix $F_1$ of $f_1$ in the canonical bases.
\item[\rm (iii)]There exists
a matrix $V\in GL_r(S)$ such that $G_1^T$ coincides with the first $s$ columns of $V$, i.e.,
$G_1^T$ can be completed to an invertible matrix $V$ of $GL_r(S)$.
\end{enumerate}
\end{lemma}
\begin{proof}
By the hypothesis, the exact sequence $0\rightarrow
S^s\xrightarrow{f_1}S^r\xrightarrow{f_0}M\rightarrow 0$ splits, so $F_1^T$ admits a right inverse
$G_1^T$, where $F_1$ is the matrix of $f_1$ in the canonical bases and $G_1$ is the matrix of
$g_1:S^r\to S^s$, with $g_1f_1=i_{S^s}$, i.e., $F_1^TG_1^T=I_s$. Moreover, there exists $g_0:M\to
S^r$ such that $f_0g_0=i_M$. From this we get also the split sequence $0\to
M\xrightarrow{g_0}S^r\xrightarrow{g_1}S^s\to 0$. Note that $M\cong \ker(g_1)$.

${\rm (i)} \Rightarrow {\rm (ii)}$: We have $S^r=\ker(g_1)\oplus Im(f_1)$; by the hypothesis
$\ker(g_1)$ is free. If $s=r$ then $\ker(g_1)=0$ and hence $f_1$ is an isomorphism, so
$f_1g_1=i_{S^s}$, i.e., $G_1^TF_1^T=I_s$. Thus, we can take $U:=F_1^{T}$.

Let $r>s$; if $\{\textbf{\emph{e}}_1, \dots ,\textbf{\emph{e}}_s\}$ is the canonical basis of
$S^s$, then $\{\textbf{\emph{u}}_1, \dots ,\textbf{\emph{u}}_s\}$ is a basis of $Im(f_1)$ with
$\textbf{\emph{u}}_i:=f_1(\textbf{\emph{e}}_i)$, $1\leq i\leq s$; let $\{\textbf{\emph{v}}_1, \dots
,\textbf{\emph{v}}_p\}$ be a basis of $\ker(g_1)$ with $p=r-s$. Then, $\{\textbf{\emph{v}}_1, \dots
,\textbf{\emph{v}}_p,\textbf{\emph{u}}_1, \dots ,\textbf{\emph{u}}_s\}$ is a basis of $S^r$. We
define $S^r\xrightarrow{h} S^r$ by $h(\textbf{\emph{e}}_i):=\textbf{\emph{u}}_i$ for $1\leq i\leq
s$, and $h(\textbf{\emph{e}}_{s+j})=\textbf{\emph{v}}_j$ for $1\leq j\leq p$. Clearly $h$ is
bijective; moreover,
$g_1h(\textbf{\emph{e}}_i)=g_1(\textbf{\emph{u}}_i)=g_1f_1(\textbf{\emph{e}}_i)=\textbf{\emph{e}}_i$
and $g_1h(\textbf{\emph{e}}_{s+j})=g_1(\textbf{\emph{v}}_j)=\textbf{0}$, i.e.,
$H^TG_1^T=\begin{bmatrix}I_s \\0\end{bmatrix}$. Let $U:=H^T$, so we observe that the last $p$
columns of $U^T$ conform a basis of $\ker(g_1)\cong M$ and the first $s$ columns of $U^T$ conform
$F_1$.

${\rm (ii)} \Rightarrow {\rm (i)}$: Let $U_{(k)}$ the $k$-th row of $U$, then
\begin{center}
$UG_1^T=[U_{(1)}\cdots U_{(s)}\cdots U_{(r)}]^TG_1^T=\begin{bmatrix}I_s\\ 0\end{bmatrix}$,
\end{center}
so
$U_{(i)}G_1^T=\textbf{\emph{e}}_i^T$, $1\leq i\leq s$, $U_{(s+j)}G_1^T=\textbf{0}$, $1\leq j\leq p$
with $p:=r-s$. This means that $(U_{(s+j)})^T\in \ker(g_1)$ and hence $\langle (U_{(s+j)})^T|1\leq
j\leq p\rangle\subseteq \ker(g_1)$. On the other hand, let $\textbf{\emph{c}}\in \ker(g_1)\subseteq
S^r$, then $\textbf{\emph{c}}^TG_1^T=\textbf{0}$ and $\textbf{\emph{c}}^TU^{-1}UG_1^T=\textbf{0}$,
thus
$\textbf{\emph{c}}^TU^{-1}\begin{bmatrix}I_s \\
0\end{bmatrix}=\textbf{0}$ and hence $(\textbf{\emph{c}}^TU^{-1})^T\in \ker(l)$, where $l:S^r\to
S^s$ is the homomorphism with matrix $\begin{bmatrix}I_s & 0\end{bmatrix}$. Let
$\textbf{\emph{d}}=[d_1, \dots ,d_r ]^T\in
\ker(l)$, then $[d_1, \dots ,d_r ]\begin{bmatrix}I_s \\
0\end{bmatrix}=\textbf{0}$ and from this we conclude that $d_1=\cdots =d_s=0$, i.e.,
$\ker(l)=\langle \textbf{\emph{e}}_{s+1}, \textbf{\emph{e}}_{s+2},\dots,
\textbf{\emph{e}}_{s+p}\rangle$. From $(\textbf{\emph{c}}^TU^{-1})^T\in \ker(l)$ we get that
$(\textbf{\emph{c}}^TU^{-1})^T=a_1\cdot \textbf{\emph{e}}_{s+1}+\cdots +a_p\cdot
\textbf{\emph{e}}_{s+p}$, so $\textbf{\emph{c}}^TU^{-1}=(a_1\cdot \textbf{\emph{e}}_{s+1}+\cdots
+a_p\cdot \textbf{\emph{e}}_{s+p})^T$, i.e., $\textbf{\emph{c}}^T=(a_1\cdot
\textbf{\emph{e}}_{s+1}+\cdots +a_p\cdot \textbf{\emph{e}}_{s+p})^TU$ and from this we get that
$\textbf{\emph{c}}\in \langle (U_{(s+j)})^T|1\leq j\leq p\rangle$. This proves that
$\ker(g_1)=\langle (U_{(s+j)})^T|1\leq j\leq p\rangle$; but since $U$ is invertible, then
$\ker(g_1)$ is free of dimension $p$. We have proved also that the last $p$ columns of $U^T$
conform a basis for $\ker(g_1)\cong M$.

${\rm (ii)} \Leftrightarrow {\rm (iii)}$: $UG_1^T=\begin{bmatrix}I_s\\0\end{bmatrix}$ if and only
if $G_1^T=U^{-1}\begin{bmatrix}I_s \\0\end{bmatrix}$, but the first $s$ columns of
$U^{-1}\begin{bmatrix}I_s
\\0\end{bmatrix}$ coincides with the first $s$ columns of $U^{-1}$;
taking $V:=U^{-1}$ we get the result.
\end{proof}

\begin{theorem}\label{7.3.6}
Let $S$ be a ring. Then any stably free $S$-module $M$ with \linebreak ${\rm rank}(M)\geq \text{\rm
sr}(S)$ is free with dimension equals to ${\rm rank}(M)$.
\end{theorem}
\begin{proof}
Since $M$ is stably free it has a minimal presentation, and hence, it is given by an exact sequence
\begin{center}
$0\rightarrow S^s\xrightarrow{f_1}S^r\xrightarrow{f_0}M\rightarrow 0$;
\end{center}
moreover, note that ${\rm rank}(M)=r-s$. Since this sequence splits, $F_1^T$ admits a right inverse
$G_1^T$, where $F_1$ is the matrix of $f_1$ in the canonical bases and $G_1$ is the matrix of
$g_1:S^r\to S^s$, with $g_1f_1=i_{S^s}$. The idea of the proof is to find a matrix $U\in GL_r(S)$
such that $UG_1^T=\begin{bmatrix}I_s\\0\end{bmatrix}$ and then apply Lemma \ref{matricialhermite2}.

We have $F_1^TG_1^T=I_s$ and from this we get that the first column $\textbf{\emph{g}}_1$ of
$G_1^T$ is unimodular, but since $r>r-s\geq \text{\rm sr}(S)$, then $\textbf{\emph{g}}_1$ is
stable, and by Proposition \ref{6.2.29}, there exists $U_1\in E_r(S)$ such that
$U_1\textbf{\emph{g}}_1=\textbf{\emph{e}}_1$. If $s=1$, we finish since
$G_1^T=\textbf{\emph{g}}_1$.

Let $s\geq 2$; we have
\begin{center}
$U_1G_1^T=\begin{bmatrix}1 & *\\
0 & F_2\end{bmatrix}$, $F_2\in M_{(r-1)\times (s-1)}(S)$.
\end{center}
Note that $U_1G_1^T$ has a left inverse (for instance $F_1^TU_1^{-1}$), and the form of this left
inverse is
\begin{center}
$L=\begin{bmatrix}1 & *\\
0 & L_2\end{bmatrix}$, $L_2\in M_{(s-1)\times (r-1)}(S)$,
\end{center}
and hence $L_2F_2=I_{s-1}$. The first column of $F_2$ is unimodular and since $r-1>r-s\geq
\text{\rm sr}(S)$ we apply again Proposition \ref{6.2.29} and we obtain a matrix $U_2'\in
E_{r-1}(S)$ such that
\begin{center}
$U_2'F_2=\begin{bmatrix}1 & *\\
0 & F_3\end{bmatrix}$, $F_3\in M_{(r-2)\times (s-2)}(S)$.
\end{center}
Let
\begin{center}
$U_2:=\begin{bmatrix}1 & 0\\
0 & U_2'\end{bmatrix}\in E_r(S)$,
\end{center}
then we have
\begin{center}
$U_2U_1G_1^T=\begin{bmatrix}1 & * & *\\
0 & 1 & *\\
0 & 0 & F_3\end{bmatrix}$.
\end{center}
By induction on $s$ and multiplying on the left by elementary matrices we get a matrix $U\in
E_r(S)$ such that
\begin{center}
$UG_1^T=\begin{bmatrix}I_s\\0\end{bmatrix}$.
\end{center}
\end{proof}
\begin{corollary}[Stafford]\label{5.3.7}
Let $D:=A_n(K)$ or $B_n(K)$, with $\text{\rm char}(K)=0$. Then, any stably free left $D$-module $M$
satisfying $\text{\rm rank}(M)\geq 2$ is free.
\end{corollary}
\begin{proof}
The results follows from Theorem \ref{7.3.6} since ${\rm sr}(D)=2$.
\end{proof}


\section{Hermite rings}

Rings for which all stably free modules are free have occupied special attention in homological
algebra. In this section we will consider matrix-constructive interpretation of such rings. The
material presented here can be considered as preparatory for the next section when we will study
the Hermite condition for skew $PBW$ extensions. Recall that all rings considered are
$\mathcal{RC}$ (see Remark \ref{5.1.6}).

\subsection{Matrix descriptions of Hermite rings}

\begin{definition}\label{PFrings}
Let $S$ be a ring.
\begin{enumerate}
\item[\rm (i)]$S$ is a $PF$ ring if every f.g. projective $S$-module is free.
\item[\rm (ii)]$S$ is a $PSF$ ring if every f.g. projective $S$-module is stably free.
\item[\rm (iii)]$S$ is a Hermite ring, property denoted by $H$, if any stably
free $S$-module is free.
\end{enumerate}
\end{definition}
The right versions of the above rings (i.e., for right modules) are defined in a similar way and
denoted by $PF_r$, $PSF_r$ and $H_r$, respectively. We say that $S$ is a $\mathcal{PF}$ ring if $S$
is $PF$ and $PF_r$ simultaneously; similarly, we define the properties $\mathcal{PSF}$ and
$\mathcal{H}$. However, we will prove below later that these properties are left-right symmetric,
i.e., they can be denoted simply by $\mathcal{PF}$, $\mathcal{PSF}$ and $\mathcal{H}$.

From Definition \ref{PFrings} we get that
\begin{align}
& \label{eqPF} H \cap\  PSF = PF.
\end{align}
The following theorem gives a matrix description of $H$ rings (see \cite{Cohn1} and compare with
\cite{Lezama5} for the particular case of commutative rings. In \cite{Chen} is presented a
different and independent proof of this theorem for right modules).
\begin{theorem}\label{6.2.1}
Let $S$ be a ring. Then, the following conditions are equivalent.
\begin{enumerate}
\item[\rm (i)]$S$ is $H$.
\item[\rm (ii)]For every $r\geq 1$, any unimodular row matrix $\textbf{u}$ over $S$ of size $1\times
r$ can be completed to an invertible matrix of $GL_r(S)$ adding $r-1$ new rows.
\item[\rm (iii)]For every $r\geq 1$, if $\textbf{u}$ is an unimodular row matrix of size $1\times
r$, then there exists a matrix $U\in GL_r(S)$ such that $\textbf{u}U=(1,0,\dots, 0)$.
\item[$\rm (iv)$]For every $r\geq 1$, given an unimodular matrix $F$ of size $s\times r$,
$r\geq s$, there exists $U\in GL_r(S)$ such that
\begin{center}
$FU= \begin{bmatrix}I_s & | & 0\end{bmatrix}$.
\end{center}
\end{enumerate}
\end{theorem}
\begin{proof}
See \cite{Gallego3}.
\end{proof}

\begin{remark}\label{6.2.3}
In a similar way as we observed in Remark \ref{5.1.2}, if we consider right modules and the right
$S$-module structure on the module $S^r$ of columns vectors, the conditions of the previous theorem
can be reformulated properly, see \cite{Gallego3}.
\end{remark}

\subsection{Matrix characterization of $PF$ rings}

In \cite{Cohn1} are given some matrix characterizations of projective-free rings, in this
subsection we present another matrix interpretation of this important class of rings. The main
result presented here (Corollary \ref{6.2.4}) extends Theorem 6.2.2 in \cite{Lezama5}. This result
has been proved independently also in \cite{Chen}, Proposition 11.4.9. A matrix proof of a
Kaplansky theorem about finitely generated projective modules over local rings is also included.

\begin{theorem}\label{6.3.1a}
Let $S$ be a Hermite ring and $M$ a f.g. projective module given by the column module of a matrix
$F\in M_s(S)$, with $F^T$ idempotent. Then, $M$ is free with $dim(M)=r$ if and only if there exists
a matrix $U\in M_s(S)$ such that $U^T\in GL_s(S)$ and
\begin{equation}\label{matrixequivalence}
(U^T)^{-1}F^TU^T=\begin{bmatrix}0 & 0\\
0 & I_r \end{bmatrix}^T.
\end{equation}
In such case, a basis of $M$ is given by the last $r$ rows of $(U^T)^{-1}$.
\end{theorem}
\begin{proof}
See \cite{Gallego3}.
\end{proof}
From the previous theorem we get the following matrix description of $PF$ rings.
\begin{corollary}\label{matrixpfrings}
Let $S$ be a ring. $S$ is $PF$ if and only if for each $s\geq 1$, given a matrix $F\in M_s(S)$,
with $F^T$ idempotent, there exists a matrix $U\in M_s(S)$ such that $U^T\in GL_s(S)$ and
\begin{equation}\label{6.3.1}
(U^T)^{-1}F^TU^T=\begin{bmatrix}0 & 0\\
0 & I_r \end{bmatrix}^T,
\end{equation}
where $r=dim(\langle F\rangle)$, $0\leq r\leq s$.
\end{corollary}
\begin{proof}
See \cite{Gallego3}.
\end{proof}
\begin{remark}\label{6.3.3a}
(i) If we consider right modules instead of left modules, then the previous corollary can be
reformulated in the following way: $S$ is $PF_r$ if and only if for each $s\geq 1$, given an
idempotent matrix $F\in M_s(S)$, there exists a matrix $U\in GL_s(S)$ such that
\begin{equation}\label{6.3.3b}
UFU^{-1}=\begin{bmatrix}0 & 0\\
0 & I_r \end{bmatrix},
\end{equation}
where $r=dim(\langle F\rangle)$, $0\leq r\leq s$, and $\langle F\rangle$ represents the right
$S$-module generated by the columns of $F$. The proof is as in the commutative case, see
\cite{Lezama5}.

(ii) Considering again left modules and disposing the matrices of homomorphisms by rows and
composing homomorphisms from the left to the right (see Remark \ref{5.1.2}), we  get the
characterization (\ref{6.3.3b}) for the $PF$ property. However, observe that in this case $\langle
F\rangle$ represents the left $S$-module generated by the rows of $F$. Note that Corollary
\ref{matrixpfrings} could has been formulated this way: In fact,
\begin{center}
$\begin{bmatrix}0 & 0\\
0 & I_r \end{bmatrix}^T=\begin{bmatrix}0 & 0\\
0 & I_r \end{bmatrix}$
\end{center}
and we can rewrite (\ref{6.3.1}) as (\ref{6.3.3b}) changing $F^T$ by $F$  (see Remark \ref{5.1.2})
and $(U^T)^{-1}$ by $U$.

(iii) If $S$ is a commutative ring, of course $PF=PF_r=\mathcal{PF}$. However, we will prove in
Corollary \ref{4.5} that the projective-free property is left-right symmetric for general rings.
\end{remark}
\begin{corollary}\label{6.2.4}
$S$ is $PF$ if and only if for each $s\geq 1$, given an idempotent matrix $F\in M_s(S)$, there
exists a matrix $U\in GL_s(S)$ such that
\begin{equation}\label{eq6.2.4}
UFU^{-1}=\begin{bmatrix}0 & 0\\
0 & I_r \end{bmatrix},
\end{equation}
where $r=dim(\langle F\rangle)$, $0\leq r\leq s$, and $\langle F\rangle$ represents the left
$S$-module generated by the rows of $F$.
\end{corollary}
\begin{proof}
This is the content of the part (ii) in the previous remark.
\end{proof}
\begin{corollary}\label{4.5}
Let $S$ be a ring. $S$ is $PF$ if and only if $S$ is $PF_r$, i.e., $PF=PF_r=\mathcal{PF}$.
\end{corollary}
\begin{proof}
Let $F\in M_s(S)$ be an idempotent matrix, if $S$ is $PF$, then there exists $P\in GL_s(S)$ such
that
\begin{equation*}
UFU^{-1}=\begin{bmatrix}0 & 0\\
0 & I_r \end{bmatrix},
\end{equation*}
where $r$ is the dimension of the left $S$-module generated by the rows of $F$. Observe that
$UFU^{-1}$ is also idempotent, moreover, the matrices $X:=UF$ and $Y:=U^{-1}$ satisfy $UFU^{-1}=XY$
and $F=YX$, then from Proposition 0.3.1 in \cite{Cohn1} we conclude that the left $S$-module
generated by the rows of $UFU^{-1}$ coincides with the left $S$-module generated by the rows of
$F$, and also, the right $S$-module generated by the columns of $UFU^{-1}$ coincides with the right
$S$-module generated by the columns of $F$. This implies that the $S$-module generated by the rows
of $F$ coincides with the right $S$-module generated by the columns of $F$. This means that $S$ is
$PF_r$. The symmetry of the problem completes the proof.
\end{proof}
Another interesting matrix characterization of $\mathcal{PF}$ rings is given in \cite{Cohn1},
Proposition 0.4.7: a ring $S$ is $\mathcal{PF}$ if and only if given an idempotent matrix $F\in
M_s(S)$ there exist matrices $X\in M_{s\times r}(S), Y\in M_{r\times s}(S)$ such that $F=XY$ and
$YX=I_r$. A similar matrix interpretation can be given for $PSF$ rings using Proposition 0.3.1 in
\cite{Cohn1} and Corollary \ref{5.2.9}.
\begin{proposition}
Let $S$ be a ring. Then,
\begin{enumerate}
\item[\rm (i)]$S$ is $PSF$ if and only if given an idempotent matrix $F\in M_r(S)$ there exist $s\geq 0$
and matrices $X\in M_{(r+s)\times r}(S), Y\in M_{r\times (r+s)}(S)$ such that
\begin{center}
$\begin{bmatrix}F & 0\\ 0 & I_s\end{bmatrix}=XY$ and $YX=I_r$.
\end{center}
\item[\rm (ii)]$PSF=PSF_r=\mathcal{PSF}$.
\end{enumerate}
\end{proposition}
\begin{proof}
Direct consequence of Proposition 0.3.1 in \cite{Cohn1} and Corollary \ref{5.2.9}.
\end{proof}
For the $H$ property we have a similar characterization that proves the symmetry of this condition.
\begin{proposition}\label{4.7}
Let $S$ be a ring. Then,
\begin{enumerate}
\item[\rm (i)]$S$ is $H$ if and only if given an idempotent matrix $F\in M_r(S)$ with factorization
\begin{center}
$\begin{bmatrix}F & 0\\ 0 & 1\end{bmatrix}=XY$ and $YX=I_r$, for some matrices $X\in M_{(r+1)\times
r}(S), Y\in M_{r\times (r+1)}(S)$,
\end{center}
there exist matrices $X'\in M_{r\times (r-1)}(S),Y'\in M_{(r-1)\times r}(S)$ such that $F=X'Y'$ and
$Y'X'=I_{r-1}$.
\item[\rm (ii)]$H=H_r=\mathcal{H}$.
\end{enumerate}
\end{proposition}
\begin{proof}
See \cite{Gallego3}.
\end{proof}
We conclude this subsection given a matrix constructive proof of a well known Kaplansky's theorem.
\begin{proposition}
Any local ring $S$ is $\mathcal{PF}$.
\end{proposition}
\begin{proof}
Let $M$ a projective left $S$-module. By Remark \ref{5.1.2}, part (ii), there exists an idempotent
matrix $F=[f_{ij}]\in M_s(S)$ such that the module generated by the rows of $F$ coincides with $M$.
According to Corollary \ref{6.2.4}, we need to show that there exists $U\in GL_s(S)$ such that the
relation (\ref{eq6.2.4}) holds. The proof is by induction on $s$.

$s=1$: In this case $F=[f_{ij}]=[f]$; since $S$ is local, its idempotents are trivial, then $f=1$ or $f=0$ and hence $M$ is free.\\
$s=2$: In view of fact that $S$ is local, two possibilities may arise:

$f_{11}$  is invertible. Then, one can find  $G\in GL_{2}(S)$ such that
$GFG^{-1}=\begin{bmatrix}1&0\\0 & f\end{bmatrix}$,
for some $f\in S$. For this it is enough to take $G=\begin{bmatrix}1& f_{11}^{-1}f_{12}\\
-f_{21}f_{11}^{-1} & 1\end{bmatrix}$; to show that this matrix is invertible with inverse
\begin{center}
$G^{-1}=\begin{bmatrix}f_{11}& -f_{12}\\ f_{21} & -f_{21}f_{11}^{-1}f_{12}+1\end{bmatrix}$
\end{center}
we can
use the relations that exist between the entries of $F$. See for example that $GG^{-1}=I_{2}$:
\begin{itemize}
\item[] $f_{11}+f_{11}^{-1}f_{12}f_{21}=1$ because $f_{11}^{2}+f_{12}f_{21}=f_{11}$ and $f_{11}$ is invertible;
\item[] $-f_{12}-f_{11}^{-1}f_{12}f_{21}f_{11}^{-1}f_{12}+f_{11}^{-1}f_{12}=-f_{12}+(1-f_{11}^{-1}f_{12}f_{21})f_{11}^{-1}f_{12}=-f_{12}+f_{11}f_{11}^{-1}f_{12}=0$;
\item[] $-f_{21}f_{11}^{-1}f_{11}+f_{21}=0$;
\item[] $f_{21}f_{11}^{-1}f_{12}-f_{21}f_{11}^{-1}f_{12}+1=1$.
\end{itemize}
Similar calculations show that $G^{-1}G=I_{2}$. Since $F$ is idempotent, $f$ so is; applying the
case $s=1$ we get the result.

$1-f_{11}$ is invertible. In the same way, we can find $H\in GL_{2}(S)$ such that
$HFH^{-1}=\begin{bmatrix}0&0\\0 & g\end{bmatrix}$; for this it is enough to take
\begin{center}
$H=\begin{bmatrix}1& -(1-f_{11})^{-1}f_{12}\\ f_{21} &
-f_{21}(1-f_{11})^{-1}f_{12}+1\end{bmatrix}$;
\end{center}
note that $H^{-1}=\begin{bmatrix}1-f_{11}& (1-f_{11})^{-1}f_{12}\\
-f_{21} & 1\end{bmatrix}$. Indeed $HH^{-1}=I_{2}$:
\begin{itemize}
\item[] $1-f_{11}+(1-f_{11})^{-1}f_{12}f_{21}=1-f_{11}+f_{11}=1$ because $f_{12}f_{21}=(1-f_{11})f_{11}$;
\item[] $(1-f_{11})^{-1}f_{12}-(1-f_{11})^{-1}f_{12}=0$;
\item[] $f_{21}(1-f_{11})+f_{21}(1-f_{11})^{-1}f_{12}f_{21}-f_{21}=f_{21}(1-f_{11})+f_{21}f_{11}-f_{21}=0$;
\item[] $f_{21}(1-f_{11})^{-1}f_{12}-f_{21}(1-f_{11})^{-1}f_{12}+1=1$.
\end{itemize}
An analogous calculation shows that $H^{-1}H=I_{2}$. Note that  $g$ is an idempotent of $S$, then
$g=0$  or $g=1$ and the statement follows.

Now suppose that the result holds for $s-1$; considering both possibilities for $f_{11}$ we have:

If $f_{11}$ is invertible, taking
\[G=\begin{bmatrix} 1& f_{11}^{-1}f_{12}& f_{11}^{-1}f_{13} &\cdots &f_{11}^{-1}f_{1s}\\ -f_{21}f_{11}^{-1} & 1 & 0 &\cdots & 0\\
 -f_{31}f_{11}^{-1} & 0 & 1&\cdots & 0\\ \vdots & & &\cdots & \\-f_{s1}f_{11}^{-1} & 0 & 0& \cdots & 1\end{bmatrix}\]
we have that $G\in GL_{s}(S)$ and its inverse is:
\[G^{-1}=\begin{bmatrix} f_{11}& -f_{12}& -f_{13} &\cdots &-f_{1s}\\ f_{21} & -f_{21}f_{11}^{-1}f_{12}+1 & -f_{21}f_{11}^{-1}f_{13}&\cdots & -f_{21}f_{11}^{-1}f_{1s}\\
 f_{31} & -f_{31}f_{11}^{-1}f_{12} & -f_{31}f_{11}^{-1}f_{13}+1&\cdots & -f_{31}f_{11}^{-1}f_{1s}\\ \vdots & & &\cdots & \\f_{s1} & -f_{s1}f_{11}^{-1}f_{12} & -f_{s1}f_{11}^{-1}f_{13}& \cdots & -f_{s1}f_{11}^{-1}f_{1s}+1\end{bmatrix}.\]
In fact, see that $GG^{-1}=I_{s}$:
\begin{itemize}
\item[] $f_{11}+f_{11}^{-1}f_{12}f_{21}+\cdots+f_{11}^{-1}f_{1s}f_{s1}=1$ because $f_{11}^{2}+f_{12}f_{21}+\cdots+f_{1s}f_{s1}=f_{11}$;
\item[] $-f_{12}-f_{11}^{-1}f_{12}f_{21}f_{11}^{-1}f_{12}+f_{11}^{-1}f_{12}-f_{11}^{-1}f_{13}f_{31}f_{11}^{-1}f_{12}-\cdots-$ \linebreak $f_{11}^{-1}f_{1s}f_{s1}f_{11}^{-1}f_{12}
$ \linebreak  $= -f_{12}+
(1-f_{11}^{-1}\sum_{i=2}^{s}f_{1i}f_{i1})f_{11}^{-1}f_{12}=-f_{12}+f_{11}f_{11}^{-1}f_{12}=0$;
\item[]$\vdots$
\item[]$-f_{1s}-f_{11}^{-1}f_{12}f_{21}f_{11}^{-1}f_{1s}-f_{11}^{-1}f_{13}f_{31}f_{11}^{-1}f_{1s}- \cdots - f_{11}^{-1}f_{1s}f_{s1}f_{11}^{-1}f_{1s}+f_{11}^{-1}f_{1s}= -f_{1s}+ (1-f_{11}^{-1}\sum_{i=2}^{s}f_{1i}f_{i1}) f_{11}^{-1}f_{1s}=-f_{1s}+f_{11}f_{11}^{-1}f_{1s}=0$;
\item[]$-f_{21}f_{11}^{-1}f_{11}+f_{21}=0$; $f_{21}f_{11}^{-1}f_{12}-f_{21}f_{11}^{-1}f_{12}+1=1$; $f_{21}f_{11}^{-1}f_{1i}-f_{21}f_{11}^{-1}f_{1i}=0$ for every $3\leq i\leq s$;
\item[]$\vdots$
\item[]$-f_{s1}f_{11}^{-1}f_{11}+f_{s1}=0$; $f_{s1}f_{11}^{-1}f_{1i}-f_{s1}f_{11}^{-1}f_{1i}=0$ for every $2\leq i\leq s-1$ and, finally, $f_{s1}f_{11}^{-1}f_{1s}-f_{s1}f_{11}^{-1}f_{1s}+1=1$.
\end{itemize}
Similarly, $G^{-1}G= I_{s}$. Moreover, $GFG^{-1}=\begin{bmatrix}1 & 0_{1,s-1}\\0_{s-1,1}&
F_{1}\end{bmatrix}$ where $F_{1}\in M_{s-1}(S)$ is an idempotent matrix. Only remains to apply the
induction hypothesis.

If $1-f_{11}$ is invertible, taking {\tiny
\[H=\begin{bmatrix}1 &-(1-f_{11})^{-1}f_{12}& -(1-f_{11})^{-1}f_{13}&\cdots &-(1-f_{11})^{-1}f_{1s}\\ f_{21}& -f_{21}(1-f_{11})^{-1}f_{12}+1 &-f_{21}(1-f_{11})^{-1}f_{13}& \cdots & -f_{21}(1-f_{11})^{-1}f_{1s} \\f_{31}& -f_{31}(1-f_{11})^{-1}f_{12} &-f_{31}(1-f_{11})^{-1}f_{13}+1& \cdots & -f_{31}(1-f_{11})^{-1}f_{1s} \\\vdots &  & & \cdots &  \\ f_{s1}& -f_{s1}(1-f_{11})^{-1}f_{12} &-f_{s1}(1-f_{11})^{-1}f_{13}& \cdots &
-f_{s1}(1-f_{11})^{-1}f_{1s}+1\end{bmatrix}\]} we have that $H\in GL_{s}(S)$ with inverse given by:
\[H^{-1}=\begin{bmatrix}1-f_{11} &(1-f_{11})^{-1}f_{12}& (1-f_{11})^{-1}f_{13}&\cdots &(1-f_{11})^{-1}f_{1s}\\ -f_{21}& 1 &0& \cdots & 0\\-f_{31}& 0 &1& \cdots & 0 \\\vdots &  & & \cdots &  \\ -f_{s1}& 0 &0& \cdots & 1\end{bmatrix}.\]
In fact, note that  $HH^{-1}=I_s$:
\begin{itemize}
\item[] $1-f_{11}+(1-f_{11})^{-1}\sum_{i=2}^{s}f_{1i}f_{i1}=1-f_{11}+f_{11}=1$ because $\sum_{i=2}^{s}f_{1i}f_{i1}=(1-f_{11})f_{11}$ and $(1-f_{11})$ is invertible; also $(1-f_{11})^{-1}f_{1i}-(1-f_{11})^{-1}f_{1i}$ for $2\leq i\leq s$;
\item[] $f_{21}(1-f_{11})+f_{21}\sum_{i=1}^{s}(1-f_{11})^{-1}f_{1i}f_{i1}-f_{21}=-f_{21}f_{11}+f_{21}f_{11}=0$; $f_{21}(1-f_{11})^{-1}f_{12}- f_{21}(1-f_{11})^{-1}f_{12}+1=1$; and $f_{21}(1-f_{11})^{-1}f_{1i}-f_{21}(1-f_{11})^{-1}f_{1i}=0$ for $3\leq i\leq s$.
\item[] $\vdots$
\item[]$f_{s1}(1-f_{11})+f_{s1}\sum_{i=1}^{s}(1-f_{11})^{-1}f_{1i}f_{i1}-f_{s1}=-f_{s1}f_{11}+f_{21}f_{11}=0$; $f_{s1}(1-f_{11})^{-1}f_{1i}-f_{s1}(1-f_{11})^{-1}f_{1i}=0$ for $3\leq i\leq s-1$ and, finally, $f_{s1}(1-f_{11})^{-1}f_{1s}- f_{s1}(1-f_{11})^{-1}f_{1s}+1=1$.
\end{itemize}
Similarly, we can to show that $H^{-1}H=I_s$. Furthermore, we have also
\begin{center}
$HFH^{-1}=\begin{bmatrix}0&
0_{1,s-1}\\0_{s-1,1}& F_{2}\end{bmatrix}$
\end{center}
 with $F_{2}\in M_{s-1}(S)$ an idempotent matrix. One more
time we apply the induction hypothesis.
\end{proof}


\section{$d$-Hermite rings and skew $PBW$ extensions}

Under suitable conditions on the ring $R$ of coefficients, most of bijective skew $PBW$ extensions
are $\mathcal{PSF}$ (see Theorem \ref{3.6}). A different situation occurs for the $\mathcal{H}$
property. In fact, as we observed before, if $K$ is a division ring, then $S:=K[x,y]$ has a module
$M$ such that $M\oplus S\cong S^2$, but $M$ is not free, i.e., $S$ is not $\mathcal{H}$. Another
example occurs in Weyl algebras: let $K$ be a field, with ${\rm char}(K)=0$, the Weyl algebra
$A_1(K)=K[t][x;\frac{d}{dt}]$ is not $\mathcal{H}$ since there exist stably free modules of rank
$1$ over $A_n(K)$ that are not free (\cite{Cohn1}, Corollary 1.5.3; see also \cite{McConnell},
Example 11.1.4). In this section we will study a weaker condition than the $\mathcal{H}$ property
for skew $PBW$ extensions: the $d$-Hermite condition. Recall that we always assume that all rings
are $\mathcal{RC}$.

\subsection{$d$-Hermite rings}

The following proposition induces the definition of $d$-Hermite rings.
\begin{proposition}\label{7.1.1}
Let $S$ be a ring. For any integer $d\geq 0$, the following statements are equivalent:
\begin{enumerate}
\item[{\rm (i)}]Any stably free module of rank $\geq d$ is free.
\item[{\rm (ii)}]Any unimodular row matrix over $S$ of length $\geq d+1$ can be
completed to an invertible matrix over $S$.
\item[{\rm (iii)}]For every $r\geq d+1$, if $\textbf{u}$ is an unimodular row matrix of size $1\times
r$, then there exists a matrix $U\in GL_{r}(S)$ such that $\textbf{u}U=(1,0,\dots, 0)$, i.e.,
$GL_r(S)$ acts transitively on $Um_r(r,S)$.
\item[$\rm (iv)$]For every $r\geq d+1$, given an unimodular matrix $F$ of size $s\times r$,
$r\geq s$, there exists $U\in GL_r(S)$ such that
\begin{center}
$FU= \begin{bmatrix}I_s & | & 0\end{bmatrix}$.
\end{center}
\end{enumerate}
\end{proposition}
\begin{proof}
We can repeat the proof of Theorem 2 in \cite{Gallego3} taking $r\geq d+1$.
\end{proof}
\begin{definition}\label{7.1.2}
Let $S$ be a ring and $d\geq 0$ an integer. $S$ is $d$-Hermite, property denoted by
$d$-$\mathcal{H}$, if $S$ satisfies any of conditions in Proposition \ref{7.1.1}.
\end{definition}
The next result extends Proposition \ref{4.7}.
\begin{proposition}
The $d$-Hermite condition is left-right symmetric.
\end{proposition}

\begin{corollary}\label{7.1.4}
Let $S$ be a ring. Then, $S$ is ${\rm sr}(S)$-$\mathcal{H}$.
\end{corollary}
\begin{proof}
This follows from Definition \ref{7.1.2} and Theorem \ref{7.3.6}.
\end{proof}
\begin{corollary}\label{7.1.4b}
Let $S$ be a ring. If ${\rm sr}(S)=1$, then $S$ is $\mathcal{H}$.
\end{corollary}
\begin{proof}
According to Corollary \ref{7.1.4} $S$ is $1$-$\mathcal{H}$, however, it is well known that rings
with stable rank $1$ are cancellable (see \cite{Evans}), so by Proposition 12 in \cite{Gallego3},
$S$ is $\mathcal{H}$.
\end{proof}
\begin{remark}\label{7.1.5}
(i) Observe that $0$-Hermite rings coincide with $\mathcal{H}$ rings, and for commutative rings,
$1$-Hermite coincides also with $\mathcal{H}$ (see \cite{Lam}, Theorem I.4.11). If $K$ is a field
with ${\rm char}(K)=0$, by Corollary \ref{5.3.7}, $A_1(K)$ is $2$-$\mathcal{H}$ but, as we observed
before, $A_1(K)$ is not $1$-$\mathcal{H}$. In general, $\mathcal{H}\subsetneq
1$-$\mathcal{H}\subsetneq 2$-$\mathcal{H}\subsetneq \cdots$ (see \cite{Cohn1}).

(ii) Note that $\mathcal{H}=1$-$\mathcal{H}$$\cap \mathcal{WF}$ (a ring $S$ is $\mathcal{WF}$,
\textit{weakly finite}, if for all $n\geq 0$, $P\oplus S^n\cong S^n$ if and only if $P=0$).

(iii) Any left Artinian ring $S$ is $\mathcal{H}$ since ${\rm sr}(S)=1$. In particular, semisimple
and semilocal rings are $\mathcal{H}$.

(iv) Rings with big stable rank can be Hermite, for example \linebreak ${\rm
sr}(\mathbb{R}[x_1,\dots,x_n])=n+1$ (\cite{McConnell}, Theorem 11.5.9), but by Quillen-Suslin
Theorem, $\mathbb{R}[x_1,\dots,x_n]$ is $\mathcal{H}$.
\end{remark}

\subsection{Stable rank}

Corollaries \ref{5.3.7} and \ref{7.1.4} motivate the task of computing the stable rank of bijective
skew $PBW$ extensions. For this purpose we need to recall the famous stable range theorem. This
theorem relates the stable rank and the Krull dimension of a ring. The original version of this
classical result is due a Bass (1968, \cite{Bass2}) and states that if $S$ is a commutative
Noetherian ring and ${\rm Kdim}(S)=d$ then ${\rm sr}(S)\leq d+1$. Heitmann extends the theorem for
arbitrary commutative rings (1984, \cite{Heitmann}). Lombardi et. al. in 2004 (\cite{Lombardi},
Theorem 2.4; see also \cite{Lombardi3}) proved again the theorem for arbitrary commutative rings
using the Zariski lattice of a ring and the boundary ideal of an element. This proof is elementary
and constructive. Stafford in 1981 (\cite{Stafford3}) proved a non-commutative version of the
theorem for left Noetherian rings.

\begin{proposition}[Stable range theorem]\label{821}
Let $S$ be  a left Noetherian ring and ${\rm lKdim}(S)=d$, then ${\rm sr}(S)\leq d+1$.
\end{proposition}
\begin{proof}
See \cite{Stafford3}.
\end{proof}

From this we get the following modest result.

\begin{proposition}\label{7.2.1}
Let $R$ be a left Noetherian ring with finite left Krull dimension and $A=\sigma(R)\langle
x_1,\dots,x_n\rangle$ a bijective skew $PBW$ extension of $R$, then
\begin{center}
$1\leq {\rm sr}(A)\leq {\rm lKdim}(R)+n+1$,
\end{center}
and $A$ is $d$-$\mathcal{H}$, with $d:=({\rm lKdim}(R)+n+1)$.
\end{proposition}
\begin{proof}
The inequalities follow from Proposition \ref{821} and Theorem 4.2 in \cite{lezamareyes1}. The
second statement follows from Corollary \ref{7.1.4}.
\end{proof}

\begin{example}
The results in \cite{lezamareyes1} for the Krull dimension of many interesting examples of
bijective skew $PBW$ extensions can be combined with Proposition \ref{7.2.1} in order to get an
upper bound for the stable rank. With this we can estimate also the $d$-Hermite condition. The next
table gives such estimations:

\newpage

\begin{table}[htb]\label{table8.1}
\centering \tiny{
\begin{tabular}{|l|l|}\hline 
\textbf{Ring} & \textbf{U. B.}\\
\hline
\cline{1-2} Habitual polynomial ring $R[x_1,\dotsc,x_n]$ & ${\rm dim}(R)+n+1$ \\
\cline{1-2} Ore extension of bijective type $R[x_1;\sigma_1 ,\delta_1]\cdots [x_n;\sigma_n,\delta_n]$  & ${\rm dim}(R)+n+1$ \\
\cline{1-2} Weyl algebra $A_n(K)$ & $2n+1$ \\
\cline{1-2} Extended Weyl algebra $B_n(K)$ & $n+1$ \\
\cline{1-2} Universal enveloping algebra of a Lie algebra $\mathfrak{g}$, $\cU(\mathfrak{g})$, $K$ commutative ring & ${\rm dim}(K)+n+1$ \\
\cline{1-2} Tensor product $R\otimes_K \cU(\cG)$ & ${\rm dim}(R)+n+1$ \\
\cline{1-2} Crossed product $R*\cU(\cG)$ & ${\rm dim}(R)+n+1$ \\
\cline{1-2} Algebra of q-differential operators $D_{q,h}[x,y]$ & $3$ \\
\cline{1-2} Algebra of shift operators $S_h$ & $3$ \\
\cline{1-2} Mixed algebra $D_h$ & $4$ \\
\cline{1-2} Discrete linear systems $K[t_1,\dotsc,t_n][x_1,\sigma_1]\dotsb[x_n;\sigma_n]$ & $2n+1$ \\
\cline{1-2} Linear partial shift operators $K[t_1,\dotsc,t_n][E_1,\dotsc,E_n]$ & $2n+1$ \\
\cline{1-2} Linear partial shift operators $K(t_1,\dotsc,t_n)[E_1,\dotsc,E_n]$ & $n+1$ \\
\cline{1-2} L. P. Differential operators $K[t_1,\dotsc,t_n][\partial_1,\dotsc,\partial_n]$ & $2n+1$ \\
\cline{1-2} L. P. Differential operators $K(t_1,\dotsc,t_n)[\partial_1,\dotsc,\partial_n]$ & $n+1$ \\
\cline{1-2} L. P. Difference operators $K[t_1,\dotsc,t_n][\Delta_1,\dotsc,\Delta_n]$ & $2n+1$ \\
\cline{1-2} L. P. Difference operators $K(t_1,\dotsc,t_n)[\Delta_1,\dotsc,\Delta_n]$ & $n+1$ \\
\cline{1-2} L. P. $q$-dilation operators $K[t_1,\dotsc,t_n][H_1^{(q)},\dotsc,H_m^{(q)}]$ & $n+m+1$ \\
\cline{1-2} L. P. $q$-dilation operators $K(t_1,\dotsc,t_n)[H_1^{(q)},\dotsc,H_m^{(q)}]$ & $m+1$ \\
\cline{1-2} L. P. $q$-differential operators $K[t_1,\dotsc,t_n][D_1^{(q)},\dotsc,D_m^{(q)}]$ & $n+m+1$ \\
\cline{1-2} L. P. $q$-differential operators $K(t_1,\dotsc,t_n)[D_1^{(q)},\dotsc,D_m^{(q)}]$ & $m+1$ \\
\cline{1-2} Diffusion algebras & $2n+1$ \\
\cline{1-2} Additive analogue of the Weyl algebra $A_n(q_1,\dotsc,q_n)$ & $2n+1$ \\
\cline{1-2} Multiplicative analogue of the Weyl algebra $\cO_n(\lambda_{ji})$ & $n+1$ \\
\cline{1-2} Quantum algebra $\cU'(\mathfrak{so}(3,K))$ & $4$ \\
\cline{1-2} 3-dimensional skew polynomial algebras & $4$ \\
\cline{1-2} Dispin algebra $\cU(osp(1,2))$ & $4$ \\
\cline{1-2} Woronowicz algebra $\cW_{\nu}(\mathfrak{sl}(2,K))$ & $4$ \\
\cline{1-2} Complex algebra $V_q(\mathfrak{sl}_3(\mathbb{C}))$ & $11$ \\
\cline{1-2} Algebra \textbf{U} & $3n+1$ \\
\cline{1-2} Manin algebra $\cO_q(M_2(K))$ & $5$ \\
\cline{1-2} Coordinate algebra of the quantum group $SL_q(2)$ & $5$ \\
\cline{1-2} $q$-Heisenberg algebra \textbf{H}$_n(q)$ & $3n+1$ \\
\cline{1-2} Quantum enveloping algebra of $\mathfrak{sl}(2,K)$, $\cU_q(\mathfrak{sl}(2,K))$ & $4$ \\
\cline{1-2} Hayashi algebra $W_q(J)$ & $3n+1$\\
\cline{1-2} Differential operators on a quantum space $S_{\textbf{q}}$,
$D_{\textbf{q}}(S_{\textbf{q}})$ & $2n+1$ \\
\cline{1-2} Witten's Deformation of $\cU(\mathfrak{sl}(2,K)$ & $4$ \\
\cline{1-2} Quantum Weyl algebra of Maltsiniotis $A_n^{\textbf{q},\lambda}$, $K$ commutative ring & ${\rm dim}(K)+2n+1$\\
\cline{1-2} Quantum Weyl algebra $A_n(q,p_{i,j})$ & $2n+1$\\
\cline{1-2} Multiparameter Weyl algebra $A_n^{Q,\Gamma}(K)$ & $2n+1$\\
\cline{1-2} Quantum symplectic space $\cO_q(\mathfrak{sp}(K^{2n}))$ & $2n+1$ \\
\cline{1-2} Quadratic algebras in $3$ variables & $4$ \\
\hline
\end{tabular}}\label{table8.1}
\caption{Stable rank for some examples of bijective skew $PBW$ extensions.}\label{table8.1}
\end{table}
\end{example}

\begin{remark}
The values presented in Table \ref{table8.1} can be improved for some particular classes of skew
$PBW$ extensions. For example, it is well known that ${\rm sr}(A_n(K))=2$ if ${\rm char}(K)=0$ (see
Remark \ref{7.3.2}). A challenging problem is to give exactly values for the stable rank of all
examples of bijective $PBW$ extensions presented in \cite{lezamareyes1}.
\end{remark}

\subsection{Kronecker's theorem}

Close related to the stable range theorem is the Kronecker's theorem staying that if $S$ is a
commutative ring with ${\rm Kdim} (S)<d$, then every finitely generated ideal $I$ of $S$ has the
same radical as an ideal generated by $d$ elements. In this subsection we want to investigate this
theorem for non-commutative rings using the Zariski lattice and the boundary ideal, but
generalizing these tools and its properties to non-commutative rings. The main result will be
applied to skew $PBW$ extensions.

\begin{definition}\label{7.2.1e}
Let $S$ be a ring and $Spec(S)$ the set of all prime ideals of $S$. The Zariski lattice of $S$ is
defined by
\[
Zar(S):=\{D(X)|X\subseteq S\}, \ \text{with}\ D(X):=\bigcap_{X\subseteq P\in Spec(S)}P.
\]
\end{definition}
\noindent $Zar(S)$ is ordered with respect the inclusion. The description of the Zariski lattice is
presented in the next proposition, $\langle X\}, \langle X\rangle, \{X\rangle$ will represent the
left, two-sided, and right ideal of $S$ generated by $X$, respectively. $\vee$ denotes the $\sup$
and $\wedge$ the $\inf$.
\begin{proposition}\label{7.2.1d}
Let $S$ be a ring, $I,I_1,I_2, I_3$ two-sided ideals of $S$, $X\subseteq S$, and
$x_1,\dots,x_n,x,y\in S$. Then,
\begin{enumerate}
\item[\rm (i)]$D(X)=D(\langle X\})=D(\langle
X\rangle)=D(\{X \rangle)$.
\item[\rm (ii)]$D(I)=rad(S)$ if and only if $I\subseteq rad(S)$. In particular, $D(0)=rad(S)$.
\item[\rm (iii)]$D(I)=S$ if and only if $I=S$.
\item[\rm (iv)]$I\subseteq D(I)$ and $D(D(I))=D(I)$. Moreover, if $I_1\subseteq I_2$, then $D(I_1)\subseteq D(I_2)$.
\item[\rm (v)]Let $\{I_j\}_{j\in \mathcal{J}}$ a family of two-sided ideals of $S$. Then, $D(\sum_{j\in \mathcal{J}}I_j)=\vee_{j\in
\mathcal{J}}D(I_j)$. In particular, $D(x_1,\dots,x_n)=D(x_1)\vee\cdots \vee D(x_n)$.
\item[\rm (vi)]$D(I_1I_2)=D(I_1)\wedge D(I_2)$. In particular, $D(\langle x\rangle \langle y\rangle)=D(x)\wedge D(y)$.
\item[\rm (vii)]$D(x+y)\subseteq D(x,y)$.
\item[\rm (viii)]If $\langle x\rangle\langle y\rangle\subseteq D(0)$, then $D(x,y)=D(x+y)$.
\item[\rm (ix)]If $x\in D(I)$, then $D(I)=D(I,x)$.
\item[\rm (x)]If $\overline{S}:=S/I$, then $D(\overline{J})=\overline{D(J)}$, for any two-sided ideal $J$ of $S$ containing
$I$.
\item[\rm (xi)]$u\in D(I)$ if and only if $\overline{u}\in rad(S/I)$. In such case, if $u\in D(I)$, there exists $k\geq 1$ such that $u^k\in I$.
\item[\rm (xii)]$Zar(S)$ is distributive:
\begin{center}
$D(I_1)\wedge [D(I_2)\vee D(I_3)]=[D(I_1)\wedge D(I_2)]\vee [D(I_1)\wedge D(I_3)]$,

$D(I_1)\vee [D(I_2)\wedge D(I_3)]=[D(I_1)\vee D(I_2)]\wedge [D(_1)\vee D(I_3)]$.
\end{center}
\end{enumerate}
\end{proposition}
\begin{proof}
See \cite{Gallego4}.
\end{proof}

\begin{definition}
Let $S$ be a ring and $v\in S$, the boundary ideal of $v$ is defined by $I_v:=\langle v
\rangle+(D(0):\langle v\rangle)$, where $(D(0):\langle v\rangle):=\{x\in S|\langle v\rangle
x\subseteq D(0)\}$.
\end{definition}
Note that $I_v\neq 0$ for every $v\in S$. On the other hand, if $v$ is invertible or if $v=0$, then
$I_v=S$. If $S$ a domain and $v\neq 0$, then $I_v=\langle v\rangle$.

\begin{definition}\label{7.2.1c}
Let $S$ be a ring such that ${\rm lKdim}(S)$ exists. We say the $S$ satisfies the boundary
condition if for any $d\geq 0$ and every $v\in S$,
\begin{center}
${\rm lKdim}(S)\leq d\Rightarrow{\rm lKdim}(S/I_v)\leq d-1$.
\end{center}
\end{definition}

\begin{example}\label{846}
(i) Any commutative Noetherian ring satisfies the boundary condition: indeed, for commutative
Noetherian rings, the classical Krull dimension and the Krull dimension coincide, so we can apply
Theorem 13.2 in \cite{Lombardi3}.

(ii) Any prime ring $S$ with left Krull dimension satisfies the boundary condition: in fact, for
prime rings, any non-zero two sided ideal is essential, so ${\rm lKdim}(S/I_v)< {\rm lKdim}(S)$
(see \cite{McConnell}, Proposition 6.3.10).

(iii) Any domain with left Krull dimension satisfies the satisfies the boundary condition: indeed,
any domain is a prime ring.
\end{example}

\begin{theorem}[Kronecker]\label{7.2.1b}
Let $S$ be a domain such that ${\rm lKdim}(S)$ exists. If ${\rm lKdim}(S)<d$ and
$u_1,\dots,u_{d},u\in S$, then there exist $x_1,\dots,x_d\in S$ such that
\begin{center}
$D(u_1,\dots,u_{d},u)=D(u_1+x_1u,\dots, u_d+x_du)$.
\end{center}
\end{theorem}
\begin{proof}
See \cite{Gallego4}.

\end{proof}

\begin{corollary}\label{7.2.1a}
Let $S$ be a domain such that ${\rm lKdim}(S)$ exists. If ${\rm lKdim}(S)<d$ and $u_1,\dots,
u_{d+1}\in S$ are such that $\langle u_1,\dots,u_{d+1}\rangle =S$, then there exist elements
$x_1,\dots, x_d\in S $ such that $\langle u_1+x_1u_{d+1}, \dots, u_d+x_du_{d+1}\rangle =S$.
\end{corollary}
\begin{proof}
The statement follows directly from Proposition \ref{7.2.1d}, part (iii), and Theorem \ref{7.2.1b}.
\end{proof}

\begin{corollary}
Let $A=\sigma(R)\langle x_1,\dots,x_n\rangle$ be a bijective skew $PBW$ extension of a left
Noetherian domain $R$. If ${\rm lKdim}(R)<d$ and $u_1,\dots,u_{d+n},u\in A$, then there exist
$y_1,\dots,y_{d+n}\in A$ such that
\begin{center}
$D(u_1,\dots,u_{d+n},u)=D(u_1+y_1u,\dots, u_{d+n}+y_{d+n}u)$.
\end{center}
\end{corollary}
\begin{proof}
This follows directly from Proposition \ref{1.1.10a}, Theorem \ref{1.3.4}, Theorem 4.2 in
\cite{lezamareyes1}, and Theorem \ref{7.2.1b}.
\end{proof}


\section{Gröbner bases for skew $PBW$ extensions}\label{chapter5}

\noindent In order to make constructive the theory of projective modules, stably free modules and
Hermite rings studied in the previous sections, now we will study the theory of Gröbner bases of
left ideals and modules for bijective skew $PBW$ extensions. This theory was initially investigated
in \cite{Gallego2}, \cite{Jimenez} and \cite{Jimenez2} for the particular case of quasi-commutative
bijective skew $PBW$ extensions. We will extend the theory to arbitrary bijective skew $PBW$
extensions, in particular, Buchberger's algorithm will be established for general bijective case.
\textbf{Note that all examples listed in Table \ref{table8.1} are covered with our theory} (compare
with Section 1.4. in \cite{Rogalski}).

We start recalling the basic facts of Gröbner theory for arbitrary skew $PBW$ extensions; we will
use the notation given in Definition \ref{1.1.6}.

\subsection{Monomial orders in skew $PBW$ extensions}\label{sec3} Let $A=\sigma(R)\langle x_1, \dots
, x_n\rangle$ be an arbitrary skew $PBW$ extension of $R$ and let $\succeq$ be a total order
defined on $Mon(A)$. If $x^{\alpha}\succeq x^{\beta}$ but $x^{\alpha}\neq x^{\beta}$ we will write
$x^{\alpha}\succ x^{\beta}$. $x^{\beta}\preceq x^{\alpha}$ means that $x^{\alpha}\succeq
x^{\beta}$. Let $f\neq 0$ be a polynomial of $A$, if
\begin{center}
$f=c_1X_1+\cdots +c_tX_t$,
\end{center}
with $c_i\in R-\{0\}$ and $X_1\succ \cdots \succ X_t$ are the monomials of $f$, then $lm(f):=X_1$
is the \textit{leading monomial} of $f$, $lc(f):=c_1$ is the \textit{leading coefficient} of $f$
and $lt(f):=c_1X_1$ is the \textit{leading term} of $f$. If $f=0$, we define
$lm(0):=0,lc(0):=0,lt(0):=0$, and we set $X\succ 0$ for any $X\in Mon(A)$. Thus, we extend
$\succeq$ to $Mon(A)\cup\{0\}$.
\begin{definition}\label{monomialorder}
Let $\succeq$ be a total order on $Mon(A)$, it says that $\succeq$ is a monomial order on $Mon(A)$
if the following conditions hold:
\begin{enumerate}
\item[\rm (i)]For every $x^{\beta},x^{\alpha},x^{\gamma},x^{\lambda}\in Mon(A)$
\begin{center}
$x^{\beta}\succeq x^{\alpha}$ $\Rightarrow$ $lm(x^{\gamma}x^{\beta}x^{\lambda})\succeq
lm(x^{\gamma}x^{\alpha}x^{\lambda})$.
\end{center}

\item[\rm (ii)]$x^{\alpha}\succeq 1$, for every $x^{\alpha}\in
Mon(A)$.
\item[\rm (iii)]$\succeq$ is degree compatible, i.e., $|\beta|\geq |\alpha|\Rightarrow x^{\beta}\succeq
x^{\alpha}$.
\end{enumerate}
\end{definition}
Monomial orders are also called \textit{admissible orders}. The condition (iii) of the previous
definition is needed in the proof of the following proposition, and this one will be used in the
division algorithm (Theorem \ref{algdivforPBW}).
\begin{proposition}\label{132}
Every monomial order on $Mon(A)$ is a well order. Thus, there are not infinite decreasing chains in
$Mon(A)$.
\end{proposition}
\begin{proof}
See Proposition 12 in \cite{Gallego2}.
\end{proof}
From now on we will assume that $Mon(A)$ is endowed with some monomial order.
\begin{definition}
Let $x^{\alpha},x^{\beta}\in Mon(A)$, we say that $x^{\alpha}$ divides $x^{\beta}$, denoted by
$x^{\alpha}|x^{\beta}$, if there exists $x^{\gamma},x^{\lambda}\in Mon(A)$ such that
$x^{\beta}=lm(x^{\gamma}x^{\alpha}x^{\lambda})$. We will say also that any monomial $x^{\alpha}\in
Mon(A)$ divides the polynomial zero.
\end{definition}
\begin{proposition}
Let $x^{\alpha},x^{\beta}\in Mon(A)$ and $f,g\in A-\{0\}$. Then,
\begin{enumerate}
\item[\rm (a)]
$lm(x^{\alpha}g)=lm(x^{\alpha}lm(g))=x^{\alpha+\exp(lm(g))}$, i.e.,
\begin{center}
$\exp(lm(x^{\alpha}g))=\alpha+\exp(lm(g)$.
\end{center}
In particular,
\begin{center}
$lm(lm(f)lm(g))=x^{\exp(lm(f))+\exp(lm(g))}$, i.e.,

$\exp(lm(lm(f)lm(g)))=\exp(lm(f))+\exp(lm(g))$
\end{center}
and
\begin{equation}
lm(x^{\alpha}x^{\beta})=x^{\alpha+\beta}, \ i.e.,\ \exp(lm(x^{\alpha}x^{\beta}))=\alpha+\beta
.\label{divrelation}
\end{equation}
\item[\rm (b)]The
following conditions are equivalent:
\begin{enumerate}
\item[\rm (i)]$x^{\alpha}|x^{\beta}$.
\item[\rm (ii)]There exists a unique $x^{\theta}\in Mon(A)$ such that
$x^{\beta}=lm(x^{\theta}x^{\alpha})=x^{\theta+\alpha}$ and hence $\beta=\theta+\alpha$.
\item[\rm (iii)]There exists a unique $x^{\theta}\in Mon(A)$ such that
$x^{\beta}=lm(x^{\alpha}x^{\theta})=x^{\alpha+\theta}$ and hence $\beta=\alpha+\theta$.
\item[\rm (iv)]$\beta_i\geq \alpha_i$ for $1\leq i\leq n$, with
$\beta:=(\beta_1,\dots,\beta_n)$ and $\alpha:=(\alpha_1,\dots,\alpha_n)$.
\end{enumerate}
\end{enumerate}
\end{proposition}
\begin{proof}
See Proposition 14 in \cite{Gallego2}.
\end{proof}
\begin{remark}
We note that a least common multiple of monomials of $Mon(A)$ there exists: in fact, let
$x^{\alpha},x^{\beta}\in Mon(A)$, then $lcm(x^{\alpha},x^{\beta})=x^{\gamma}\in Mon(A)$, where
$\gamma=(\gamma_1,\dots,\gamma_n)$ with $\gamma_i:=\max\{\alpha_i,\beta_i\}$ for each $1\leq i\leq
n$.
\end{remark}

\subsection{Reduction in skew $PBW$ extensions}

Some natural computational conditions on $R$ will be assumed in the rest of this work (see
\cite{Lezama2}).
\begin{definition}\label{LGSring}
A ring $R$ is left Gröbner soluble {\rm(}$LGS${\rm)} if the following conditions hold:
\begin{enumerate}
\item[\rm (i)]$R$ is left Noetherian.
\item[\rm (ii)]Given $a,r_1,\dots,r_m\in R$ there exists an
algorithm which decides whether $a$ is in the left ideal $Rr_1+\cdots+Rr_m$, and if so, find
$b_1,\dots,b_m\in R$ such that $a=b_1r_1+\cdots+b_mr_m$.
\item[\rm (iii)]Given $r_1,\dots,r_m\in R$ there exists an
algorithm which finds a finite set of generators of the left $R$-module
\begin{center}
$Syz_R[r_1\ \cdots \ r_m]:=\{(b_1,\dots,b_m)\in R^m|b_1r_1+\cdots+b_mr_m=0\}$.
\end{center}
\end{enumerate}
\end{definition}
\begin{remark}\label{LGS}
The three above conditions imposed to $R$ are needed in order to guarantee a Gröbner theory in the
rings of coefficients, in particular, to have an effective solution of the membership problem in
$R$ (see (ii) in Definition \ref{reductionsigmapbw} below). F\textbf{rom now on we will assume that
$A=\sigma(R)\langle x_1,\dots,x_n\rangle$ is a skew $PBW$ extension of $R$, where $R$ is a $LGS$
ring and $Mon(A)$ is endowed with some monomial order.}
\end{remark}
\begin{definition}\label{reductionsigmapbw}
Let $F$ be a finite set of non-zero elements of $A$, and let $f,h\in A$, we say that $f$ reduces to
$h$ by $F$ in one step, denoted $f\xrightarrow{\,\, F\,\, } h$, if there exist elements
$f_1,\dots,f_t\in F$ and $r_1,\dots,r_t\in R$ such that
\begin{enumerate}
\item[\rm (i)]$lm(f_i)|lm(f)$, $1\leq i\leq t$, i.e., there exists
$x^{\alpha_i}\in Mon(A)$ such that $lm(f)=lm(x^{\alpha_i}lm(f_i))$, i.e.,
$\alpha_i+\exp(lm(f_i))=\exp(lm(f))$.
\item[\rm
(ii)]$lc(f)=r_1\sigma^{\alpha_1}(lc(f_1))c_{\alpha_1,f_1}+\cdots+r_t\sigma^{\alpha_t}(lc(f_t))c_{\alpha_t,f_t}$,
where $c_{\alpha_i,f_i}$ are defined as in Theorem \ref{coefficientes}, i.e.,
$c_{\alpha_i,f_i}:=c_{\alpha_i,\exp(lm(f_i))}$.
\item[\rm (iii)]$h=f-\sum_{i=1}^tr_ix^{\alpha_i}f_i$.
\end{enumerate}
We say that $f$ reduces to $h$ by $F$, denoted $f\xrightarrow{\,\, F\,\, }_{+}h$, if there exist
$h_1,\dots ,h_{t-1}\in A$ such that
\begin{center}
$\begin{CD} f @>{F}>> h_1 @>{F}>> h_2 @>{F}>>\cdots @>{F}>>h_{t-1} @>{F}>>h.
\end{CD}$
\end{center}
$f$ is reduced {\rm(}also called minimal{\rm)} w.r.t.. $F$ if $f = 0$ or there is no one step
reduction  of $f$ by $F$, i.e., one of the first two conditions of Definition
\ref{reductionsigmapbw} fails. Otherwise, we will say that $f$ is reducible w.r.t. $F$. If
$f\xrightarrow{\,\, F\,\, }_{+} h$ and $h$ is reduced w.r.t. $F$, then we say that $h$ is a
remainder for $f$ w.r.t. $F$.
\end{definition}
\begin{remark}
(i) By Theorem \ref{coefficientes}, the coefficients $c_{\alpha_i,f_i}$ in the previous definition
are unique and satisfy
\begin{center}
$x^{\alpha_i}lm(f_i)=c_{\alpha_i,f_i}x^{\alpha_i+\exp(lm(f_i))}+p_{\alpha_i,f_i}$,
\end{center}
where $p_{\alpha_i,f_i}=0$ or $\deg(p_{\alpha_i,f_i})<|\alpha_i+\exp(lm(f_i))|$, $1\leq i\leq t$.

(ii) $lm(f)\succ lm(h)$ and $f-h\in \langle F\}$, where $\langle F\}$ is the left ideal of $A$
generated by $F$.

(iii) The remainder of $f$ is not unique.

(iv) By definition we will assume that $0\xrightarrow {F} 0$.
\end{remark}
From the reduction relation we get the following interesting properties.
\begin{proposition}\label{xtetaf}
Let $A$ be a skew $PBW$ extension such that $c_{\alpha,\beta}$ is invertible for each $\alpha,\beta
\in \mathbb{N}^n$. Let $f,h\in A$, $\theta \in \mathbb{N}^n$ and $F=\{f_1,\dots ,f_t\}$ be a finite
set of non-zero polynomials of $A$. Then,
\begin{enumerate}
\item[\rm{(i)}]If $f\xrightarrow{\,\, F\,\, } h$, then there exists $p\in
A$ with $p=0$ or $lm(x^{\theta}f)\succ lm(p)$ such that $x^{\theta}f+p\xrightarrow{\,\, F\,\, }
x^{\theta}h$. In particular, if $A$ is quasi-commutative, then $p=0$.
\item[\rm{(ii)}]If $f\xrightarrow{\,\, F\,\, }_+
h$ and $p\in A$ is such that $p=0$ or $lm(h)\succ lm(p)$, then $f+p\xrightarrow{\,\, F\,\, }_+
h+p$.
\item[\rm{(iii)}]If $f\xrightarrow{\,\, F\,\, }_+ h$, then there
exists $p\in A$ with $p=0$ or $lm(x^{\theta}f)\succ lm(p)$ such that
$x^{\theta}f+p\xrightarrow{\,\, F\,\, }_+ x^{\theta}h$. If $A$ is quasi-commutative, then $p=0$.
\item[\rm{(iv)}]If $f\xrightarrow{\,\, F\,\, }_+ 0$, then there
exists $p\in A$ with $p=0$ or $lm(x^{\theta}f)\succ lm(p)$ such that
$x^{\theta}f+p\xrightarrow{\,\, F\,\, }_+ 0$. If $A$ is quasi-commutative, then $p=0$.
\end{enumerate}
\end{proposition}
\begin{proof}
See Proposition 20 in \cite{Gallego2}.
\end{proof}

The next theorem is the theoretical support of the division algorithm for skew $PBW$ extensions.
\begin{theorem}\label{algdivforPBW}
Let $F=\{f_1,\dots ,f_t\}$ be a finite set of non-zero polynomials of $A$ and $f\in A$, then the
division algorithm below produces polynomials $q_1,\dots ,q_t,h\in A$, with $h$ reduced w.r.t. $F$,
such that $f\xrightarrow{\,\, F\,\, }_{+} h$ and
\begin{center}
$f=q_1f_1+\cdots +q_tf_t+h$,
\end{center}
with
\begin{center}
$lm(f)=\max\{lm(lm(q_1)lm(f_1)),\dots ,lm(lm(q_t)lm(f_t)),lm(h)\}$.
\end{center}

\begin{center}
\fbox{\parbox[c]{11cm}{
\begin{center}
{\rm \textbf{Division algorithm in}} $A$
\end{center}
\begin{description}
\item[]{\rm \textbf{INPUT}:} $f,f_1,\dots ,f_t\in A \, \, \text{with}\,\, f_j\neq 0\, (1\leq j\leq t)$
\item[]{\rm \textbf{OUTPUT}:} $q_1,\dots ,q_t,h\in A\,\,\text{with}\,\, f=q_1f_1+\cdots
+q_tf_t+h$, $h$ reduced w.r.t. $\{f_1,\dots ,f_t\}$ and\\
$lm(f)=\max\{lm(lm(q_1)lm(f_1)),\dots ,lm(lm(q_t)lm(f_t)),lm(h)\}$
\item[]{\rm \textbf{INITIALIZATION}:} $q_1:=0,q_2:=0,\dots ,q_t:=0,h:=f$
\item[]{\rm \textbf{WHILE}} $h\neq 0$
and there exists $j$ such that $lm(f_j)$ divides $lm(h)$ \textbf{\emph{DO}}
\begin{quote}Calculate $J:=\{j\,|\,lm(f_j)$ divides $lm(h)\}$

\smallskip
{\rm \textbf{FOR}} $j\in J$ {\rm \textbf{DO}}

\smallskip
\begin{quote}
Calculate $\alpha_j\in \mathbb{N}^n$ such that $\alpha_j+\exp(lm(f_j))=\exp(lm(h))$
\end{quote}

\smallskip
{\rm \textbf{IF}} the equation $lc(h)=\sum_{j\in J}r_j\sigma^{\alpha_j}(lc(f_j))c_{\alpha_j,f_j}$
is soluble, where $c_{\alpha_j,f_j}$ are defined as in the Theorem \ref{coefficientes} {\rm
\textbf{THEN}}
\begin{quote}Calculate one solution $(r_j)_{j\in J}$

\smallskip
$h:=h-\sum_{j\in J}r_jx^{\alpha_j}f_j$

\smallskip
{\rm \textbf{FOR}} $j\in J$ {\rm \textbf{DO}}
\begin{quote}$q_j:=q_j+r_jx^{\alpha_j}$\end{quote}
\end{quote}
\emph{\textbf{ELSE}}
\begin{quote}
Stop
\end{quote}
\end{quote}
\end{description}}}
\end{center}
\end{theorem}
\begin{proof}
See \cite{Gallego2}, Theorem 21.
\end{proof}

The following example illustrates the above procedure.
\begin{example}
We consider the \textit{diffusion algebra} $\mathcal{A}$ in Example \ref{1.1.18} with  $n=2$,
$K=\mathbb{Q}$, $c_{12}=-2$ and $c_{21}=-1$. In this bijective skew $PBW$ extension,
$D_2D_1=2D_1D_2+x_2D_1-x_1D_2$ and the automorphisms $\sigma_1$ and $\sigma_2$ are the identity. We
consider the deglex order with $D_1\succ D_2$ and the polynomials $f_1:=x_1x_2D_1D_2$,
$f_2:=x_2D_1$, $f_3=x_1D_2$, $f=x_1x_2^2D_1^2D_2+x_1^2x_2D_2$ in $\mathcal{A}$. We want to divide
$f$ by the polynomials $f_1$, $f_2$ and $f_3$.

\textit{Step 1}. We start with $h:=f$, $q_1:=0$, $q_2:=0$, $q_3:=0$. Since  $lm(f_j)\mid lm(f)$ for
$j=1,2,3$, we compute $\alpha_j=(\alpha_{j1},\alpha_{j2})\in \mathbb{N}^{2}$ such that
$\alpha_j+\exp(lm(f_j))=\exp(lm(h))$ and the corresponding value of
$\sigma^{\alpha_j}(lc(f_j))c_{\alpha_j,\beta_j}$, where \linebreak $\beta_{j}=\exp(lm(f_j))$:
\begin{align*}
&(\alpha_{11},\alpha_{12})+(1,1)=(2,1)\Rightarrow
\alpha_{11}=1,\alpha_{12}=0,\\
&\sigma^{\alpha_1}(lc(f_1))c_{\alpha_1,\beta_1}=x_1x_2,\\
&(\alpha_{21},\alpha_{22})+(1,1)=(2,1)\Rightarrow
\alpha_{21}=1,\alpha_{22}=1,\\
&\sigma^{\alpha_1}(lc(f_2))c_{\alpha_2,\beta_2}=2x_2,\\
&(\alpha_{31},\alpha_{32})+(1,1)=(2,1)\Rightarrow
\alpha_{31}=2,\alpha_{32}=0,\\
&\sigma^{\alpha_1}(lc(f_3))c_{\alpha_3,\beta_3}=x_1.
\end{align*}
Now, we solve the equation
\begin{center}
$lc(h)=x_1x_2^2=r_1(x_1x_2)+r^2(2x_2)+r_3(x_1)\Rightarrow r_1=3x_2,\,r_2=-\frac{1}{2}x_1x_2,\,
r_3=-x_2^2,$
\end{center}
and with the relations defining $\mathcal{A}$, we compute
\begin{align*}
h=&h-(r_1x^{\alpha_1}f_1+r_2x^{\alpha_2}f_2+r_3x^{\alpha_3}f_3)\\
=&h-3x_1x_2^2D_1^2D_2+\frac{1}{2}x_1x_2^2(2D_1^2D_2+x_2D_1^2-x_1D_1D_2)\\
=&\frac{1}{2}x_1x_2^3D_1^2-\frac{1}{2}x_1^2x_2^2D_1D_2+x_1^2x_2D_2.
\end{align*}
We compute also
\begin{center}
$q_1:=3x_2D_1$, $q_2:=-\frac{1}{2}x_1x_2D_1D_2$, $q_3:=-x_2^2D_1^2$.
\end{center}

\textit{Step 2.} $lm(h)=D_1^2$, $lc(h)=\frac{1}{2}x_1x_2^3$. In this case, $lm(f_j)\mid lm(f)$ only
for $j=2$ and we have that $\alpha_2=(\alpha_{21},\alpha_{22})\in \mathbb{N}^{3}$ such that
$\alpha_j+\exp(lm(f_j))=\exp(lm(h))$ is $\alpha=(1,0)$; moreover,
$\sigma^{\alpha}(lc(f_2))c_{\alpha,\beta}=x_2$ and $r=\frac{1}{2}x_1x_2^2$ is such that
$lc(h)=rx_2$. Thus we have:
\begin{align*}
h=&h-rx^{\alpha_2}f_2\\
=&-\frac{1}{2}x_1^2x_2^2D_1D_2+x_1^2x_2D_2.
\end{align*}
and
\begin{center}
$q_1:=3x_2D_1$, $q_2:=-\frac{1}{2}x_1x_2D_1D_2+\frac{1}{2}x_1x_2^2D_1$, $q_3:=-x_2^2D_1^2$.
\end{center}

\textit{Step 3.} Note that $lm(h)=D_1D_2$ and $lm(f_j)\mid lm(h)$ for $j=1,2,3$. In this case we
have:
\begin{align*}
&(\alpha_{11},\alpha_{12})+(1,1)=(1,1)\Rightarrow
\alpha_{11}=0,\alpha_{12}=0,\\
&\sigma^{\alpha_1}(lc(f_1))c_{\alpha_1,\beta_1}=x_1x_2,\\
&(\alpha_{21},\alpha_{22})+(1,0)=(1,1)\Rightarrow
\alpha_{21}=0,\alpha_{22}=1,\\
&\sigma^{\alpha_2}(lc(f_2))c_{\alpha_2,\beta_2}=2x_2,\\
&(\alpha_{31},\alpha_{32})+(0,1)=(1,1,1)\Rightarrow
\alpha_{31}=1,\alpha_{32}=0,\\
&\sigma^{\alpha_3}(lc(f_3))c_{\alpha_3,\beta_3}=x_1.
\end{align*}
We solve
\begin{center}
$-\frac{1}{2}x_1^2x_2^2=r_1x_1x_2+r_2(2x_2)+r_3x_1\Rightarrow r_1=3x_1x_2,\,r_2=-x_1^2x_2,\,
r_3=-\frac{3}{2}x_1x_2^2;$
\end{center}
thus
\begin{align*}
h=&h-(r_1x^{\alpha_1}f_1+r_2x^{\alpha_2}f_2+r_3x^{\alpha_3}f_3)\\
=&h-(3x_1^2x_2^2D_1D_2-x_1^2x_2^2(2D_1D_2+x_2D_1-x_1D_2)-\frac{3}{2}x_1^2D_1D_2)\\
=&x_1^2x_2^3D_1+(x_1^2x_2-x_1^3x_2^2)D_2
\end{align*}
and also
\begin{center}
$q_1:=3x_2D_1-3x_1x_2$, $q_2:=-\frac{1}{2}x_1x_2D_1D_2+\frac{1}{2}x_1x_2^2D_1-x_1^2x_2D_2$,
$q_3:=-x_2^2D_1^2-\frac{3}{2}x_1x_2^2D_1$.
\end{center}

\textit{Step 4}. Finally, note that
$h=x_1^2x_2^3D_1+(x_1^2x_2-x_1^3x_2^2)D_2=x_1^2x_2^2f_1+(x_1x_2-x_1^2x_2^2)f_3$, thus
\begin{center}
$f=q_1f_1+q_2f_2+q_3f_3$
\end{center}
where
\begin{center}
$q_1:=3x_2D_1-3x_1x_2$,
$q_2:=-\frac{1}{2}x_1x_2D_1D_2+\frac{1}{2}x_1x_2^2D_1-x_1^2x_2D_2+x_1^2x_2^2$,
$q_3:=-x_2^2D_1^2-\frac{3}{2}x_1x_2^2D_1+x_1x_2-x_1^2x_2^2$.
\end{center}
Moreover,
\begin{center}
$\max\{lm(lm(q_1)lm(f_1)), lm(lm(q_2)lm(f_2)), lm(lm(q_3)lm(f_3))\}$\\
$=\max\{D_1^2D_2, D_1^2D_2,D_1^2D_2\}=lm(f)$.
\end{center}

\end{example}

\subsection{Gröbner bases of left ideals}

Our next purpose is to recall the definition of a Gröbner bases for the left ideals of the skew
$PBW$ extension $A=\sigma(R)\langle x_1,\dots,x_n\rangle$.
\begin{definition}
Let $I\neq 0$ be a left ideal of $A$ and let $G$ be a non empty finite subset of non-zero
polynomials of $I$, we say that $G$ is a Gröbner basis for $I$ if each element $0\neq f\in I$ is
reducible w.r.t. $G$.
\end{definition}
We will say that $\{0\}$ is a Gröbner basis for $I=0$.
\begin{theorem}\label{teogrobnersigmapbw}
Let $I\neq 0$ be a left ideal of $A$ and let $G$ be a finite subset of non-zero polynomials of $I$.
Then the following conditions are equivalent:
\begin{enumerate}
\item[\rm(i)]$G$ is a Gröbner basis for $I$.
\item[\rm(ii)]For any polynomial $f\in A$,
\begin{center}
$f\in I$ if and only if $f\xrightarrow{\,\, G\,\, }_{+} 0$.
\end{center}
\item[\rm(iii)]For any $0\neq f\in I$ there exist $g_1,\dots ,g_t\in
G$ such that $lm(g_j)|lm(f)$, $1\leq j\leq t$, {\rm(}i.e., there exist $\alpha_j\in \mathbb{N}^n$
such that $\alpha_j+\exp(lm(g_j))=\exp(lm(f))${\rm)} and
\begin{center}
$lc(f)\in \langle \sigma^{\alpha_1}(lc(g_1))c_{\alpha_1,g_1},\dots
,\sigma^{\alpha_t}(lc(g_t))c_{\alpha_t,g_t}\}$.
\end{center}
\item[\rm(iv)]For $\alpha \in \mathbb{N}^n$, let $\langle \alpha ,I\}$ be the
left ideal of $R$ defined by
\begin{center}
$\langle \alpha ,I\}:=\langle lc(f)|f\in I,\exp(lm(f))=\alpha\}$.
\end{center}
Then, $\langle \alpha ,I\}=J$, with
\begin{center}
$J:=\langle \sigma^{\beta}(lc(g))c_{\beta,g}|g\in G, \ \text{with} \ \beta + \exp(lm(g))=\alpha\}$.
\end{center}
\end{enumerate}
\end{theorem}
\begin{proof}
See Theorem 24 in \cite{Gallego2}.
\end{proof}
From this theorem we get the following consequences.
\begin{corollary}\label{153}
Let $I\neq 0$ be a left ideal of $A$. Then,
\begin{enumerate}
\item[\rm(i)]If $G$ is a Gröbner basis for $I$, then $I=\langle G\}$.
\item[\rm(ii)]Let $G$ be a Gröbner basis for $I$, if $f\in I$ and
$f\xrightarrow{\,\, G\,\, }_{+} h$, with $h$ reduced, then $h=0$.
\item[\rm(iii)]Let $G=\{g_1,\dots,g_t\}$ be a set of non-zero polynomials of $I$ with $lc(g_i)\in R^{*}$ for each $1\leq i\leq t$. Then,
$G$ is a Gröbner basis of $I$ if and only if given $0\neq r\in I$ there exists $i$ such that
$lm(g_i)$ divides $lm(r)$.
\end{enumerate}
\end{corollary}
\begin{proof}
(i) This is a direct consequence of Theorem \ref{teogrobnersigmapbw}.

(ii) Let $f\in I$ and $f\xrightarrow{\,\, G\,\, }_{+} h$, with $h$ reduced; since $f-h\in \langle
G\}=I$, then $h\in I$; if $h\neq 0$ then $h$ can be reduced by $G$, but this is not possible since
$h$ is reduced.

(iii) If $G$ is a Gröbner basis of $I$, then given $0\neq r\in I$, $r$ is reducible w.r.t. $G$,
hence there exists $i$ such that $lm(g_i)$ divides $lm(r)$. Conversely, if this condition holds for
some $i$, then $r$ is reducible w.r.t. $G$ since the equation
$lc(r)=r_1\sigma^{\alpha_i}(lc(g_i)c_{\alpha_i,g_i}$, with $\alpha_i+\exp(lm(g_i))=\exp(lm(r))$, is
soluble with solution $r_1=lc(r)c_{\alpha_i,g_i}'(\sigma^{\alpha_i}(lc(g_i)))^{-1}$, where
$c_{\alpha_i,g_i}'$ is a left inverse of $c_{\alpha_i,g_i}$.
\end{proof}

\subsection{Buchberger's algorithm for left ideals}\label{5.4}

In \cite{Gallego2} was constructed the Buchberger's algorithm for computing Gröbner bases of left
ideals for the particular case of quasi-commutative bijective skew $PBW$ extensions. In this
subsection we extend the Buchberger's procedure to the general case of bijective skew $PBW$
extensions without assuming that they are quasi-commutative. Complementing Remark \ref{LGS},
\textbf{from now on we will assume that $A=\sigma(R)\langle x_1,\dots,x_n\rangle$ is bijective}.

We start fixing some notation and proving a preliminary key lemma.

\begin{definition}\label{BF}
Let $F := \{g_1,\dots,g_s\}\subseteq A$, $X_F$ the least common multiple of $\{lm(g_1),\dots
,lm(g_s)\}$, $\theta\in \mathbb{N}^n$, $\beta_i:=\exp(lm(g_i))$ and $\gamma_i\in \mathbb{N}^n$ such
that $\gamma_i+\beta_i=\exp(X_F)$, $1\leq i\leq s$. $B_{F,\theta}$ will denote a finite set of
generators of
\begin{center}
$S_{F,\theta}:=Syz_R[\sigma^{\gamma_1+\theta}(lc(g_1))c_{\gamma_1+\theta,\beta_1} \ \cdots \
\sigma^{\gamma_s+\theta}(lc(g_s))c_{\gamma_s+\theta,\beta_s})]$.
\end{center}
For $\theta=\textbf{\emph{0}}:=(0,\dots,0)$, $S_{F,\theta}$ will be denoted by $S_F$ and
$B_{F,\theta}$ by $B_F$.
\end{definition}
\begin{remark}\label{RemarkSyzygy}
Let $(b_1,\ldots,b_{s})\in S_{F,\theta}$. Since $A$ is bijective, then there exists an unique
$(b_1',\ldots,b_s')\in S_{F}$ such that
 $b_i=\sigma^{\theta}(b_{i}')c_{\theta,\gamma_i}$ for $1\leq i\leq s$: in fact,
the existence and uniqueness of $(b_1',\ldots,b_s')$ it follows of the bijectivity of $A$. Now,
since $(b_1,\ldots,b_{s}) \in S_{F,\theta}$, then $\sum_{i=1}^sb_i\sigma^{\theta+\gamma_i}
(lc(g_i))c_{\theta+\gamma_i,\beta_i}$ $=0$. Replacing $b_i$ by
$\sigma^{\theta}(b_{i}')c_{\theta,\gamma_i}$ in the last equation, we obtain
\begin{center}
$\sum_{i=1}^s\sigma^{\theta}(b_i')c_{\theta,\gamma_i}\sigma^{\theta+\gamma_i}(lc(g_i))c_{\theta,\gamma_i}^{-1}c_{\theta,\gamma_i}
c_{\theta+\gamma_i,\beta_i}=0$;
\end{center}
multiplying by $c_{\theta,\gamma_i+\beta_i}^{-1}$ we get
\begin{center}
$\sum_{i=1}^s\sigma^{\theta}(b_i')c_{\theta,\gamma_i}\sigma^{\theta+\gamma_i}(lc(g_i))c_{\theta,\gamma_i}^{-1}c_{\theta,\gamma_i}
c_{\theta+\gamma_i,\beta_i}c_{\theta,\gamma_i+\beta_i}^{-1}=0$;
\end{center}
now we can use the identities of
Remark \ref{identities}, so
\begin{center}
$\sum_{i=1}^s\sigma^{\theta}(b_i')\sigma^{\theta}(\sigma^{\gamma_i}(lc(g_i)))\sigma^{\theta}(c_{\gamma_i,\beta_i})
=0$,
\end{center}
and since $\sigma^{\theta}$ is injective then
$\sum_{i=1}^sb_i'\sigma^{\gamma_i}(lc(g_i))c_{\gamma_i,\beta_i} =0$, i.e., $(b_1',\ldots,b_s')\in
S_F$.
\end{remark}
\begin{lemma}\label{Sumofsicigies}
Let $g_1,\ldots, g_s \in A$ , $c_1,\ldots, c_s \in R-\{0\}$ and $\alpha_1,\ldots,\alpha_{s}\in
\mathbb{N}^{n}$ such that $\alpha_{1}+\exp(g_1)=\cdots=\alpha_{s}+\exp(g_s):=\delta$. If
$lm(\sum_{i=1}^{s}c_{i}x^{\alpha_{i}}g_{i})\prec x^{\delta}$, then there exist
$r_{1},\ldots,r_{k}\in R$ and $l_{1},\ldots,l_{s}\in A$ such that
\[\sum_{i=1}^{s}c_{i}x^{\alpha_{i}}g_{i}=\sum_{j=1}^{k}r_{j}x^{\delta-\exp(X_{F})}\biggl(\sum_{i=1}^{s}b_{ji}x^{\gamma_{i}}g_{i}\biggr)+
\sum_{i=1}^{s}l_{i}g_{i},\] where $X_{F}$ is the least common multiple of $lm(g_{1}),\ldots,
lm(g_{s})$, $\gamma_{i}\in \mathbb{N}^{n}$ is such that $\gamma_{i}+\exp(g_{i}) =\exp(X_{F})$,
$1\leq i\leq s$, and
\begin{center}
$B_{F}:=\{\textbf{b}_1,\dots,\textbf{b}_k\}:=\{(b_{11},\dots,b_{1s}),\dots,
(b_{k1},\dots,b_{ks})\}$.
\end{center}
Moreover, $lm(x^{\delta-\exp(X_{F})}\sum_{i=1}^{s} b_{ji}x^{\gamma_{i}}g_{i}) \prec x^{\delta}$ for
every $1\leq j\leq k$, and \linebreak $lm(l_{i}g_{i})\prec x^{\delta}$ for every $1\leq i\leq s$.
\begin{proof}
Let $x^{\beta_i}:=lm(g_{i})$ for $1\leq i\leq s$; since $x^{\delta}=lm(x^\alpha_i lm(g_i))$, then
$lm(g_i)\mid x^{\delta}$ and hence $X_F\mid x^{\delta}$, so there exists $\theta \in \mathbb{N}^n$
such that $\exp(X_F)+\theta=\delta$. On the other hand, $\gamma_i+\beta_i=\exp(X_F)$ and
$\alpha_i+\beta_i=\delta$, so $\alpha_i=\gamma_i+\theta$ for every $1\leq i\leq s$. Now,
$lm(\sum_{i=1}^{s}c_{i}x^{\alpha_{i}}g_{i})\prec x^{\delta}$ implies that
$\sum_{i=1}^{s}c_{i}\sigma^{\alpha_{i}}(lc(g_{i}))c_{\alpha_{i},\beta_{i}}=0$. So we have
$\sum_{i=1}^{s}c_{i}\sigma^{\theta+\gamma_{i}}(lc(g_{i}))c_{\theta+\gamma_{i},\beta_{i}}=0$. This
implies that $(c_{1}, \ldots, c_{s})\in S_{F,\theta}$;
 from Remark \ref{RemarkSyzygy} we know that exists an unique $(c'_1,\ldots,c'_s)\in S_{F}$ such that $c_{i}=\sigma^{\theta}(c'_{i})
 c_{\theta,\gamma_{i}}$. Then,
 \begin{center}
$\sum_{i=1}^{s}c_{i}x^{\alpha_{i}}g_{i}=\sum_{i=1}^{s}\sigma^{\theta}(c'_{i})c_{\theta,\gamma_{i}}x^{\alpha_{i}}g_{i}.$
\end{center}
Now,
\begin{center}
$x^{\theta}c'_{i}x^{\gamma_{i}}=(\sigma^{\theta}(c'_{i})x^{\theta}+p_{c'_{i},\theta})x^{\gamma_{i}}
=\sigma^{\theta}(c'_{i})x^{\theta}x^{\gamma_{i}}+p_{c'_{i},\theta}x^{\gamma_{i}}
=\sigma^{\theta}(c'_{i})c_{\theta,\gamma_{i}}x^{\theta+\gamma_{i}}+\sigma^{\theta}(c'_{i})p_{\theta,\gamma_{i}}+
p_{c'_{i},\theta}x^{\gamma_{i}}
=\sigma^{\theta}(c'_{i})c_{\theta,\gamma_{i}}x^{\theta+\gamma_{i}}+p'_{i}$
\end{center}
where $p'_{i}:= \sigma^{\theta}(c'_{i})p_{\theta,\gamma_{i}}+p_{c'_{i},\theta}x^{\gamma_{i}}$; note
that $p'_{i}=0$ or $lm(p'_{i})\prec x^{\theta+\gamma_{i}}$ for each $i$. Thus,
$\sigma^{\theta}(c'_{i})c_{\theta,\gamma_{i}}x^{\theta+\gamma_{i}}=x^{\theta}c'_{i}x^{\gamma_{i}}+p_{i}$,
with $p_{i}=0$ or $lm(p_i)\prec x^{\theta+\gamma_{i}}$. Hence,
\begin{center}
$\sum_{i=1}^{s}c_{i}x^{\alpha_{i}}g_{i}=\sum_{i=1}^{s}\sigma^{\theta}(c'_{i})c_{\theta,\gamma_{i}}x^{\alpha_{i}}g_{i}
=\sum_{i=1}^{s} (x^{\theta}c'_{i}x^{\gamma_{i}}+p_{i})g_{i} =\sum_{i=1}^{s}
x^{\theta}c'_{i}x^{\gamma_{i}}g_{i}+\sum_{i=1}^{s}p_{i}g_{i},$
\end{center}
with $p_{i}g_{i}=0$ or $lm(p_{i}g_{i})\prec x^{\theta+\gamma_{i}+\beta_{i}}=x^{\delta}$. On the
other hand, since \linebreak $(c'_1,\ldots,c'_s)\in S_{F}$, then there exist
$r'_{1},\ldots,r'_{k}\in R$ such that
$(c'_1,\ldots,c'_s)=r'_{1}\textbf{\emph{b}}_{1}+\cdots+r'_{k}\textbf{\emph{b}}_{k}=r'_1(b_{11},\ldots,b_{1s})+
\cdots+r'_k(b_{k1},\ldots,b_{ks})$, thus $c'_{i}=\sum_{j=1}^{k}r'_{j}b_{ji}$. Using this, we have

\begin{align*}
\sum_{i=1}^{s} x^{\theta}c'_{i}x^{\gamma_{i}}g_{i}&=\sum_{i=1}^{s} x^{\theta}\bigl(\sum_{j=1}^{k}r'_{j}b_{ji}\bigr)x^{\gamma_{i}}g_{i}\\
&=\sum_{i=1}^{s}\bigl(\sum_{j=1}^{k}x^{\theta}r'_{j}b_{ji}\bigr)x^{\gamma_{i}}g_{i}\\
&=\sum_{i=1}^{s}\bigl(\sum_{j=1}^{k}(\sigma^{\theta}(r'_{j})x^{\theta}+p_{r'_{j},\theta})b_{ji}\bigr)x^{\gamma_{i}}g_{i}\\
&=\sum_{i=1}^{s}\bigl(\sum_{j=1}^{k}\sigma^{\theta}(r'_{j})x^{\theta}b_{ji}x^{\gamma_{i}}g_{i}+\sum_{j=1}^{k}p_{r'_{j},\theta}
b_{ji}x^{\gamma_{i}}g_{i}\bigr)\\
&=\sum_{j=1}^{k}\sum_{i=1}^{s}\sigma^{\theta}(r'_{j})x^{\theta}b_{ji}x^{\gamma_{i}}g_{i}+
\sum_{i=1}^{s}\sum_{j=1}^{k}p_{r'_{j},\theta}b_{ji}x^{\gamma_{i}}g_{i}\\
&=\sum_{j=1}^{k}\sigma^{\theta}(r'_{j})x^{\theta}\sum_{i=1}^{s}b_{ji}x^{\gamma_{i}}g_{i}+\sum_{i=1}^{s}q_{i}g_{i},
\end{align*}
where $q_{i}:=\sum_{j=1}^{k}p_{r'_{j},\theta}b_{ji}x^{\gamma_{i}}=0$ or $lm(q_{i})\prec
x^{\theta+\gamma_{i}}$. Therefore,
\begin{align*}
\sum_{i=1}^{s}c_{i}x^{\alpha_{i}}g_{i}&=\sum_{j=1}^{k}r_{j}x^{\theta}\sum_{i=1}^{s}b_{ji}x^{\gamma_{i}}g_{i}+\sum_{i=1}^{s}l_ig_{i},
\end{align*}
with $l_{i}:=p_{i}+q_{i}$ for $1\leq i\leq s$ and $r_{j}:=\sigma^{\theta}(r'_{j})$ for $1\leq j\leq
k$. Finally, it is easy to see that
$lm(x^{\theta}\bigl(\sum_{i=1}^{s}b_{ji}x^{\gamma_{i}}g_{i}))\prec x^{\delta}$ since that
$lm(\sum_{i=1}^{s}b_{ji}x^{\gamma_{i}}g_{i}) \prec x^{\gamma_{i}+\beta_{i}}$, and
$lm(l_ig_i)=lm(p_ig_i+q_ig_i)\prec x^{\delta}$.
\end{proof}
\end{lemma}
With the notation of Definition \ref{BF} and Lemma \ref{Sumofsicigies}, we can prove the main
result of the present section.
\begin{theorem}\label{1.5.19}
Let $I\neq 0$ be a left ideal of $A$ and let $G$ be a finite subset of non-zero generators of $I$.
Then the following conditions are equivalent:
\begin{enumerate}
\item[\rm (i)]$G$ is a Gröbner basis of $I$.
\item[\rm (ii)]For all $F:= \{g_1,\dots,g_s\}\subseteq G$, and for any $(b_1,\dots,b_s)\in B_{F}$,
\begin{center}
$\sum_{i=1}^sb_ix^{\gamma_i}g_i\xrightarrow{\,\, G\,\, }_+ 0$.
\end{center}
\end{enumerate}
\begin{proof}[Proof.]
(i) $\Rightarrow$ (ii): We observe that $f:=\sum_{i=1}^sb_ix^{\gamma_i}g_i\in I$, so by Theorem
\ref{teogrobnersigmapbw} $f\xrightarrow{\,\, G\,\, }_+ 0$.

(ii) $\Rightarrow$ (i): Let $0\neq f\in I$, we will prove that the condition (iii) of Theorem
\ref{teogrobnersigmapbw} holds. Let $G:=\{g_1,\dots,g_t\}$, then there exist $h_1,\dots,h_t\in A$
such that $f=h_1g_1+\cdots+h_tg_t$ and we can choose $\{h_i\}_{i=1}^t$ such that
\begin{center}
$x^{\delta}:=\max\{lm(lm(h_i)lm(g_i))\}_{i=1}^t$
\end{center}
is minimal. Let $lm(h_i):=x^{\alpha_i}$,
$c_{i}:=lc(h_{i})$, $lm(g_i)=x^{\beta_i}$ for $1\leq i\leq t$ and $F:=\{g_i\in G\mid
lm(lm(h_i)lm(g_i))=x^{\delta}\}$; renumbering the elements of $G$ we can assume that
$F=\{g_1,\dots,g_s\}$. We will consider two possible cases.

\textit{Case 1}: $lm(f)=x^{\delta}$. Then $lm(g_i)\mid lm(f)$ for $1\leq i\leq s$ and
\begin{center}
$lc(f)=c_{1}\sigma^{\alpha_1}(lc(g_1))c_{\alpha_1,\beta_1}+\cdots+c_s\sigma^{\alpha_s}(lc(g_s))c_{\alpha_s,\beta_s}$,
\end{center}
i.e., the condition (iii) of Theorem \ref{teogrobnersigmapbw} holds.

\textit{Case 2}: $lm(f)\prec x^{\delta}$. We will prove that this produces a contradiction. To
begin, note that $f$ can be written as
\begin{align}\label{writingf}
f=\sum_{i=1}^{s}c_{i}x^{\alpha_{i}}g_{i}+\sum_{i=1}^{s}(h_{i}-c_{i}x^{\alpha_{i}})g_{i}+\sum_{i=s+1}^{t}h_{i}g_{i};
\end{align}
we have $lm((h_{i}-c_{i}x^{\alpha_{i}})g_{i})\prec x^{\delta}$ for every $1\leq i\leq s$ and
$lm(h_{i}g_{i})\prec x^{\delta}$ for every $s+1\leq i\leq t$, so
\begin{center}
$lm(\sum_{i=1}^{s}(h_{i}-c_{i}x^{\alpha_{i}})g_{i})\prec x^{\delta}$ and
$lm(\sum_{i=s+1}^{t}h_{i}g_{i})\prec x^{\delta}$,
\end{center}
and hence $lm(\sum_{i=1}^{s}c_{i}x^{\alpha_{i}}g_{i})\prec x^{\delta}$. By Lemma
\ref{Sumofsicigies} (and its notation), we have
\begin{align}
\sum_{i=1}^{s}c_{i}x^{\alpha_{i}}g_{i}=\sum_{j=1}^{k}r_{j}x^{\delta-\exp(X_{F})}\bigl(\sum_{i=1}^{s}b_{ji}x^{\gamma_{i}}g_{i}\bigr)+
\sum_{i=1}^{s}l_{i}g_{i},
\end{align}
where $lm(x^{\delta-\exp(X_{F})}\sum_{i=1}^{s}b_{ji}x^{\gamma_{i}}g_{i})\prec x^{\delta}$ for every
$1\leq j\leq k$ and $lm(l_{i}g_{i})\prec x^{\delta}$ for $1\leq i\leq s$. By the hypothesis,
$\sum_{i=1}^sb_{ji}x^{\gamma_i}g_i\xrightarrow{\,\,G\,\,}_{+}0$, whence, by Theorem
\ref{algdivforPBW}, there exist $q_1,\dots,q_t\in A$ such that
$\sum_{i=1}^sb_{ji}x^{\gamma_i}g_i=\sum_{i=1}^tq_ig_i$, with
\begin{center}
$lm(\sum_{i=1}^sb_{ji}x^{\gamma_i}g_i)=\max\{lm(lm(q_i)lm(g_i))\}_{i=1}^t$,
\end{center}
but $(b_{j1},\dots,b_{js})$ $\in B_{F}$, so $lm(\sum_{i=1}^sb_{ji}x^{\gamma_i}g_i)\prec X_{F}$ and
hence \linebreak $lm(lm(q_i)lm(g_i))\prec X_{F}$ for every $1\leq i\leq t$. Thus,
\begin{center}
$\sum_{j=1}^{k}r_{j}x^{\delta-\exp(X_{F})}\bigl(\sum_{i=1}^{s}b_{ji}x^{\gamma_{i}}g_{i}\bigr)=
\sum_{j=1}^{k}r_{j}x^{\delta-\exp(X_{F})}\bigl(\sum_{i=1}^{t}q_{i}g_{i}\bigr)=\sum_{i=1}^{t}\sum_{j=1}^{k}r_{j}x^{\delta-\exp(X_{F})}q_{i}g_{i}=
\sum_{i=1}^{t}\widetilde{q}_{i}g_{i},$
\end{center}
with $\widetilde{q}_{i}:=\sum_{j=1}^{k}r_{j}x^{\delta-\exp(X_{F})}q_{i}$ and
$lm(\widetilde{q}_{i}g_{i})\prec x^{\delta}$ for every $1\leq i\leq t$. Substituting
$\sum_{i=1}^{s}c_{i}x^{\alpha_{i}}g_{i}=\sum_{i=1}^{t}
\widetilde{q}_{i}g_{i}+\sum_{i=1}^{s}l_{i}g_{i}$ into equation (\ref{writingf}), we obtain
\[f=\sum_{i=1}^{t} \widetilde{q}_{i}g_{i}+\sum_{i=1}^{s}(h_{i}-c_{i}x^{\alpha_{i}})g_{i}+\sum_{i=1}^{s}l_{i}g_{i}+\sum_{i=s+1}^{t}h_{i}g_{i},\]
and so we have expressed $f$ as a combination of polynomials $g_{1},\ldots,g_{t}$, where every term
has leading monomial $\prec x^{\delta}$. This contradicts the minimality of $x^{\delta}$ and we
finish the proof.
\end{proof}
\end{theorem}
\begin{corollary}\label{algorithmforbijective}
Let $F=\{f_1,\dots ,f_s\}$ be a set of non-zero polynomials of $A$. The algorithm below produces a
Gröbner basis for the left ideal $\langle F\}$ of $A$ {\rm(}$P(X)$ denotes the set of subsets of
the set $X${\rm)}:
\begin{center}
\fbox{\parbox[c]{11cm}{
\begin{center}
{\rm \textbf{Buchberger's algorithm for \\ bijective skew $PBW$ extensions}}
\end{center}
\begin{description}
\item[]{\rm \textbf{INPUT}:} $F := \{f_1,\dots,f_s\}\subseteq A$,
$f_i\neq 0$, $1\leq i\leq s$
\item[]{\rm \textbf{OUTPUT}:} $G=\{g_1,\dots ,g_t\}$ a Gröbner basis for $\langle F\}$
\item[]{\rm \textbf{INITIALIZATION}:} $G:=\emptyset, G':=F$
\item[]{\rm \textbf{WHILE}} $G'\neq G$ {\rm \textbf{DO}}
\begin{quote}$D:=P(G')-P(G)$

\smallskip

$G:=G'$

\smallskip

{\rm \textbf{FOR}} each $S:=\{g_{i_1},\dots ,g_{i_k}\}\in D$ {\rm \textbf{DO}}

\smallskip

\begin{quote}
Compute $B_S$

\smallskip
{\rm \textbf{FOR}} each $\textbf{b}=(b_1,\dots ,b_k)\in B_S$ {\rm \textbf{DO}}
\begin{quote}Reduce $\sum_{j=1}^kb_jx^{\gamma_j}g_{i_j}\xrightarrow{\,\, G'\,\, }_+ r$,
with $r$ reduced with respect to $G'$ and $\gamma_j$ defined as in Definition \ref{BF}
\begin{quote}{\rm \textbf{IF}} $r\neq 0$ {\rm \textbf{THEN}}
\begin{quote}$G':=G'\cup \{r\}$
\end{quote}
\end{quote}
\end{quote}
\end{quote}
\end{quote}
\end{description}}}
\end{center}
\end{corollary}
From Theorem \ref{1.3.4} and the previous corollary we get the following direct conclusion.
\begin{corollary}\label{existence}
Each left ideal of $A$ has a Gröbner basis.
\end{corollary}

\subsection{Gröbner bases of modules}

\noindent In this subsection we present the general theory of Gröbner bases for submodules of
$A^m$, $m\geq 1$, where $A=\sigma(R)\langle x_1,\dots,x_n\rangle$ is a bijective skew $PBW$
extension of $R$, with $R$ a $LGS$ ring (see Definition \ref{LGSring}) and $Mon(A)$ endowed with
some monomial order (see Definition \ref{monomialorder}). $A^m$ is the left free $A$-module of
column vectors of length $m\geq 1$; since $A$ is a left Noetherian ring (Theorem \ref{1.3.4}), then
$A$ is an $IBN$ ring (Invariant Basis Number, see \cite{Lezama6}), and hence, all bases of the free
module $A^m$ have $m$ elements. Note moreover that $A^m$ is a left Noetherian, and hence, any
submodule of $A^m$ is finitely generated. This theory was studied in \cite{Jimenez} and
\cite{Jimenez2}, but now we will extend Buchberger's algorithm to the general bijective case
without assuming that $A$ is quasi-commutative. The results presented in this section are an easy
generalization of those of the previous sections, i.e., taking $m=1$ we get the theory of Gröbner
bases for the left ideals of $A$ developed before. We will omit the proofs since most of them can
be consulted in \cite{Jimenez} and \cite{Jimenez2} or they are an easy adaptation of those of the
previous sections. The theory presented in this section has been also studied by Gómez-Torrecillas
et al. (see \cite{Gomez-Torrecillas} , \cite{Gomez-Torrecillas2}) for left $PBW$ algebras over
division rings and assuming some special commutative conditions.

\subsubsection{Monomial orders on $Mon(A^m)$} \noindent In the rest of this section we will
represent the elements of $A^m$ as row vectors, if this not represent confusion. We recall that the
canonical basis of $A^m$ is
\begin{center}
$\textbf{\emph{e}}_1=(1,0,\dots ,0),\textbf{\emph{e}}_2=(0,1,0,\dots, 0),\dots
,\textbf{\emph{e}}_m=(0,0,\dots ,1)$.
\end{center}
\begin{definition}
A monomial in $A^m$ is a vector $\textbf{X}=X\textbf{e}_i$, where $X=x^{\alpha}\in Mon(A)$ and
$1\leq i\leq m$, i.e.,
\begin{center}
$\textbf{X}=X\textbf{e}_i=(0,\dots ,X,\dots ,0)$,
\end{center}
where $X$ is in the $i$-th position, named the index of $\textbf{X}$, $ind(\textbf{X}):=i$. A term
is a vector $c\textbf{X}$, where $c\in R$. The set of monomials of $A^m$ will be denoted by
$\textrm{Mon}(A^m)$. Let $\textbf{Y}=Y\textbf{e}_j\in Mon(A^m)$, we say that $\textbf{X}$ divides
$\textbf{Y}$ if $i=j$ and $X$ divides $Y$. We will say that any monomial $\textbf{X}\in Mon(A^m)$
divides the null vector $\textbf{\emph{0}}$. The least common multiple of $\textbf{X}$ and
$\textbf{Y}$, denoted by $lcm(\textbf{X},\textbf{Y})$, is $\textbf{\emph{0}}$ if $i\neq j$, and
$U\textbf{e}_i$, where $U=lcm(X,Y)$, if $i=j$. Finally, we define
$\exp(\textbf{X}):=\exp(X)=\alpha$ and $\deg(\textbf{X}):=\deg(X)=|\alpha|$.
\end{definition}
We now define monomials orders on $Mon(A^m)$.
\begin{definition}
A monomial order on $Mon(A^m)$ is a total order $\succeq$ satisfying the following three
conditions:
\begin{enumerate}
\item[\rm(i)] $lm(x^{\beta}x^{\alpha})\textbf{e}_{i}\succeq x^{\alpha}\textbf{e}_{i}$, for every
monomial $\textbf{X}=x^{\alpha}\textbf{e}_{i}\in Mon(A^{m})$ and any monomial $x^{\beta}$ in
$Mon(A)$.
\item[\rm(ii)] If $\textbf{Y}=x^{\beta}\textbf{e}_{j}\succeq \textbf{X}=x^{\alpha}\textbf{e}_{i}$, then
$lm(x^{\gamma}x^{\beta})\textbf{e}_{j}\succeq lm(x^{\gamma}x^{\alpha})\textbf{e}_{i}$ for every
monomial $x^{\gamma}\in Mon(A)$.
\item[\rm (iii)]$\succeq$ is degree compatible, i.e., $\deg(\textbf{X})\geq \deg(\textbf{Y})\Rightarrow \textbf{X}\succeq
\textbf{Y}$.
\end{enumerate}
If $\textbf{X}\succeq \textbf{Y}$ but $\textbf{X}\neq \textbf{Y}$ we will write $\textbf{X}\succ
\textbf{Y}$. $\textbf{Y}\preceq\textbf{X}$ means that $\textbf{X}\succeq\textbf{Y}$.
\end{definition}
\begin{proposition}\label{wellorder}
Every monomial order on $Mon(A^m)$ is a well order.
\end{proposition}
Given a monomial order $\succeq$ on $Mon(A)$, we can define two natural orders on $Mon(A^m)$.
\begin{definition}
Let $\textbf{X}=X\textbf{e}_i$ and $\textbf{Y}=Y\textbf{e}_j\in Mon(A^m)$.
\begin{enumerate}
\item [\rm(i)]The TOP {\rm(}term over position{\rm)} order is defined by
\begin{center}
$\textbf{X}\succeq\textbf{Y}\Longleftrightarrow
\begin{cases} X\succeq Y & \\
\text{or} & \\
X=Y \text{and}& i>j.
\end{cases}$
\end{center}
\item[\rm(ii)]The TOPREV order is defined by
\begin{center}
$\textbf{X}\succeq\textbf{Y}\Longleftrightarrow
\begin{cases} X\succeq Y & \\
\text{or} & \\
X=Y \text{and}& i<j.
\end{cases}$
\end{center}
\end{enumerate}
\end{definition}
\begin{remark}
(i) Note that with TOP we have
\begin{center}
$\textbf{\emph{e}}_m\succ \textbf{\emph{e}}_{m-1}\succ \cdots \succ\textbf{\emph{e}}_1$
\end{center}
and
\begin{center}
$\textbf{\emph{e}}_1\succ \textbf{\emph{e}}_{2}\succ \cdots \succ \textbf{\emph{e}}_m$
\end{center}
for TOPREV.

(ii) The POT (position over term) and POTREV  orders defined in \cite{Loustaunau} and
\cite{Lezama2} for modules over classical polynomial commutative rings are not degree compatible.

(iii) Other examples of monomial orders in $Mon(A^m)$ are considered in \cite{Gomez-Torrecillas2},
e.g, orders with weight.
\end{remark}
We fix a monomial order on $Mon(A)$, let $\textbf{\emph{f}}\neq \textbf{0}$ be a vector of $A^m$,
then we may write $\textbf{\emph{f}}$ as a sum of terms in the following way
\begin{center}
$\textbf{\emph{f}}=c_1\textbf{\emph{X}}_1+\cdots +c_t\textbf{\emph{X}}_t$,
\end{center}
where $c_1,\dots ,c_t\in R-0$ and $\textbf{\emph{X}}_1\succ\textbf{\emph{X}}_2\succ\cdots
\succ\textbf{\emph{X}}_t$ are monomials of $Mon(A^m)$.
\begin{definition}
With the above notation, we say that
\begin{enumerate}
\item[\rm(i)]$lt(\textbf{f}):=c_1\textbf{X}_1$ is the leading term of $\textbf{f}$.
\item[\rm(ii)]$lc(\textbf{f}):=c_1$ is the leading coefficient of $\textbf{f}$.
\item[\rm(iii)]$lm(\textbf{f}):=\textbf{X}_1$ is the leading monomial of $\textbf{f}$.
\end{enumerate}
\end{definition}
For $\textbf{\emph{f}}=\textbf{0}$ we define
$lm(\textbf{0})=\textbf{0},lc(\textbf{0})=0,lt(\textbf{0})=\textbf{0}$, and if $\succeq$ is a
monomial order on $Mon(A^m)$, then we define $\textbf{X}\succ\textbf{0}$ for any $\textbf{X}\in
Mon(A^m)$. So, we extend $\succeq$ to $Mon(A^m)\bigcup\{\textbf{0}\}$.

\subsubsection{Division algorithm and Gröbner bases for submodules of $A^m$}

The reduction process, Theorem \ref{algdivforPBW} and the Division Algorithm for left ideals can be
easy adapted for submodules of $A^m$.

\begin{definition}
Let $M\neq 0$ be a submodule of $A^m$ and let $G$ be a non empty finite subset of non-zero vectors
of $M$, we say that $G$ is a Gröbner basis for $M$ if each element $0\neq \textbf{f}\in M$ is
reducible w.r.t. $G$.
\end{definition}
We will say that $\{\textbf{0}\}$ is a Gröbner basis for $M=0$.
\begin{theorem}\label{teogrobnersigmapbwformodules}
Let $M\neq 0$ be a submodule of $A^m$ and let $G$ be a finite subset of non-zero vectors of $M$.
Then the following conditions are equivalent:
\begin{enumerate}
\item[\rm(i)]$G$ is a Gröbner basis for $M$.
\item[\rm(ii)]For any vector $\textbf{f}\in A^m$,
\begin{center}
$\textbf{f}\in M$ if and only if $\textbf{f}\xrightarrow{\,\, G\,\, }_{+} \textbf{\emph{0}}$.
\end{center}
\item[\rm(iii)]For any $\textbf{\emph{0}}\neq \textbf{f}\in M$ there exist $\textbf{g}_1,\dots ,\textbf{g}_t\in
G$ such that $lm(\textbf{g}_j)|lm(\textbf{f})$, $1\leq j\leq t$, {\rm(}i.e.,
$ind(lm(\textbf{g}_{j}))=ind(lm(\textbf{f}))$ and there exist $\alpha_j\in \mathbb{N}^n$ such that
$\alpha_j+\exp(lm(\textbf{g}_j))=\exp(lm(\textbf{f}))${\rm)} and
\begin{center}
$lc(\textbf{f})\in \langle \sigma^{\alpha_1}(lc(\textbf{g}_1))c_{\alpha_1,\textbf{g}_1},\dots
,\sigma^{\alpha_t}(lc(\textbf{g}_t))c_{\alpha_t,\textbf{g}_t}\}$.
\end{center}
\item[\rm(iv)]For $\alpha \in \mathbb{N}^n$ and $1\leq u\leq m$, let $\langle \alpha ,M\}_u$ be the
left ideal of $R$ defined by
\begin{center}
$\langle \alpha ,M\}_u:=\langle lc(\textbf{f})|\textbf{f}\in M,ind(lm(\textbf{f}))=u,
\exp(lm(\textbf{f}))=\alpha\}$.
\end{center}
Then, $\langle \alpha ,M\}_u=J_u$, with
\begin{center}
$J_u:=\langle \sigma^{\beta}(lc(\textbf{g}))c_{\beta,\textbf{g}}|\textbf{g}\in G,
ind(lm(\textbf{g}))=u\, \ \text{and} \ \beta + \exp(lm(\textbf{g}))=\alpha\}$.
\end{center}
\end{enumerate}
\end{theorem}
\begin{proof}
See \cite{Jimenez2}.
\end{proof}
From this theorem we get the following consequences.
\begin{corollary}\label{1117}
Let $M\neq 0$ be a submodule of $A^m$. Then,
\begin{enumerate}
\item[\rm(i)]If $G$ is a Gröbner basis for $M$, then $M=\langle G\rangle$.
\item[\rm(ii)]Let $G$ be a Gröbner basis for $M$, if $\textbf{f}\in M$ and
$\textbf{f}\xrightarrow{\,\, G\,\, }_{+} \textbf{h}$, with $\textbf{h}$ reduced, then
$\textbf{h}=\textbf{\emph{0}}$.
\item[\rm(iii)]Let $G=\{\textbf{g}_1,\dots,\textbf{g}_t\}$ be a set of non-zero vectors of $M$ with
$lc(\textbf{g}_i)\in R^{*}$ for each $1\leq i\leq t$. Then, $G$ is a Gröbner basis of $M$ if and
only if given $0\neq \textbf{r}\in M$ there exists $i$ such that $lm(\textbf{g}_i)$ divides
$lm(\textbf{r})$.
\end{enumerate}
\end{corollary}
\begin{proof}
The proof is an easy adaptation of the proof of Corollary \ref{153}.
\end{proof}
Note that the remainder of $\textbf{\emph{f}}\in A^m$ with respect to a Grobner basis is not
unique. Moreover, changing the term order, a Gröbner basis could not be again a Gröbner basis. In
fact, a counterexample was given in \cite{Lezama2} for the trivial case when $A=R[x_1,\dots,x_n]$
is the commutative polynomial ring.

\subsubsection{Buchberger's algorithm for modules}\label{5.5.4}

\noindent Recall that we are assuming that $A$ is a bijective skew $PBW$ extension, we will observe
that every submodule $M$ of $A^m$ has a Gröbner basis, and also we will construct the Buchberger's
algorithm for computing such bases. The results obtained here improve those of \cite{Jimenez2} and
\cite{Jimenez} and generalize the results obtained in Section \ref{5.4} for left ideals.

We start fixing some notation and proving a preliminary general result.
\begin{definition}\label{BFformodules}
Let $F := \{\textbf{g}_1,\dots,\textbf{g}_s\}\subseteq A^m$ such that the least common multiple of
$\{lm(\textbf{g}_1),\dots ,lm(\textbf{g}_s)\}$, denoted by $\textbf{X}_F$, is non-zero. Let
$\theta\in \mathbb{N}^n$, $\beta_i:=\exp(lm(\textbf{g}_i))$ and $\gamma_i\in \mathbb{N}^n$ such
that $\gamma_i+\beta_i=\exp(\textbf{X}_F)$, $1\leq i\leq s$. $B_{F,\theta}$ will denote a finite
set of generators of
\begin{center}
$S_{F,\theta}:=Syz_R[\sigma^{\gamma_1+\theta}(lc(\textbf{g}_1))c_{\gamma_1+\theta,\beta_1} \ \cdots
\ \sigma^{\gamma_s+\theta}(lc(\textbf{g}_s))c_{\gamma_s+\theta,\beta_s})]$.
\end{center}
For $\theta=\textbf{\emph{0}}:=(0,\dots,0)$, $S_{F,\theta}$ will be denoted by $S_F$ and
$B_{F,\theta}$ by $B_F$.
\end{definition}
\begin{lemma}\label{SumForModule}
Let $\textbf{g}_1,\ldots, \textbf{g}_s \in A^m$ , $c_1,\ldots, c_s \in R-\{0\}$ and
$\alpha_1,\ldots,\alpha_{s}\in \mathbb{N}^{n}$ be such that
$lm(x^{\alpha_{1}}lm(\textbf{g}_1))=\cdots
=lm(x^{\alpha_{s}}lm(\textbf{g}_s))=:\textbf{X}_{\delta}$. If
$lm(\sum_{i=1}^{s}c_{i}x^{\alpha_{i}}\textbf{g}_{i}) \prec \textbf{X}_{\delta}$, then there exist
$r_{1},\ldots,r_{k}\in R$ and $l_{1},\ldots,l_{s}\in A$ such that
\[\sum_{i=1}^{s}c_{i}x^{\alpha_{i}}\textbf{g}_{i}=
\sum_{j=1}^{k}r_{j}x^{\delta-\exp(\textbf{X}_{F})}\biggl(\sum_{i=1}^{s}b_{ji}x^{\gamma_{i}}\textbf{g}_{i}\biggr)+
\sum_{i=1}^{s}l_{i}\textbf{g}_{i},\] where $\textbf{X}_{F}$ is the least common multiple of
$lm(\textbf{g}_{1}),\ldots, lm(\textbf{g}_{s})$, $\gamma_{i}\in \mathbb{N}^{n}$ is such that
$\gamma_{i}+\exp(\textbf{g}_{i})=\exp(\textbf{X}_{F})$, $1\leq i\leq s$, and
\begin{center}
$B_{F}:=\{\textbf{b}_1,\dots,\textbf{b}_k\}:=\{(b_{11},\dots,b_{1s}),\dots,
(b_{k1},\dots,b_{ks})\}$.
\end{center}
Moreover, $lm(x^{\delta-\exp(\textbf{X}_{F})}\sum_{i=1}^{s}b_{ji}x^{\gamma_{i}}\textbf{g}_{i})\prec
\textbf{X}_{\delta}$ for every $1\leq j\leq k$, and \linebreak $lm(l_{i}\textbf{g}_{i}) \prec
\textbf{X}_{\delta}$ for every $1\leq i\leq s$.
\begin{proof}
It is easy to adapt the proof of Lemma \ref{Sumofsicigies}.

\end{proof}
\end{lemma}

\begin{theorem}
Let $M\neq 0$ be a submodule of $A^m$ and let $G$ be a finite subset of non-zero generators of $M$.
Then the following conditions are equivalent:
\begin{enumerate}
\item[\rm (i)]$G$ is a Gröbner basis of $M$.
\item[\rm (ii)]For all $F:= \{\textbf{g}_1,\dots,\textbf{g}_s\}\subseteq G$, with $\textbf{X}_F\neq \textbf{\emph{0}}$, and for any
$(b_1,\dots,b_s)\in B_{F}$,
\begin{center}
$\sum_{i=1}^sb_ix^{\gamma_i}\textbf{g}_i\xrightarrow{\,\, G\,\, }_+ 0$.
\end{center}
\end{enumerate}
\end{theorem}
\begin{proof}
See the proof of Theorem \ref{1.5.19}.
\end{proof}
\begin{corollary}\label{algorithmforbijectivemodules}
Let $F=\{\textbf{f}_1,\dots ,\textbf{f}_s\}$ be a set of non-zero vectors of $A^m$. The algorithm
below produces a Gröbner basis for the submodule $\langle \textbf{f}_1,\dots ,\textbf{f}_s\rangle$
{\rm(}$P(X)$ denotes the set of subsets of the set $X${\rm)}:
\begin{center}
\fbox{\parbox[c]{11cm}{
\begin{center}
{\rm \textbf{Buchberger's algorithm for modules\\
over bijective skew $PBW$ extensions}}
\end{center}
\begin{description}
\item[]{\rm \textbf{INPUT}:} $F := \{\textbf{f}_1,\dots,\textbf{f}_s\}\subseteq A^m$,
$\textbf{f}_i\neq \textbf{\emph{0}}$, $1\leq i\leq s$
\item[]{\rm \textbf{OUTPUT}:} $G=\{\textbf{g}_1,\dots ,\textbf{g}_t\}$ a Gröbner basis for $\langle F\rangle$
\item[]{\rm \textbf{INITIALIZATION}:} $G:=\emptyset, G':=F$
\item[]{\rm \textbf{WHILE}} $G'\neq G$ {\rm \textbf{DO}}
\begin{quote}$D:=P(G')-P(G)$

\smallskip

$G:=G'$

\smallskip

{\rm \textbf{FOR}} each $S:=\{\textbf{g}_{i_1},\dots ,\textbf{g}_{i_k}\}\in D$, with
$\textbf{X}_S\neq \textbf{\emph{0}}$, {\rm \textbf{DO}}

\smallskip

\begin{quote}
Compute $B_S$

\smallskip
{\rm \textbf{FOR}} each $\textbf{b}=(b_1,\dots ,b_k)\in B_S$ {\rm \textbf{DO}}
\begin{quote}Reduce $\sum_{j=1}^kb_jx^{\gamma_j}\textbf{g}_{i_j}\xrightarrow{\,\, G'\,\, }_+ \textbf{r}$,
with $\textbf{r}$ reduced with respect to $G'$ and $\gamma_j$ defined as in Definition
\ref{BFformodules}
\begin{quote}{\rm \textbf{IF}} $\textbf{r}\neq \textbf{\emph{0}}$ {\rm \textbf{THEN}}
\begin{quote}$G':=G'\cup \{\textbf{r}\}$
\end{quote}
\end{quote}
\end{quote}
\end{quote}
\end{quote}
\end{description}}}
\end{center}
\end{corollary}
From Theorem \ref{1.3.4} and the previous corollary we get the following direct conclusion.
\begin{corollary}
Every submodule of $A^m$ has a Gröbner basis.
\end{corollary}
\begin{example}
We will illustrate the above algorithm with the bijective skew $PBW$ extension $\mathcal{R}$ of
Example \ref{1.1.21}. For computational reasons, we rewrite the generators and relations for this
algebra in the following way:
\[x:=b,\ \ \ \ y:=a,\ \ \ \ z:=c,\ \ \ \ w:=d,\ \ \ \ \]
and
\begin{align*}
yx&=q^{-1}xy,\ \ \ \ wx=qxw,\ \ \ \ zy=qyz,\ \ \ \ wz=qzw\\
zx&=\mu^{-1}xz,\ \ \ \ wy=yw+(q-q^{-1})xz,
\end{align*}
and, therefore, $\mathcal{R}\cong \sigma(k[x])\langle y,z,w\rangle$. On $Mon(\mathcal{R})$ we
consider the order deglex with $y\succ z\succ w$ and in $Mon(A^2)$ the TOPREV order,  whence
$\textbf{\emph{e}}_1> \textbf{\emph{e}}_2$. Moreover, we will take $K=\mathbb{Q}$,
$\mu=\frac{1}{2}$ and $q=\frac{2}{3}$. From above relations, we obtain that
$\sigma_1(x)=\frac{3}{2}x$, $\sigma_2(x)=2x$ and $\sigma_{3}(x)=\frac{2}{3}x$. Let
$\textbf{\emph{f}}_1=xyw\textbf{\emph{e}}_1+w\textbf{\emph{e}}_2$ and
$\textbf{\emph{f}}_2=zw\textbf{\emph{e}}_1+xy\textbf{\emph{e}}_2$. We will construct a Gr\"obner
basis for $M:=\langle \textbf{\emph{f}}_1,\textbf{\emph{f}}_2\rangle$.

\textit{Step 1.} We start with $G:=\varnothing$, $G':=\{\textbf{\emph{f}}_1,\textbf{\emph{f}}_2\}$. Since $G'\neq G$, we make $D:=P(G')-P(G)$, i.e., $D:=\{S_1,S_2,S_{1,2}\}$, where $S_1:=\{\textbf{\emph{f}}_1\}$, $S_2:=\{\textbf{\emph{f}}_2\}$, $S_{1,2}:=\{\textbf{\emph{f}}_1,\textbf{\emph{f}}_2\}$. We also make $G:=G'$, and for every $S\in D$ such that $\textbf{\emph{X}}_{S}\neq \textbf{\emph{0}}$ we compute $B_S$:\\
$\centerdot$ For $S_1$ we have $Syz_{\mathbb{Q}[x]}[\sigma^{\gamma_1}(lc(\textbf{\emph{f}}_1))c_{\gamma_1,\beta_1}]$, where $\beta_1=\exp(lm(\textbf{\emph{f}}_1))=(1,0,1)$, $\gamma_1=(0,0,0)$ and $c_{\gamma_1,\beta_1}=1$; thus $B_{S_1}=\{0\}$ and we do not add any vector to $G'$.\\
$\centerdot$ For $S_2$ we have an identical situation.\\
$\centerdot$ For $S_{1,2}$ we have $X_{1,2}=lcm\{lm(f_1),lm(f_2)\}=yzw\textbf{\emph{e}}_{1}$, thus
$\gamma_1=(0,1,0)$ and $\gamma_{2}=(1,0,0)$. Since $zyw=\frac{2}{3}yzw$, then
$c_{\gamma_1,\beta_1}=\frac{2}{3}$ and $\sigma^{\gamma_1}(lc(f_1))=\sigma_2(x)=2x$. Analogously,
$c_{\gamma_2,\beta_2}=1$ and $\sigma^{\gamma_2}(lc(f_2))=\sigma_{1}(x^2)=\frac{9}{4}x^2$.  Hence,
we must computing a system of generators for  $Syz_{\mathbb{Q}[x]}[\frac{4}{3}x,\frac{9}{4}x^2]$.
Such generator set can be $B_{S_{1,2}}=\{(\frac{3}{4}x,-\frac{4}{9})\}$. From this we get
\begin{align*}
\frac{3}{4}xz\textbf{\emph{f}}_1-\frac{4}{9}y\textbf{\emph{f}}_{2}=
&\frac{3}{4}xz(xyw\textbf{\emph{e}}_1+w\textbf{\emph{e}}_2)-
\frac{4}{9}y(x^2zw\textbf{\emph{e}}_1+xy\textbf{\emph{e}}_2)\\
=&x^2zyw{\emph{e}}_1+\frac{3}{4}xzw\textbf{\emph{e}}_2-x^2yzw\textbf{\emph{e}}_1-\frac{2}{3}xy^2\textbf{\emph{e}}_2\\
=& -\frac{2}{3}xy^2\textbf{\emph{e}}_2+\frac{3}{4}xzw\textbf{\emph{e}}_2:=\textbf{\emph{f}}_3,
\end{align*}
Observe that $\boldsymbol{f}_3$ is reduced with respect to $G'$. We make $G':=\{\textbf{\emph{f}}_1,\textbf{\emph{f}}_2,\textbf{\emph{f}}_3\}$.\\

\textit{Step 2}: since $G =\{\boldsymbol{f}_1, \boldsymbol{f}_2\} \neq G' = \{\boldsymbol{f}_1,
\boldsymbol{f}_2, \boldsymbol{f}_3\}$, we make $D:= {\mathcal{P}}(G') - {\mathcal{P}}(G)$, i.e.,
$D:= \{S_3, S_{1, 3}, S_{2, 3}, S_{1, 2, 3}\}$, where $S_1:= \{\boldsymbol{f}_1\}, S_{1, 3}:=
\{\boldsymbol{f}_1, \boldsymbol{f}_3\},$ \linebreak $ S_{2, 3}:= \{\boldsymbol{f}_2,
\boldsymbol{f}_3\}, S_{1, 2, 3}:= \{\boldsymbol{f}_1, \boldsymbol{f}_2, \boldsymbol{f}_3\}$. We
make $G := G'$, and for every $S \in D$ such that $\boldsymbol{X}_S \ne \textbf{0}$ we must compute
$B_{S}$. Since $\boldsymbol{X}_{S_{1,3}} = \boldsymbol{X}_{S_{2,3}}= \boldsymbol{X}_{S_{1,2,3}}=
\textbf{0}$, we only need to consider $S_3$.

$\centerdot$ We have to compute
\[Syz_{\mathbb{Q}[x]}[\sigma^{\gamma_3}(lc(\boldsymbol{f}_3))c_{\gamma_3, \beta_3}],\]
where $\beta_3 =$ $\exp(lm(\boldsymbol{f}_3)) = (2,0,0)$; $\boldsymbol{X}_{S_{3}}= lcm
\{lm(\boldsymbol{f}_3)\} = lm(\boldsymbol{f}_3) = y^2\boldsymbol{e}_2$;
$\exp(\boldsymbol{X}_{S_{3}}) = (0,2,0)$; $\gamma_3 =$ $\exp(\boldsymbol{X}_{S_{3}}) - \beta_3$ =
(0,0, 0); $x^{\gamma_3}x^{\beta_3} = y^2$, so $c_{\gamma_3, \beta_3} = 1$. Hence
\begin{align*}
\sigma^{\gamma_3}(lc(\boldsymbol{f}_3))c_{\gamma_3, \beta_3} &= \sigma^{\gamma_3}(-\frac{2}{3}x) 1
= \sigma_2^0 \sigma_3^0(-\frac{2}{3}x) = -\frac{2}{3}x,
\end{align*}
and $Syz_{\mathbb{Q}[x]}[-\frac{2}{3}x] = \{0\}$, i.e., $B_{S_{3}} = \{0\}$. This means that we not
add any vector to $G'$ and hence $G = \{\boldsymbol{f}_1, \boldsymbol{f}_2, \boldsymbol{f}_3\}$ is
a Gröbner basis for $M$.

\end{example}

\begin{remark}
There are some classical and elementary applications of Gröbner theory that we will study in a
forthcoming paper, for example, we can solve the membership problem, we can compute the syzygy
module, the intersection and quotient of ideals and submodules, the matrix presentation of a
finitely presented module, the kernel and the image of homomorphism between modules, the one side
inverse of a matrix, etc. With this, we can make constructive the theory of projective modules,
stably free modules and Hermite rings studied in this work.
\end{remark}


\end{document}